\documentclass[reqno,10pt]{amsart}

\usepackage[utf8]{inputenc}
\usepackage[backend=biber,giveninits=true,sorting=nyt]{biblatex}
\usepackage{amsfonts}
\usepackage{latexsym}
\usepackage{amssymb}
\usepackage{amsthm}
\usepackage{verbatim}
\usepackage{amsmath}
\usepackage{mathrsfs}
\usepackage{esint}
\usepackage{color}
\usepackage{fullpage}
\usepackage{amscd}
\usepackage{enumitem} 
\usepackage{mathtools}
\usepackage{relsize}
\usepackage[colorlinks=true, pdfstartview=FitV, linkcolor=blue, citecolor=blue, urlcolor=blue]{hyperref}
\usepackage{bbm}
\usepackage{thmtools}
\usepackage{cleveref}
\numberwithin{equation}{section}
\usepackage{mleftright}
\usepackage{bibentry}

\newtheorem{theorem}{Theorem}[section]
\newtheorem{lemma}[theorem]{Lemma}
\newtheorem{proposition}[theorem]{Proposition}
\newtheorem{corollary}[theorem]{Corollary}

\newtheorem{definition}[theorem]{Definition}
\newtheorem{remark}[theorem]{Remark}
\theoremstyle{remark}

\newcommand{\RED}{\color{black}}
\newcommand{\BLACK}{\color{black}}
\newcommand{\PINK}{\color{black}}

\newcommand{\calA}{\mathcal{A}}
\newcommand{\calB}{\mathcal{B}}
\newcommand{\C}{\mathbb{C}}
\newcommand{\calC}{\mathcal{C}}
\newcommand{\calD}{\mathcal{D}}

\newcommand{\calH}{\mathcal{H}}
\newcommand{\calI}{\mathcal{I}}

\newcommand{\calK}{\mathcal{K}}
\newcommand{\calL}{\mathcal{L}}
\newcommand{\M}{\mathbb{M}}
\newcommand{\calM}{\mathcal{M}}
\newcommand{\N}{\mathbb{N}}

\newcommand{\calQ}{\mathcal{Q}}
\newcommand{\R}{\mathbb{R}}
\newcommand{\calR}{\mathcal{R}}
\newcommand{\calS}{\mathcal{S}}
\newcommand{\calT}{\mathcal{T}}
\newcommand{\U}{U}
\newcommand{\calU}{\mathcal{U}}
\newcommand{\calV}{\mathcal{V}}
\newcommand{\W}{{\mathbb{W}}}

\newcommand{\calY}{\mathcal{Y}}
\newcommand{\Z}{\mathbb{Z}}
\newcommand{\calZ}{\mathcal{Z}}

\DeclareMathOperator{\dev}{dev}

\DeclareMathOperator{\sym}{sym}
\DeclareMathOperator{\tr}{{\rm tr}}

\DeclareMathOperator{\supp}{supp}
\DeclareMathOperator*{\argmin}{Arg\,min}

\newcommand{\closure}[1]{\overline{#1}}

\newcommand{\strong}{\to}
\newcommand{\weak}{\,\xrightharpoonup{}\,}
\newcommand{\weakstar}{\xrightharpoonup{*}}
\newcommand{\strongtwoscale}{\xrightarrow{\,2\,}}
\newcommand{\weaktwoscale}{\xrightharpoonup{2}}

\newcommand{\weakstartwoscale}{\xrightharpoonup{2-*}}

\renewcommand{\div}{\operatorname{div}}

\newcommand{\fvol}{f}
\newcommand{\fsurf}{g}
\newcommand{\Dir}{D}

\newcommand{\eps}{\varepsilon}

\newcommand{\zerothI}[1]{\overline{#1}}
\newcommand{\firstI}[1]{\widehat{#1}}
\newcommand{\scaling}{{S}}

\newcommand{\Mb}{\mathcal{M}_b}
\newcommand{\mres}[1]{\lfloor{#1}}
\newcommand{\charfun}[1]{\mathbbm{1}_{#1}}
\newcommand{\genprod}{\stackrel{\text{gen.}}{\otimes}}

\newcommand{\calXzero}[1]{\mathcal{X}_0({#1})}
\newcommand{\calYzero}[1]{\Upsilon_0({#1})}
\newcommand{\calZzero}[1]{\mathring{\Upsilon}_0({#1})}

\newcommand{\ext}[1]{\widetilde{#1}}
\newcommand{\tangential}{_\nu^\perp}

\newcommand{\normal}{_\nu}

\newcommand{\E}{{\bf E}}
\newcommand{\D}{\nabla}
\newcommand{\Hkin}{\mathbb{H}_{\rm kin}}
\newcommand{\Hiso}{H_{\rm iso}}
\newcommand{\chikin}{\chi_{\rm kin}}
\newcommand{\chiiso}{\chi_{\rm iso}}
\newcommand{\disspot}{\mathbf{R}}
\newcommand{\gendisspot}{\widehat{\disspot}}
\newcommand{\red}[1]{#1^{\rm red}}

\title{Homogenization of elasto-plastic plate equations with vanishing hardening}

\author[M. Bu\v{z}an\v{c}i\'{c}]{Marin Bu\v{z}an\v{c}i\'{c}}
\address[M. Bu\v{z}an\v{c}i\'{c}]{University of Zagreb Faculty of Electrical Engineering and Computing, Unska 3, 10000 Zagreb, Croatia}
\email{marin.buzancic@fer.unizg.hr}
\author[I. Vel\v{c}i\'{c}]{Igor Vel\v{c}i\'{c}}
\address[I. Vel\v{c}i\'{c}]{University of Zagreb Faculty of Electrical Engineering and Computing, Unska 3, 10000 Zagreb, Croatia}
\email{igor.velcic@fer.unizg.hr}
\author[J. \v{Z}ubrini\'{c}]{Josip \v{Z}ubrini\'{c}}
\address[J. \v{Z}ubrini\'{c}]{University of Zagreb Faculty of Electrical Engineering and Computing, Unska 3, 10000 Zagreb, Croatia}
\email{josip.zubrinic@fer.unizg.hr}


\begin{document}

\begin{abstract}

We study the asymptotic behavior of thin heterogeneous elastoplastic plates in the framework of linearized elastoplasticity, focusing on the regime where the plate thickness vanishes much faster than the characteristic scale of the material's periodic microstructure.
In contrast to earlier analyzes that required restrictive geometric assumptions on admissible yield surfaces, our approach accommodates general relations between phases without imposing any specific ordering.

The analysis proceeds in two main steps.
First, we rigorously derive a heterogeneous plate model with both isotropic and kinematic hardening through a dimension reduction procedure based on evolutionary $\Gamma$--convergence.
This result extends existing plate models for homogeneous materials to the heterogeneous setting and allows for general forms of hardening and dissipation potentials.
In the second step, we perform two-scale homogenization while simultaneously letting the hardening tend to zero.
This process yields an effective elasto-perfectly plastic plate model and, crucially, provides a characterization of the dissipation potential at the interfaces between different phases.
The resulting dissipation functional takes the form of a non-local inf-convolution of the traces of plastic strains on both sides of the interface, reflecting the Kirchhoff--Love structure of admissible displacements.
\end{abstract}

\maketitle

\vspace{-\baselineskip}

\section{Introduction}
In this work, we consider a lower dimensional homogenized thin plate model within the framework of linearized elasto-plasticity.
Our investigation builds on the results in \cite{Buzancic.Davoli.Velcic.2024.1st} and \cite{Buzancic.Davoli.Velcic.2024.2nd}, where the authors analyzed the asymptotic behavior of perfectly plastic composite materials when both the thickness $h$ of the plate and the periodicity $\varepsilon_h$ of the microstructure tend to zero.
The effective models are dependent on the interaction between homogenization and dimension reduction, which can be described by the ratio $\lim_{h \to 0} \frac{h}{\varepsilon_h} = \gamma \in [0, +\infty]$.
In particular, \cite{Buzancic.Davoli.Velcic.2024.2nd} imposed certain geometric (and very restrictive) constraints on admissible yield surfaces in the extreme regimes $\gamma = 0$ and $\gamma = +\infty$ to ensure lower semicontinuity of the dissipation potential -- a constraint not required in the intermediate regimes considered in \cite{Buzancic.Davoli.Velcic.2024.1st}.
In the present contribution, {\PINK our results are related to the limiting case $\gamma = 0$\BLACK}, corresponding to the situations where the thickness of the plate vanishes on a much faster scale than the material's periodicity, without assuming any specific ordering between phases.
A key element in the analysis is the two step process, in which we first derive a heterogeneous plate model with hardening via dimension reduction, and then secondly we perform homogenization and simultaneously let hardening tend to zero.
With this approach at hand we are able to obtain the limit model together with the appropriate definition of the dissipation potential at the interface of different phases (this was not needed in \cite{Buzancic.Davoli.Velcic.2024.2nd} since there at the interface one always puts the "softer" dissipation potential).
More precisely, the effective dissipation functional is defined as the non-local inf-convolution of the traces of plastic strains from both sides of the interface having a specific rank-one structure, which is inherited from the Kirchhoff--Love structure of the displacements.
From the mathematical point of view, to obtain the limit model, besides the tools developed in \cite{Buzancic.Davoli.Velcic.2024.2nd}, one needs to obtain the appropriate compactness results at the interface of the phases (see \Cref{section two-scale limits}), to define appropriately the dissipation potential at the interface, to prove lower semicontinuity result with respect to this definition (see \Cref{semicontinuity of the dissipation functional}) and to prove the appropriate lower bound inequality for plastic dissipation with respect to this definition (this replaces the construction of the recovery sequence, see \Cref{section stress-plastic strain duality on the cell} and \Cref{subs:dis}).

Next we outline the literature on elasto-plasticity important for this work.
In the context of non-vanishing hardening we mention \cite{Mielke.2005, Mielke.Roubicek.2015} where rate-independent systems are introduced and analyzed with existence and uniqueness results.
In \cite{Liero.Mielke.2011,Liero.Roche.2012} the limit homogeneous plate models are obtained under the assumption of non-vanishing kinematic hardening which takes the specific form.
The study of composite elastoplastic materials is a challenging endeavor.
Homogenization of elasto-plasticity quasistatic case with hardening is analyzed in \cite{Mielke.Timofte.2007}, see also \cite{Schweizer.Veneroni.2010, Schweizer.Veneroni.2015, Heida.Schweizer.2016, Heida.Schweizer.2018} for the homogenization of the evolution plasticity with inertia term and non-periodic case respectively.
In all of these situations the micro variable is part of the evolution equations, i.e. it is not possible to obtain the macroscopic evolution by eliminating micro variable from the effective equations.
For completeness, we also mention \cite{Cristowiak.Kreisbeck.2020, Cristowiak.Kreisbeck.2017, Davoli.Ferreira.Kreisbeck.2021, Davoli.Kreisbeck.2022} for an analysis of large-strain stratified composites in crystal plasticity.
In the context of linearized elasto-perfect plasticity (i.e. without hardening), the existence result for homogeneous material is given in \cite{DalMaso.DeSimone.Mora.2006}.
The analysis in the case without hardening is more involved from the mathematical point of view, since one has to leave the framework of Sobolev functions and work in the framework of functions with bounded deformation.
Derivation of the plate model without hardening is done in \cite{Davoli.Mora.2013,Davoli.Mora.2015} for the homogeneous case.
Heterogeneous $3d$ elasto-perfect plasticity was analyzed in \cite{Francfort.Giacomini.2012}.
The important part of the analysis was to define properly the dissipation potential at the interface, which is achieved through proper inf-convolution.
Namely, since the plastic strain is a measure, it can concentrate at the interface of the phases and thus it is important to know how the dissipation potential behaves at the interface (this is not needed in the case with hardening, when plastic strain is $L^2$ function).
Homogenization of elasto-perfect plasticity is analyzed in \cite{Francfort.Giacomini.2014} where the authors used the stress-strain approach to obtain the effective model.
In that way one can bypass the construction of recovery sequence to prove the global stability.
Based on that approach the authors of already mentioned works \cite{Buzancic.Davoli.Velcic.2024.1st, Buzancic.Davoli.Velcic.2024.2nd} derived the limit plate models in different regimes by doing simultaneous homogenization and dimension reduction.

In this work we firstly derive the limit plate model with hardening and in that way we generalize the plate model obtained in \cite{Liero.Roche.2012}.
The generalization is done in the way that we analyze the heterogeneous case with both isotropic and kinematic hardening and also we {\RED do not\BLACK} assume any specific form of dissipation potential and hardening tensor that was assumed in \cite{Liero.Mielke.2011}, see \Cref{remark comparison}.
Then, using the approach started in \cite{Buzancic.Davoli.Velcic.2024.1st} we let simultaneously the hardening and the periodicity of the material to zero to obtain the effective plate model in the case without hardening.
From a mathematical point of view, starting from the $2d$ elasto-plastic equations with hardening has a significant consequence for the effective plate model, as it affects the behavior of the limit plastic strain at the interface and, consequently, the definition of the dissipation functional at the interface, which in this case takes the form of a non-local inf-convolution{\PINK,\BLACK} see \Cref{remark interface potential}.
It is expected that this model can also be obtained from three-dimensional elasto-plasticity by simultaneously letting the hardening, the plate thickness, and the material oscillations tend to zero, at least in certain regimes -- specifically, when the hardening is not too small compared to the thickness, and the plate thickness is much smaller than the scale of the material oscillations.

The structure of the paper is as follows.
\Cref{preliminaries and notation} introduces our notation and recalls some preliminary results on convex functions of measures, disintegration of Radon measures, as well as some auxiliary claims about stress tensors.
In \Cref{limit of quasistatic evolutions - h to 0} we formulate the three-dimensional elasto-plastic plate model with hardening and study the quasistatic evolution of the thin plate.
We justify the plate model via $\Gamma$--convergence.
The section concludes with a characterization of the limiting quasistatic evolution and admissible stress configurations through the formulation in terms of in-plane and out-of-plane displacements and plastic strains.
The main result of \Cref{limit of quasistatic evolutions - h to 0} is \Cref{main result 1} (see also \Cref{improved convergence}).
This derived model of a heterogeous elasto-plastic plate with hardening is then used as a starting point for the analysis in \Cref{limit of quasistatic evolutions - eps to 0}, which is devoted to the simultaneous homogenization and vanishing of hardening effects in materials with periodic microstructure.
We make use of two-scale convergence and unfolding techniques of symmetrized gradients of $BD$ function and Hessians of $BH$ functions.
Furthermore, we discuss semicontinuity of the dissipation functional, duality between stress and plastic strain, as well as a two-scale version of the lower bound of dissipation functional in the absence of any assumption on the ordering of the phases.
Eventually, at the end of the section we prove the main result of the paper \Cref{main result 2} (see also \Cref{cor improved 2}), where we show the convergence of quasistatic evolutions to the limiting quasistatic evolution.
As a by product of our analysis one can easily see the limit homogenized plate model in the case of non-vanishing hardening, the result which is to the best of our knowledge not present in the literature.

\section{Notation and preliminary results} \label{preliminaries and notation}

Points $x \in \R^3$ will be expressed as pairs $(x',x_3)$, with $x' \in \R^2$ and $x_3 \in \R$, whereas we will write $y \in \calY$ to identify points on a flat 2-dimensional torus, $\mathcal{Y}=\R^2/ Z^2$.
We will denote by $I$ the open interval $I := \left(-\frac{1}{2}, \frac{1}{2}\right)$.
The notation $\nabla_{x'}$ and $\nabla_{y}$ will describe the gradients with respect to $x'$ and $y$, respectively.

For $N=2,3$, we use the notation $\M^{N \times N}$ to identify the set of real $N \times N$ matrices.
We will always implicitly assume this set to be endowed with the classical Frobenius scalar product $A : B := \sum_{i,j}A_{ij}\,B_{ij}$ and the associated norm $|A| := \sqrt{A:A}$, for $A,B\in \mathbb{M}^{N\times N}$.
The subspaces of symmetric and deviatoric matrices will be denoted by $\M^{N \times N}_{\sym}$ and $\M^{N \times N}_{\dev}$, respectively, and we denote by $\mathcal{L}(\M^{3 \times 3}_{\sym})$ and $\mathcal{L}(\M^{3 \times 3}_{\dev})$ the set of symmetric endomorphisms on these spaces.
For the trace and deviatoric part of a matrix $A \in \mathbb{M}^{N\times N}$ we will adopt the notation ${\rm tr}{A}$, and
\[
	A_{\dev} = A - \frac{1}{N}{\rm tr}{A}.
\]

Given two vectors $a, b \in \R^N$, we will adopt standard notation for their scalar product and Euclidean norm, namely $a \cdot b$ and $|a|$.
The dyadic (or tensor) product of $a$ and $b$ will be identified as by $a \otimes b$; correspondingly, the {\em symmetrized tensor product} $a \odot b$ will be the symmetric matrix with entries $(a \odot b)_{ij} := \frac{a_i b_j + a_j b_i}{2}$.
We recall that ${\rm tr}{\left(a \odot b\right)} = a \cdot b$, and $|a \odot b|^2 = \frac{1}{2}|a|^2|b|^2 + \frac{1}{2}(a \cdot b)^2$, so that
\begin{equation*}
	\frac{1}{\sqrt{2}}|a||b| \leq |a \odot b| \leq |a||b|.
\end{equation*}

Given a vector $v \in \R^3$, we will use the notation $v^{1,2}$ to denote the two-dimensional vector having its same first two components
\begin{equation*}
	v^{1,2} := \begin{pmatrix} v_1 \\ v_2 \end{pmatrix}.
\end{equation*}
In the same way, for every $A \in \M^{3 \times 3}$, we will use the notation $A^{1,2}$ to identify the minor
\begin{equation*}
	A^{1,2} := \begin{pmatrix} A_{11} & A_{12} \\ A_{21} & A_{22} \end{pmatrix}.
\end{equation*}
For $h>0$, we denote by $\scaling_h : \R^3 \to \R^3$ the mapping
\begin{equation} \label{scaling S_h}
    \scaling_h a = (a_1,a_2, h a_3),
\end{equation}
and we will denote by $S_{\frac{1}{h}}$ its inverse mapping.

We will adopt standard notation for the Lebesgue and Hausdorff measure, as well as for Lebesgue and Sobolev spaces, and for spaces of continuously differentiable functions.
Given a set $U \subset \R^N$, we will denote its closure by $\closure{U}$ and its characteristic function by $\charfun{U}$.
For a Banach space $V$ we denote by $V'$ its dual.

Let $E$ be an Euclidean space.
We will distinguish between the spaces $C_c^k(\U;E)$ ($C^k$ functions with compact support contained in $\U$) and $C_0^k(\U;E)$ ($C^k$ functions "vanishing on $\partial{\U}$").
The notation $C(\calY;E)$ will indicate the space of all continuous functions which are $[0,1]^2$-periodic.
Analogously, we will define $C^k(\calY;E) := C^k(\R^2;E) \cap C(\calY;E)$.
With a slight abuse of notation, $C^k(\calY;E)$ will be identified with the space of all $C^k$ functions on the 2-dimensional torus.
For a measure $\mu$ defined on some measurable space, the space $L^p_{\mu}$ will denote the space of $L^p$ integrable functions with respect to $\mu$.

Throughout the text, the letter $C$ stands for generic positive constants whose value may vary from line to line.

A collection of all preliminary results which will be used throughout the paper can be found in \cite[Section 2]{Buzancic.Davoli.Velcic.2024.1st}.
For an overview on basic notions in measure theory, functions of bounded variation ($BV$), as well as functions of bounded deformation ($BD$) and bounded Hessian ($BH$), we refer the reader to, e.g., \cite{fonseca2007modern}, \cite{ambrosio2000functions}, \cite{braides1998approximation}, to the monograph \cite{Temam.1985}, as well as to \cite{Demengel.1984}.
By $BV_{\rm loc}(U)$ we denote the space of functions that belong to $BV(\widetilde{U})$, for every open, bounded $\widetilde{U}$ compactly contained in $U$.

\subsection{Convex functions of measures} \label{Convex functions of measures}

Let $U$ be an open set of $\R^N$ and $X$ a finite-dimensional vector space, and denote by $\Mb(U;X)$ the set of finite $X$-valued Radon measures on the set $U$.
For every $\mu \in \Mb(U;X)$ let $\frac{d\mu}{d|\mu|}$ be the Radon--Nikodym derivative of $\mu$ with respect to its variation $|\mu|$.
By $\mu\mres{A}$ we denote the restriction of the measure $\mu$ on the set $A$.
Let $H : X \to [0,+\infty)$ be a convex and positively one-homogeneous function such that
\begin{equation} \label{coercivity of H}
	r |\xi| \leq H(\xi) \leq R |\xi| \quad \text{for every}\, \xi \in X,
\end{equation}
where $r$ and $R$ are two constants, with $0 < r \leq R$.

Using the theory of convex functions of measures (see \cite{Goffman.Serrin.1964} and \cite{Demengel.Temam.1984}) it is possible to define a non-negative Radon measure $H(\mu) \in \Mb^+(U)$ as
\[
	H(\mu)(A) := \int_{A} H\left(\frac{d\mu}{d|\mu|}\right) \,d|\mu|,
\]
for every Borel set $A \subset U$, as well as an associated
 functional $\calH: \Mb(U;X) \to [0,+\infty)$ given by
\[
	\calH(\mu) := H(\mu)(\U) = \int_{\U} H\left(\frac{d\mu}{d|\mu|}\right) \,d|\mu|
\]
and being lower semicontinuous on $\Mb(U;X)$ with respect to weak* convergence, cf. \cite[Theorem 2.38]{ambrosio2000functions}).

Let $a,\, b \in [0,T]$ with $a \leq b$.
The \emph{total variation} of a function $\mu : [0,T] \to \Mb(U;X)$ on $[a,b]$ is defined as
\begin{equation*}
	\calV(\mu; a, b) := \sup\left\{ \sum_{i = 1}^{n-1} \left\|\mu(t_{i+1}) - \mu(t_i)\right\|_{\Mb(U;X)} : a = t_1 < t_2 < \ldots < t_n = b,\ n \in \N \right\}.
\end{equation*}
Analogously, the \emph{$\calH$-variation} of a function $\mu : [0,T] \to \Mb(U;X)$ on $[a,b]$ is given by
\begin{equation*}
	\calV_{\calH}(\mu; a, b) := \sup\left\{ \sum_{i = 1}^{n-1} \calH\left(\mu(t_{i+1}) - \mu(t_i)\right) : a = t_1 < t_2 < \ldots < t_n = b,\ n \in \N \right\}.
\end{equation*}
From \eqref{coercivity of H} it follows that
\begin{equation} \label{equivalence of variations}
	r \calV(\mu; a, b) \leq \calV_{\calH}(\mu; a, b) \leq R \calV(\mu; a, b).
\end{equation}

\subsection{Generalized products}

Let $S$ and $T$ be measurable spaces and let $\mu$ be a measure on $S$.
Given a measurable function $f : S \to T$, we denote by $f_{\#}\mu$ the \emph{push-forward} of $\mu$ under the map $f$, defined by
\[
	f_{\#}\mu(B) := \mu\left(f^{-1}(B)\right), \quad \text{ for every measurable set $B \subseteq T$}.
\]
In particular, for any measurable function $g : T \to \closure{\R}$ we have
\[
	\int_{S} g \circ f \,d\mu = \int_{T} g \,d(f_{\#}\mu).
\]
Note that in the previous formula $S = f^{-1}(T)$.

Let $S_1 \subset \R^{N_1}$, $S_2 \subset \R^{N_2}$, for some $N_1,N_2 \in \N$, be open sets, and let $\eta \in \Mb^+(S_1)$.
We say that a function $x_1 \in S_1 \mapsto \mu_{x_1} \in \Mb(S_2; \R^M )$ is $\eta$-measurable if $x_1 \in S_1 \mapsto \mu_{x_1}(B)$ is $\eta$-measurable for every Borel set $B \subseteq S_2$.

Given a $\eta$-measurable function $x_1 \mapsto \mu_{x_1}$ such that $\int_{S_1}|\mu_{x_1}|\,d\eta<+\infty$, then the \emph{generalized product} $\eta \genprod \mu_{x_1}$ satisfies $ \eta \genprod \mu_{x_1} \in \Mb(S_1 \times S_2; \R^M )$ and is such that 
\[
	\langle \eta \genprod \mu_{x_1}, \varphi \rangle := \int_{S_1} \left( \int_{S_2} \varphi(x_1,x_2) \,d\mu_{x_1}(x_2) \right) \,d\eta(x_1),
\]
for every bounded Borel function $\varphi : S_1 \times S_2 \to \R$.

\subsection{Traces of stress tensors} \label{sub:traces}

In this last subsection we collect some properties of classes of maps which will include our elasto-plastic stress tensors.

We suppose here that $\U$ is an open bounded set of class $C^2$ in $\R^N$.
If $\sigma \in L^2(\U;\M^{N \times N}_{\sym})$ and $\div\sigma \in L^2(\U;\R^N)$, then we can define a distribution $[ \sigma \nu ]$ on $\partial{\U}$ by
\begin{equation} \label{traces of the stress}
	[ \sigma \nu ](\psi) := \int_{\U} \psi \cdot \div\sigma \,dx + \int_{\U} \sigma : \nabla \psi \,dx,
\end{equation}
for every $\psi \in H^1(\U;\R^N)$.
It follows that $[ \sigma \nu ] \in H^{-1/2}(\partial{\U};\R^N)$ (see, e.g., \cite[Chapter 1, Theorem 1.2]{temam2001navier}).
If, in addition, $\sigma \in L^{\infty}(\U;\M^{N \times N}_{\sym})$ and $\div\sigma \in L^N(\U;\R^N)$, then \eqref{traces of the stress} holds for $\psi \in W^{1,1}(\U;\R^N)$.
By Gagliardo's extension theorem \cite[Theorem 1.II]{Gagliardo.1957}, in this case we have $[ \sigma \nu ] \in L^{\infty}(\partial{\U};\R^N)$, and 
\begin{equation*}
	[ \sigma_k \nu ] \weakstar [ \sigma \nu ] \quad \text{weakly* in $L^{\infty}(\partial{\U};\R^N)$},
\end{equation*}
whenever $\sigma_k \weakstar \sigma$ weakly* in $L^{\infty}(\U;\M^{N \times N}_{\sym})$ and $\div\sigma_k \weak \div\sigma$ weakly in $L^N(\U;\R^N)$.

For $\sigma \in L^2(\U;\M^{N \times N}_{\sym})$ such that $\div\sigma \in L^2(\U;\R^N)$, we will consider the normal and tangential parts of $[ \sigma \nu ]$, defined by
\begin{equation*}
	[ \sigma \nu ]\normal := ([ \sigma \nu ] \cdot \nu) \nu, \quad
	[ \sigma \nu ]\tangential := [ \sigma \nu ]-([ \sigma \nu ] \cdot \nu) \nu.
\end{equation*}
Since $\nu \in C^1(\partial{\U};\R^N)$, we have that $[ \sigma \nu ]\normal,\, [ \sigma \nu ]\tangential \in H^{-1/2}(\partial{\U};\R^N)$.
If, in addition, $\sigma_{\dev} \in L^{\infty}(\U;\M^{N \times N}_{\dev})$, then it was proved in \cite[Lemma 2.4]{Kohn.Temam.1983} that $[ \sigma \nu ]\tangential \in L^{\infty}(\partial{\U};\R^N)$ and
\begin{equation*}
	\|[ \sigma \nu ]\tangential\|_{L^{\infty}(\partial{\U};\R^N)} \leq \frac{1}{\sqrt2} \|\sigma_{\dev}\|_{L^{\infty}(\U;\M^{N \times N}_{\dev})}.
\end{equation*}

More generally, if $\U$ has Lipschitz boundary and is such that there exists a compact set $S \subset \partial{U}$ with $\calH^{N-1}(S) = 0$ such that $\partial{U} \setminus S$ is a $C^2$-hypersurface, then arguing as in \cite[Section 1.2]{Francfort.Giacomini.2012} we can uniquely determine $[ \sigma \nu ]\tangential$ as an element of $L^{\infty}(\partial{U};\R^N)$ through any approximating sequence $\{\sigma_n\} \subset C^{\infty}(\closure{\U};\M^{N \times N}_{\sym})$ such that 
\begin{align*}
	& \sigma_n \strong \sigma \quad \text{strongly in } L^2(\U;\M^{N \times N}_{\sym}),\\
	& \div\sigma_n \strong \div\sigma \quad \text{strongly in } L^2(\U;\R^N),\\
	& \| (\sigma_n)_{\dev} \|_{L^{\infty}(\U;\M^{N \times N}_{\dev})} \leq \| \sigma_{\dev} \|_{L^{\infty}(\U;\M^{N \times N}_{\dev})},
\end{align*}
as the weak $*$ limit of $[\sigma_n\nu]_{\nu}^\perp=[(\sigma_n)_{\rm dev}\nu]_{\nu}^\perp$.


\section{Dimension reduction of heterogeneous elasto-plastic plate model with hardening} \label{limit of quasistatic evolutions - h to 0}

In this section we derive the limit model for heterogeneous plate with hardening from $3d$ elasto-plasticity with hardening.
In \Cref{section setup} we introduce the basic objects and state the equations of elasto-plasticity.
In \Cref{section rescaled} we rescale the problem on the canonical domain, in \Cref{section quasistatic evolution} we give the energetic formulation of the equations of elasto-plasticity and state the existence and uniqueness result.
In \Cref{section gamma converegence results} we introduce the necessary objects to define limit quasistatic evolution and prove the auxiliary results for obtaining $\Gamma$--convergence.
In \Cref{seclimquasiev}, we prove the convergence of the three-dimensional quasistatic evolution to the limiting (reduced) quasistatic evolution. Finally, in \Cref{Admissible stress configurations}, we formulate the equations in terms of the limiting stress, which will be used in the next section to derive the limit model without hardening, since the approach there relies on the stress-strain formulation
The main results of this section are \Cref{main result 1} and \Cref{improved convergence}, which give the convergence result.

\subsection{Setup}\label{section setup} 

In this section we introduce elasticity and hardening tensor as well as dissipation potential and state the equations of elasto-plasticity, We also introduce the set of admissible unknowns of the equations.
We will not rely in our analysis on the formulation of the equations given here, but will use an equivalent energetic formulation given in \Cref{section quasistatic evolution}.

Let $\omega\subset\R^2$ be a planar, bounded Lipschitz domain and $I = \left(-\frac{1}{2}, \frac{1}{2}\right)\subset \R$, an open interval.
Given a small positive number $h>0$, we define the three-dimensional thin plate
\[
\Omega^h := \omega \times (h I),
\]
with the boundary partitioned into the lateral surface $\Gamma^h_0 := \partial\omega \times (h I)$ and the transverse boundary $\Gamma^h := \omega \times \partial(h I)$.
We will denote the Dirichlet boundary of $\omega$ with $\gamma_D \subset \partial \omega$, as well as $\Gamma_D^h = \gamma_D \times (hI)$.
The assumption is that $\gamma_D$ is of positive measure.

\medskip
\noindent{\bf The elasticity tensor.} \nopagebreak

Let $\C \in L^\infty(\omega; \mathcal{L}(\M^{3 \times 3}_{\sym}))$ be the {\em elasticity tensor}, considered as a map from $\omega$ taking values in the set of symmetric positive definite linear operators, $\C(x'): \M^{3 \times 3}_{\sym} \to \M^{3 \times 3}_{\sym}$, defined as
\begin{equation*}
	\C(x') \xi := \C_{\dev}(x')\,\xi_{\dev} + \left(k(x')\,{\rm tr}{\xi}\right)\,I_{3 \times 3} \quad \text{ for a.e. } x' \in \omega \text{ and } \xi \in \M^{3 \times 3}_{\sym},
\end{equation*}
where $\C_{\dev} \in L^\infty(\omega; \mathcal{L}(\M^{3 \times 3}_{\dev}))$, $k \in L^\infty(\omega)$.
Additionally we assume that there exist two constants $r_{\C}$ and $R_{\C}$, with $0 < r_{\C} \leq R_{\C}$, such that for a.e. $x' \in \omega$ we have:
\begin{align*}
	& r_{\C} |\xi|^2 \leq \C_{\dev}(x') \xi : \xi \leq R_{\C} |\xi|^2 \quad \text{ for every }\xi \in \M^{3 \times 3}_{\dev},\\
	& r_{\C} \leq k(x') \leq R_{\C}.
\end{align*}

\medskip
\noindent{\bf The hardening tensors.} \nopagebreak

Let $\Hkin \in L^\infty(\omega; \mathcal{L}(\M^{3 \times 3}_{\dev}))$ be the {\em kinematic hardening tensor} and $\Hiso \in L^\infty(\omega)$ be the {\em isotropic hardening tensor}.
We assume that there exist constants $r_{\Hkin}$, $r_{\Hiso}$, $R_{\Hkin}$ and $R_{\Hiso}$, with $0 \leq r_{\Hkin} \leq R_{\Hkin}$, $0 \leq r_{\Hiso} \leq R_{\Hiso}$ and $r_{\Hkin} + r_{\Hiso} > 0$, such that for a.e. $x' \in \omega$ we have:
\begin{align*}
	& r_{\Hkin} |\xi|^2 \leq \Hkin(x') \xi : \xi \leq R_{\Hkin} |\xi|^2 \quad \text{ for every }\xi \in \M^{3 \times 3}_{\dev},\\
	& r_{\Hiso} \leq \Hiso(x') \leq R_{\Hiso};
\end{align*}
\begin{remark}
The assumptions above allow for the cases where $\Hkin \equiv 0$ and $\Hiso > 0$ (which corresponds to pure isotropic hardening) or $\Hkin > 0$ and $\Hiso \equiv 0$ (which corresponds to pure kinematic hardening).
In the latter case the isotropic hardening internal state variable $\alpha$ (see below) can be removed from the set of unknowns, but it will not affect the analysis presented in the paper.
\end{remark}

\medskip
\noindent{\bf The dissipation potential.} \nopagebreak

We assume that for each point $x' \in \omega$ there exist convex compact sets $K(x') \in \M^{3 \times 3}_{\dev}$.
Additionally, we assume that there exist two constants $r_K$ and $R_K$, with $0 < r_K \leq R_K$, such that for every $x'$
\begin{equation*}
	\{ \xi \in \M^{3 \times 3}_{\dev} : |\xi| \leq r_K \} \subseteq K(x') \subseteq\{ \xi \in \M^{3 \times 3}_{\dev}: |\xi| \leq R_K \}.
\end{equation*}
For $x' \in \omega$, we define the dissipation potential $\disspot : \omega \times \M^{3 \times 3}_{\dev} \to [0,+\infty]$ to be 
\begin{equation*}
	\disspot(x', p) := \sup\limits_{\tau \in K(x')} \tau : p.
\end{equation*}
It follows that, for $x' \in \omega$, $\disspot(x',\cdot)$ is convex, positively 1-homogeneous, and satisfies
\begin{equation} \label{coercivity of R}
	r_K |p| \leq \disspot(x', p) \leq R_K |p| \quad \text{for every}\, p \in \M^{3 \times 3}_{\dev}.
\end{equation}
Also we have for $x' \in \omega $ 
\begin{equation} \label{distriangle} 
	\disspot(x', p_1+p_2)\leq \disspot(x', p_1)+ \disspot(x', p_2),\quad \forall p_1,p_2 \in \M^{3 \times 3}_{\dev}.
\end{equation} 
We also {\PINK assume\BLACK} that the function $x'\mapsto \disspot (x',p)$ is Borel measurable for every $p \in \M^{3 \times 3}_{\dev}$ which makes $\disspot(\cdot,\cdot)$ Carath\'{e}odory function.

Finally, we define the generalized dissipation potential $\gendisspot : \omega \times \M^{3 \times 3}_{\dev} \times \R \to [0,+\infty]$ as
\begin{equation} \label{dissipation potential with hardening}
	\gendisspot(x', p, \alpha)
	=
	\begin{cases}
	\disspot(x', p) & \text{if $\disspot(x', p) \leq \alpha$}, \\
	+ \infty & \text{else}.
	\end{cases}
\end{equation}
This potential encompasses the cases of both von Mises and Tresca flow rule (see, e.g., \cite[Section 4]{han2012plasticity}).

\begin{remark} \label{remark on Reshetnyak theorem conditions - gendisspot}
We point out that $\gendisspot$ is a Borel function.
Furthermore, for each $x' \in \omega$, the function $(p, \alpha) \mapsto \gendisspot(x', p, \alpha)$ is lower semicontinuous, convex and positively 1-homogeneous.
The {\PINK latter condition means\BLACK} $\gendisspot(x', \lambda p,\lambda \alpha) = \lambda \gendisspot(x', p, \alpha)$ for all $\lambda \geq 0$ and $(p, \alpha) \in \M^{3 \times 3}_{\dev} \times \R$.
For a fixed $x' \in \omega$, the corresponding generalized elastic domain $\widehat{K}(x') \subset \M^{3 \times 3}_{\dev} \times \R$ is defined via 
\begin{equation*}
	\widehat{K}(x') := \partial \gendisspot(x', 0),
\end{equation*}
which is the subdifferential of $\gendisspot(x', \cdot)$ at $0$.
\end{remark}

\medskip
\noindent{\bf Admissible tuples and energy.} \nopagebreak

On $\Gamma_\Dir^h$ we prescribe a Dirichlet boundary datum being the trace of a map $w^h(t,\cdot) \in H^1(\Omega^h;\R^3)$.

The elasto-plastic properties of $\Omega^h$ are described in terms of the linearized strain tensor ${\E}u = \frac{1}{2}\left(\nabla u+\nabla u^T\right) \in \M^{3 \times 3}_{\sym}$, the plastic strain tensor $p \in \M^{3 \times 3}_{\dev}$,
and the internal isotropic hardening variable $\alpha \in \R$, via the stored energy density $\W : \omega \times \M^{3 \times 3}_{\sym} \times \M^{3 \times 3}_{\dev} \times \R \to \R$, which is assumed to be given in the following form
\begin{equation} \label{stored energy density}
	\W(x', {\E}, p, \alpha) := 
	\frac{1}{2}\C(x')({\E}-p):({\E}-p) 
	+ \frac{1}{2}\Hkin(x')p:p + \frac{1}{2}\Hiso(x')\alpha\cdot\alpha.
\end{equation}
Given time-dependent volume and surface loadings $\fvol^h(t,\cdot)\in L^2(\Omega^h;\R^3)$ and $\fsurf^h(t,\cdot)\in L^2(\partial\Omega^h\setminus \Gamma_\Dir^h;\R^3)$, the full elasto-plastic evolution problem can be written in the form
\begin{equation} \label{system of elastoplastic equations on the thin plate}
\begin{split}
	-\div(\partial_{{\E}}\W(\cdot, {\E}u, p, \alpha)) = \fvol^h(t,\cdot) \quad & \text{in}\,\,\Omega^h,\\
	0 \in \partial_{(p, \alpha)}\W(\cdot, {\E}u, p, \alpha) + \partial \gendisspot(\cdot, \dot{p}, \dot{\alpha}) \quad & \text{in}\,\,\Omega^h,\\
	u = w^h(t,\cdot) \quad & \text{on}\,\, \Gamma_\Dir^h,\\
	\partial_{{\E}}\W(\cdot, {\E}u, p, \alpha) \nu = \fsurf^h(t,\cdot) \quad & \text{on}\,\,\partial\Omega^h\setminus \Gamma_\Dir^h,
\end{split}
\end{equation}
where $\nu$ denotes the outer normal vector on $\partial\Omega^h$ and $\dot{p}, \dot{\alpha}$ denote the time derivative of $p$ and $\alpha$ respectively.
We call $\sigma = \partial_{{\E}}\W \in \M^{3 \times 3}_{\sym}$ the stress.

The {\em set of admissible displacements, strains and internal variables} for the boundary datum $w^h$ is denoted by 
$\calA^{\rm hard}(\Omega^h, w^h)$ and is defined as the class of all tuples 
$(u, e, p, \alpha) \in H^1(\Omega^h) \times L^2(\Omega^h;\M^{3 \times 3}_{\sym}) \times L^2(\Omega^h;\M^{3 \times 3}_{\dev}) \times L^2(\Omega^h)$ satisfying
\begin{eqnarray*}
	& {\E}u = e + p \quad \text{ in } \Omega^h,
	\\& u = w^h \quad {\PINK\text{ on } \Gamma_\Dir^h\BLACK}
	\\& \disspot(x', p(x)) \leq \alpha(x) \quad \text{a.e.} \text{ in } \Omega^h.
\end{eqnarray*}
The function $u$ represents the {\em displacement} of the plate, $e$ and $p$ are called the {\em elastic} and {\em plastic strain}, respectively, while $\alpha$ is referred to as the \emph{isotropic hardening internal state variable}.

We will use the formulation given by \eqref{system of elastoplastic equations on the thin plate} in the form of the energetic formulation from \cite{Mielke.Roubicek.2015}, which {\RED does not\BLACK} contain any derivative in time.
Prior to that we will define the problem on the canonical domain as it is standard in dimension reduction problems.

\subsection{The rescaled problem} \label{section rescaled} 

As usual in dimension reduction problems, it is convenient to perform a change of variables in such a way to rewrite the system on a fixed domain independent of $h$.
To this purpose, we set
\begin{equation*}
	\Omega \,:=\, \omega \times I, \qquad 
	\Gamma_\Dir \,:=\, \gamma_\Dir \times I.
\end{equation*}
We consider the change of variables $\psi_h : \closure{\Omega} \to \closure{\Omega^h}$, defined as
\begin{equation} \label{eq:def-psih}
	\psi_h(x',x_3) := (x', hx_3) \quad \text{for every}\, (x',x_3) \in \closure{\Omega},
\end{equation}
and the linear operator $\Lambda_h : \M_{\sym}^{3 \times 3} \to \M_{\sym}^{3 \times 3}$ given by 
\begin{equation} \label{definition Lambda_h}
	\Lambda_h \xi:=\begin{pmatrix}
	\xi_{11} & \xi_{12} & \frac{1}{h}\xi_{13}
	\vspace{0.1 cm}\\
	\xi_{21} & \xi_{22} & \frac{1}{h}\xi_{23}
	\vspace{0.1 cm}\\
	\frac{1}{h}\xi_{31} & \frac{1}{h}\xi_{32} & \frac{1}{h^2}\xi_{33}
	\end{pmatrix}
	\quad\text{for every }\xi \in \M^{3 \times 3}_{\sym}.
\end{equation}
To any tuple $(\upsilon, \eta, q, \beta) \in \calA^{\rm hard}(\Omega^h, w^h)$, we associate a rescaled tuple $(u^h, e^h, p^h, \alpha^h) \in H^1(\Omega;\R^3) \times L^2(\Omega;\M^{3 \times 3}_{\sym}) \times L^2(\Omega;\M^{3 \times 3}_{\sym}) \times L^2(\Omega)$ defined as follows:
\begin{equation*}
	u^h := \scaling_h v \circ \psi_h, \qquad
	e^h := \Lambda_{h}^{-1}\eta \circ \psi_h, \qquad
	p^h := \Lambda_{h}^{-1}q \circ \psi_h, \qquad 
	\alpha^h := \beta \circ \psi_h,
\end{equation*}
{\PINK where $\scaling_h$ has been introduced in \eqref{scaling S_h}.\BLACK}
By the definition of the class $\calA^{\rm hard}(\Omega^h, w^h)$ it follows that the scaled tuple $(u^h, e^h, p^h, \alpha^h)$ satisfies the following
\begin{eqnarray}
	& {\E}u^h = e^h + p^h \quad \text{ in } \Omega, \label{straindec*}\\
	& u^h = w \quad {\PINK\text{ on } \Gamma_\Dir\BLACK}, \label{boundcondp*}\\
	& \disspot(x', \Lambda_h p^h(x)) \leq \alpha^h(x) \quad \text{a.e.} \text{ in } \Omega,\\
	& p^h_{11} + p^h_{22} + \frac{1}{h^2}p^h_{33} = 0 \quad \text{ in } \Omega. \label{trace*}
\end{eqnarray}
Here $w = \scaling_h w^h \circ \psi_h$.
{\RED 
We are thus led to introduce the class of all tuples $(u^h, e^h, p^h, \alpha^h)$ satisfying \eqref{straindec*}--\eqref{trace*}:
\begin{align*}
	\calA_{h}^{\rm hard}(w) := \Big\{ 
	& (u^h, e^h, p^h, \alpha^h) \in H^1(\Omega;\R^3) \times L^2(\Omega;\M^{3 \times 3}_{\sym}) \times L^2(\Omega;\M^{3 \times 3}_{\sym}) \times L^2(\Omega) :
	\\& {\E}u^h = e^h + p^h \quad \text{ in } \Omega,
	\\& u^h = w \quad {\PINK\text{ on } \Gamma_\Dir\BLACK},
	\\& \disspot(x', \Lambda_h p^h(x)) \leq \alpha^h(x) \quad \text{a.e.} \text{ in } \Omega, 
	\\& p^h_{11} + p^h_{22} + \frac{1}{h^2}p^h_{33} = 0 \quad \text{ in } \Omega 
	\Big\},
\end{align*}
and to reformulate the system \eqref{system of elastoplastic equations on the thin plate} in an abstract form for the primary variables $(u^h, e^h, p^h, \alpha^h)$.
\BLACK}

\subsection{Quasistatic evolution for thin plate}\label{section quasistatic evolution} 

In this section we define the quasistatic evolution using the energetic formulation and state the existence and uniqueness result.

To express the energy balance it is useful to introduce the quadratic forms
\begin{align}
	\calQ^{\rm el}_{h}(e^h) & := \frac{1}{2} \int_{\Omega} \C(x') \Lambda_h e^h(x) : \Lambda_h e^h(x) \,dx, \label{definition Q^el_h}\\
	\calQ^{\rm hard}_{h}(p^h, \alpha^h) & := \frac{1}{2} \int_{\Omega} \Hkin(x') \Lambda_h p^h(x) : \Lambda_h p^h(x) \,dx + \frac{1}{2} \int_{\Omega} \Hiso(x') \alpha^h(x) \cdot \alpha^h(x) \,dx, \label{definition Q^hard_h}
\end{align}
the total load
\begin{equation} \label{definition ell}
	\big\langle \ell(t), u^h \big\rangle := \int_{\Omega} \fvol(t,x) \cdot u^h(x) \,dx + \int_{\partial{\Omega} \setminus \Gamma_\Dir} \fsurf(t,x) \cdot u^h(x) \,d\calH^{2},
\end{equation}
where $\fvol(t,\cdot):=\scaling_{1/h}\fvol^h(t, \cdot) \circ \psi_h$ and $\fsurf(t,\cdot):=\scaling_{1/h}\fsurf^h(t, \cdot) \circ \psi_h$ on $(\partial \omega \setminus \gamma_D) \times I$, $\fsurf(t,\cdot):=\frac{1}{h}\scaling_{1/h}\fsurf^h(t, \cdot) \circ \psi_h$ on $\omega \times \partial I$.
We also introduce the dissipation functional
\begin{equation} \label{definition R_h}
	\calR_{h}(p^h, \alpha^h) := \int_{\Omega} \gendisspot(x', \Lambda_h p^h(x), \alpha^h(x)) \,dx,
\end{equation}
defined for $u^h \in H^1(\Omega;\R^3)$, $e^h, p^h \in L^2(\Omega;\M^{3 \times 3}_{\sym})$ and $\alpha^h \in L^2(\Omega)$.

We note that, considering definition \eqref{dissipation potential with hardening}, the condition $\disspot(x', \Lambda_h p^h(x)) \leq \alpha^h(x)$ a.e. in $\Omega$ implies that, in fact, $\calR_{h}(p^h, \alpha^h) = \int_{\Omega} \disspot(x', \Lambda_h p^h(x)) \,dx$.

The $\calR_{h}$-variation of a map $(p^h, \alpha^h) : [0,T] \to L^2(\Omega;\M^{3 \times 3}_{\sym}) \times L^2(\Omega)$ on $[a,b]\subset[0,T]$ is defined as
\begin{equation*}
	\calV_{\calR_{h}}(p^h, \alpha^h; a, b) := \sup\left\{ \sum_{i = 1}^{n} \calR_{h}\!\left(p^h(t_{i+1}) - p^h(t_i), \alpha^h(t_{i+1}) - \alpha^h(t_i)\right) : a = t_1 < t_2 < \ldots < t_n = b,\ n \in \N \right\}.
\end{equation*}
\begin{remark} \label{remupozo1} 
Note that if $ \calV_{\calR_{h}}(p^h, \alpha^h; a, b)<+\infty$, then $\alpha^h(t_1)\leq \alpha^h(t_2)$ for a.e. $x \in \Omega$ and $a\leq t_1 \leq t_2\leq b$.
\end{remark}

For every $t \in [0, T]$ we prescribe a boundary datum $w(t) \in H^1(\Omega;\R^3)$ and assume that the map $t\mapsto w(t)$ is absolutely continuous from $[0, T]$ into $H^1(\Omega;\R^3)$.

\begin{definition} \label{h-quasistatic evolution}
Let $h > 0$.
An \emph{$h$-quasistatic evolution} for the boundary datum $w(t)$ and loads $\ell(t) \in (H^1(\Omega;\R^3))'$ is a function $t \mapsto (u^h(t), e^h(t), p^h(t), \alpha^h(t))$ from $[0,T]$ into $H^1(\Omega;\R^3) \times L^2(\Omega;\M^{3 \times 3}_{\sym}) \times L^2(\Omega;\M^{3 \times 3}_{\sym}) \times L^2(\Omega)$ which satisfies the following conditions:
\begin{enumerate}[label=(qs\arabic*)$_{h}$]
	\item \label{h-qs S} for every $t \in [0,T]$ we have $(u^h(t), e^h(t), p^h(t), \alpha^h(t)) \in \calA_{h}^{\rm hard}(w(t))$ and
	\begin{align*}
		& \calQ^{\rm el}_{h}(e^h(t)) + \calQ^{\rm hard}_{h}(p^h(t), \alpha^h(t)) - \big\langle \ell(t), u^h(t) \big\rangle
		\\& \leq \calQ^{\rm el}_{h}(\eta) + \calQ^{\rm hard}_{h}(\pi, \beta) + \calR_{h}(\pi - p^h(t), \beta - \alpha^h(t)) - \big\langle \ell(t), \upsilon \big\rangle,
	\end{align*}
	for every $(\upsilon, \eta, \pi, \beta) \in \calA_{h}^{\rm hard}(w(t))$.
	\item \label{h-qs E} 
	for every $t \in [0, T]$
	\begin{align*}
		& \calQ^{\rm el}_{h}(e^h(t)) + \calQ^{\rm hard}_{h}(p^h(t), \alpha^h(t)) 
		+ \calV_{\calR_{h}}(p^h, \alpha^h; 0, t) - \big\langle \ell(t), u^h(t) \big\rangle
		\\& = \calQ^{\rm el}_{h}(e^h(0)) + \calQ^{\rm hard}_{h}(p^h(0), \alpha^h(0)) - \big\langle \ell(0), u^h(0) \big\rangle 
		\\& + \int_0^t \left( \int_{\Omega} \C(x') \Lambda_h e^h(s) : {\E}\dot{w}(s) \,dx 
		- \big\langle \ell(s), \dot{w}(s) \big\rangle - \big\langle \dot{\ell}(s), u^h(s) \big\rangle \right) \,ds.
	\end{align*}
\end{enumerate}
\end{definition}

The following existence and uniqueness result of a quasistatic evolution for a general multi-phase material with hardening follows from classical results for the existence of energetic rate-independent systems, given stored energies with quadratic coercivity (see, e.g., \cite[Theorem 3.5.2]{Mielke.Roubicek.2015}).

\begin{theorem} \label{existence result of a quasistatic evolution}
Let $h > 0$ and $w \in W^{1,1} (0,T; H^1(\Omega;\R^3))$, $\ell \in W^{1,1}(0,T;H^1(\Omega;\R^3)')$.
Let $(u^h_0, e^h_0, p^h_0, \alpha^h_0) \in \calA_{h}^{\rm hard}(w(0))$ satisfy the global stability condition \ref{h-qs S}.
Then, there exists a unique $h$-quasistatic evolution $t \mapsto (u^h(t), e^h(t), p^h(t), \alpha^h(t))$ for the boundary datum $w(t)$ such that $u^h(0) = u_0$,\, $e^h(0) = e^h_0$,\, $p^h(0) = p^h_0$, and $\alpha^h(0) = \alpha^h_0$ and such that the functions $t \mapsto u^h(t)$, $t \mapsto e^h(t)$, $t \mapsto p^h(t)$, $t \mapsto \alpha^h(t)$ {\PINK belong to the space $W^{1,1}$ (and thus have an absolutely continuous representative)\BLACK} from $[0, T]$ into $H^1(\Omega;\R^3)$, $L^2(\Omega;\M^{3 \times 3}_{\sym})$, $L^2(\Omega;\M^{3 \times 3}_{{\RED\sym\BLACK}})$, and $L^2(\Omega)$, respectively.
\end{theorem}
\begin{remark} \label{remark stable states} 
{\PINK
We note that, for every $h > 0$, the set of admissible initial data in \Cref{existence result of a quasistatic evolution} that satisfy the global stability condition \ref{h-qs S} is non-empty.

Indeed, let $(u^h_{\rm min}, e^h_{\rm min}, p^h_{\rm min}, \alpha^h_{\rm min}) \in \calA_{h}^{\rm hard}(\tilde{w})$ be the solution of the problem 
\[
	\min_{(\upsilon, \eta, \pi, \beta) \in \calA_{h}^{\rm hard}(\tilde{w})}
	\left\{ 
	\calQ^{\rm el}_{h}(\eta) + \calQ^{\rm hard}_{h}(\pi, \beta) + \calR_{h}(\pi - \tilde{p}, \beta - \tilde{\alpha}) - \big\langle \tilde{\ell}, \upsilon  \big\rangle 
	\right\}.
\]
for some $\tilde{w} \in H^1(\Omega;\R^3)$, $\tilde{\ell} \in (H^1(\Omega;\R^3))'$, and $\tilde{p} \in L^2(\Omega;\M^{3 \times 3}_{\dev})$, $\tilde{\alpha} \in L^2(\Omega)$ for which there exists a pair $(\tilde{u}, \tilde{e})$ such that $(\tilde{u}, \tilde{e}, \tilde{p}, \tilde{\alpha}) \in \calA_{h}^{\rm hard}(\tilde{w})$ (e.g., we can take $(\tilde{u}, \tilde{e}, \tilde{p}, \tilde{\alpha}) = (\tilde{w}, {\E}\tilde{w}, 0,0)$).  
The existence and uniqueness of such a minimizer is a consequence of the coercivity of $\calQ^{\rm el}_{h}$ and $\calQ^{\rm hard}_{h}$, the convexity of $\calR_{h}$ over the convex set $\calA_{h}^{\rm hard}(\tilde{w})$, as well as fact that the minimizing functional attains a finite value in $(\tilde{u}, \tilde{e}, \tilde{p}, \tilde{\alpha})$.

Then, by using the triangle inequality \eqref{distriangle}, we deduce that the tuple $(u^h_{\rm min}, e^h_{\rm min}, p^h_{\rm min}, \alpha^h_{\rm min})$ is also a solution of the problem
\[
	\min_{(\upsilon, \eta, \pi, \beta) \in \calA_{h}^{\rm hard}(\tilde{w})}
	\left\{ 
	\calQ^{\rm el}_{h}(\eta) + \calQ^{\rm hard}_{h}(\pi, \beta) + \calR_{h}(\pi - p^h_{\rm min}, \beta - \alpha^h_{\rm min}) - \big\langle \tilde{\ell}, \upsilon  \big\rangle 
	\right\}.
\]
This guarantees that the solution of the second problem exists.
Taking, specifically, $\tilde{w} = w(0)$ and $\tilde{\ell} = \ell(0)$ gives a solution which by definition precisely satisfies the global stability condition \ref{h-qs S}.
\BLACK}
\end{remark}

\begin{remark} 
It is shown in \cite[Section 3.5]{Mielke.Roubicek.2015} that under the assumption of \Cref{existence result of a quasistatic evolution} the energetic formulation given in \Cref{h-quasistatic evolution} (when moved back to the physical domain $\Omega^h$) is equivalent to the problem \eqref{system of elastoplastic equations on the thin plate} for every $t \in [0,T]$ (the second equation is satisfied for a.e. $t \in [0,T]$).
As already mentioned{\RED,\BLACK} one of the advantages of the energetic formulation is that it {\RED does not\BLACK} contain any time derivative of the unknowns. 
\end{remark}

\begin{remark} \label{improved convergence 1}
Although \cite[Theorem 3.5.2]{Mielke.Roubicek.2015} {\RED says\BLACK} that there exists a unique absolutely continuous solution of quasistatic evolution given by \Cref{h-quasistatic evolution}, it can be shown, by using the proof of \cite[Theorem 5.2]{DalMaso.DeSimone.Mora.2006} and \Cref{stability condition equivalence}, that every solution of the problem given in \Cref{h-quasistatic evolution}, under the assumption of \Cref{existence result of a quasistatic evolution}, is absolutely continuous. Namely, the proof of \cite[Theorem 5.2]{DalMaso.DeSimone.Mora.2006} relies on energy equality which is the analogue of \ref{h-qs E} and the analogue of \Cref{stability condition equivalence}. 
\end{remark}

\begin{corollary} \label{nakcor}
There exists $C>0$, independent of $h$, such that the solution of h-quasistatic evolution given by \Cref{existence result of a quasistatic evolution} satisfies for a.e. $t \in [0,T]$
	\begin{align} \nonumber
	\|\dot{u}^h(t)\|_{H^1(\Omega;\R^3)} +\|\Lambda_h \dot{e}^h(t)\|_{L^2(\Omega;\M^{3 \times 3}_{\sym})} +\|\Lambda_h \dot{p}^h\|_{L^2(\Omega;\M^{3 \times 3}_{\dev})}+\|\dot{\alpha}^h(t)\|_{L^2(\Omega)} \\ \label{ocjena1} \leq C(\|\dot{w}(t)\|_{H^1(\Omega;\R^3)}+\|\dot{f}(t) \|_{L^2(\Omega;\R^3)}+\|\dot{g}(t)\|_{L^2(\partial \Omega \setminus \Gamma_D;\R^3)}).
	\end{align}
\end{corollary}
\begin{proof} 
See the proof of \cite[Theorem 3.5.2]{Mielke.Roubicek.2015}).
\end{proof}

We prove the following lemma, which gives a condition that is equivalent to global stability \ref{h-qs S}.

\begin{lemma} \label{stability condition equivalence}
Let $w \in H^1(\Omega;\R)$.
The tuple $(u^h, e^h, p^h, \alpha^h) \in \calA_{h}^{\rm hard}(w)$ is a solution of the minimization problem
\begin{equation} \label{minimum problem h}
	\min_{(\upsilon, \eta, \pi, \beta) \in \calA_{h}^{\rm hard}(w)}
	\left\{ 
	\calQ^{\rm el}_{h}(\eta) + \calQ^{\rm hard}_{h}(\pi, \beta) + \calR_{h}(\pi - p^h, \beta - \alpha^h) 
	- \big\langle \ell(t), \upsilon - u^h \big\rangle 
	\right\}
\end{equation}
if and only if
\begin{equation} \label{minimality condition h}
	- \calR_{h}(\pi, \beta) \leq 
	\int_{\Omega} \left( \C\!\left(x'\right) \Lambda_h e^h : \Lambda_h \eta + \Hkin\!\left(x'\right) \Lambda_h p^h : \Lambda_h \pi + \Hiso\!\left(x'\right) \alpha^h \cdot \beta \right) \,dx - \big\langle \ell(t), \upsilon \big\rangle
\end{equation}
for every $(\upsilon, \eta, \pi, \beta) \in \calA_{h}^{\rm hard}(0)$.
\end{lemma}

\begin{proof}
Let $(u^h, e^h, p^h, \alpha^h) \in \calA_{h}^{\rm hard}(w)$ be a solution to \eqref{minimum problem h} and fix $(\upsilon, \eta, \pi, \beta) \in \calA_{h}^{\rm hard}(0)$. 
For every $\epsilon > 0$ the tuples $(u^h+\epsilon\upsilon, e^h+\epsilon\eta, p^h+\epsilon\pi, \alpha^h+\epsilon\beta)$ belong to $\calA_{h}^{\rm hard}(w)$, hence
\begin{equation*}
	\calQ^{\rm el}_{h}(e^h+\epsilon\eta) + \calQ^{\rm hard}_{h}(p^h+\epsilon\pi, \alpha^h+\epsilon\beta) + \calR_{h}(\epsilon\pi, \epsilon\beta) - \big\langle \ell(t), \epsilon\upsilon \big\rangle
	\geq \calQ^{\rm el}_{h}(e^h) + \calQ^{\rm hard}_{h}(p^h, \alpha^h).
\end{equation*}
Using the positive homogeneity of $\calR_{h}$ and linearity of $\ell$, we obtain
\begin{equation*}
	\calQ^{\rm el}_{h}(e^h+\epsilon\eta) + \calQ^{\rm hard}_{h}(p^h+\epsilon\pi, \alpha^h+\epsilon\beta) + \epsilon\calR_{h}(\pi, \beta) - \epsilon\big\langle \ell(t), \upsilon \big\rangle
	\geq \calQ^{\rm el}_{h}(e^h) + \calQ^{\rm hard}_{h}(p^h, \alpha^h).
\end{equation*}
A simple calculation then gives
\begin{align*}
	& \epsilon\int_{\Omega} \left( \C\!\left(x'\right) \Lambda_h e^h : \Lambda_h \eta + \Hkin\!\left(x'\right) \Lambda_h p^h : \Lambda_h \pi + \Hiso\!\left(x'\right) \alpha^h \cdot \beta \right) \,dx
	\\& + \epsilon^2\calQ^{\rm el}_{h}(\eta) + \epsilon^2\calQ^{\rm hard}_{h}(\pi, \beta) + \epsilon\calR_{h}(\pi, \beta) - \epsilon\big\langle \ell(t), \upsilon \big\rangle 
	\geq 0.
\end{align*}
Dividing by $\epsilon$ and sending $\epsilon \to 0$ yields \eqref{minimality condition h}. 

For the converse implication, let $(u^h, e^h, p^h, \alpha^h) \in \calA_{h}^{\rm hard}(w)$ satisfy \eqref{minimality condition h}.
Then for arbitrary tuple $(\upsilon, \eta, \pi, \beta) \in \calA_{h}^{\rm hard}(w)$ we have
\begin{equation*}
	(\upsilon-u^h, \eta-e^h, \pi-p^h, \beta-\alpha^h) \in \calA_{h}^{\rm hard}(0).
\end{equation*}
Thus
{\allowdisplaybreaks
\begin{align*}
	& \calR_{h}(\pi-p^h, \beta-\alpha^h) 
	\\*& \geq - \int_{\Omega} \left( \C\!\left(x'\right) \Lambda_h e^h : \Lambda_h (\eta-e^h) + \Hkin\!\left(x'\right) \Lambda_h p^h : \Lambda_h (\pi-p^h) + \Hiso\!\left(x'\right) \alpha^h \cdot (\beta-\alpha^h) \right) \,dx
	\\*& \hspace{1em} + \big\langle \ell(t), \upsilon-u^h \big\rangle\\
	& = \calQ^{\rm el}_{h}(e^h) + \calQ^{\rm el}_{h}(\eta-e^h) - \calQ^{\rm el}_{h}(\eta) + \calQ^{\rm hard}_{h}(p^h, \alpha^h) + \calQ^{\rm hard}_{h}(\pi-p^h, \beta-\alpha^h) - \calQ^{\rm hard}_{h}(\pi, \beta)
	\\*& \hspace{1em} + \big\langle \ell(t), \upsilon-u^h \big\rangle,
\end{align*}
}\noindent
where the last equality is a straightforward computation. 
From the above, we immediately deduce
\begin{align*}
	& \calQ^{\rm el}_{h}(\eta) + \calQ^{\rm hard}_{h}(\pi, \beta) + \calR_{h}(\pi-p^h, \beta-\alpha^h) - \big\langle \ell(t), \upsilon-u^h \big\rangle
	\\& \geq \calQ^{\rm el}_{h}(e^h) + \calQ^{\rm el}_{h}(\eta-e^h) + \calQ^{\rm hard}_{h}(p^h, \alpha^h) + \calQ^{\rm hard}_{h}(\pi-p^h, \beta-\alpha^h) 
	\geq \calQ^{\rm el}_{h}(e^h) + \calQ^{\rm hard}_{h}(p^h, \alpha^h),
\end{align*}
hence $(u^h, e^h, p^h, \alpha^h) \in \calA_{h}^{\rm hard}(w)$ is a solution to \eqref{minimum problem h}.
\end{proof}

\subsection{\texorpdfstring{$\Gamma$}{Γ}-convergence results} \label{section gamma converegence results} 

To characterize the limit problem we will need to introduce the Kirchhoff--Love displacements and moments as well as the reduced elasticity tensor. In this section we will also, following the standard $\Gamma$--convergence approach, prove the compactness statement, required lower semicontinuity result and we will make the construction of the recovery sequence. These results will be used in \Cref{seclimquasiev} to prove the main convergence theorem. 

\medskip
\noindent{\bf Kirchhoff--Love displacements and moments.} \nopagebreak

We consider the set of {\em Kirchhoff--Love displacements}, defined as
\begin{equation*}
	BD_{KL}(\Omega) := \big\{u \in BD(\Omega) : \ ({\E}u)_{i3}=0 \quad\text{for } i=1,2,3\big\}.
\end{equation*}
We note that $u \in BD_{KL}(\Omega)$ if and only if $u_3 \in BH(\omega)$ and there exists ${\PINK\bar{u}\BLACK} \in BD(\omega)$ such that 
\begin{equation*}
	u_{\alpha}={\PINK\bar{u}\BLACK}_{\alpha}-x_3\partial_{x_\alpha}u_3, \quad \alpha=1,2.
\end{equation*}
Namely, if $u \in BD_{KL}(\Omega)$, then 
\begin{equation*} \label{symmetric gradient of KL functions}
	{\E}u = \begin{pmatrix} \begin{matrix} {\E}_{x'}{\PINK\bar{u}\BLACK} - x_3 {\D}^2_{x'}u_3 \end{matrix} & \begin{matrix} 0 \\ 0 \end{matrix} \\ \begin{matrix} 0 & 0 \end{matrix} & 0 \end{pmatrix}.
\end{equation*}
Here for ${\E}_{x'}$ denotes the symmetrized gradient of a function in $BD(\omega)$. 
We call ${\PINK\bar{u}\BLACK}, u_3$ the {\em Kirchhoff--Love components} of $u$. 
Furthermore, we will consider the space
\begin{equation*}
	H^1_{KL}(\Omega) := BD_{KL}(\Omega) \cap H^1(\Omega;\R^3).
\end{equation*}
Note that if $u \in H^1_{KL}(\Omega)$, then $\bar{u} \in H^1(\omega;\R^2)$, $u_3 \in H^2(\omega)$.

Since the space
\begin{equation*}
	\big\{\xi\in \M^{3 \times 3}_{\sym}: \ \xi_{i3}=0 \text{ for }i=1,2,3\big\}
\end{equation*}
is canonically isomorphic to $\M^{2 \times 2}_{\sym}$, given a function $u \in H^1_{KL}(\Omega)$ we will usually identify ${\E}u$ with a function in $L^2(\Omega;\M^{2 \times 2}_{\sym})$. 

We will need the following definitions. 

\begin{definition} \label{moments of functions}
	For $f \in L^2(\Omega;\M^{2 \times 2}_{\sym})$ we denote by $\zerothI{f}$, $\firstI{f} \in L^2(\omega;\M^{2 \times 2}_{\sym})$ and $f^\perp \in L^2(\Omega;\M^{2 \times 2}_{\sym})$ the following components of $f$:
	\begin{equation*}
		\zerothI{f}(x')
		:= \int_{I} f(x',x_3)\, dx_3, \qquad
		\firstI{f}(x')
		:= 12 \int_{I} x_3 f(x',x_3)\,dx_3
	\end{equation*}
	for $x' \in \omega$, and
	\begin{equation*}
		f^\perp(x) := f(x) 
		- \zerothI{f}(x') - x_3 \firstI{f}(x')
	\end{equation*}
	for $x \in \Omega$.
	We name $\zerothI{f}$ the \emph{zeroth order moment} of $f$ and $\firstI{f}$ the \emph{first order moment} of $f$.
\end{definition}
The coefficient in the definition of $\firstI{f}$ is chosen from the computation $\int_{I} x_3^2 \,dx_3 = \frac{1}{12}$. 
It ensures that if $f$ is of the form $f(x) = x_3 g(x')$, for some $g \in L^2(\omega;\M^{2 \times 2}_{\sym})$, then $\firstI{f} = g$.


Analogously, we have the following definition of zeroth and first order moments of measures. 

\begin{definition} \label{moments of measures}
	For $\mu \in M_b(\Omega \cup \Gamma_\Dir;\M^{2 \times 2}_{\sym})$ we define ${\PINK\zerothI{\mu}\BLACK}$, ${\PINK\firstI{\mu}\BLACK} \in M_b(\omega \cup \gamma_\Dir;\M^{2 \times 2}_{\sym})$ and $\mu^\perp \in M_b(\Omega \cup \Gamma_\Dir;\M^{2 \times 2}_{\sym})$ as follows:
	\begin{equation*}
		\int_{\omega \cup \gamma_\Dir} \varphi: d{\PINK\zerothI{\mu}\BLACK} := \int_{\Omega \cup \Gamma_\Dir} \varphi: d\mu, \qquad
		\int_{\omega \cup \gamma_\Dir} \varphi: d{\PINK\firstI{\mu}\BLACK} := 12\int_{\Omega \cup \Gamma_\Dir} x_3 \varphi: d\mu 
	\end{equation*}
	for every $\varphi \in C(\omega \cup \gamma_\Dir;\M^{2 \times 2}_{\sym})$, and
	\begin{equation*}
		\mu^\perp := \mu - {\PINK\zerothI{\mu}\BLACK} \otimes \calL^{1}_{x_3} - {\PINK\firstI{\mu}\BLACK} \otimes x_3 \calL^{1}_{x_3},
	\end{equation*}
	where $\otimes$ is the usual product of measures, and $\calL^{1}_{x_3}$ is the Lebesgue measure restricted to the third component of $\R^3$. 
	We call ${\PINK\zerothI{\mu}\BLACK}$ the \emph{zeroth order moment} of $\mu${\RED,\BLACK} and ${\PINK\firstI{\mu}\BLACK}$ the \emph{first order moment} of $\mu$.
\end{definition}

\noindent{\bf The reduced elasticity tensor.}

For a fixed $x' \in \omega$, let $\M_{x'} : \M^{2 \times 2}_{\sym} \to \M^{3 \times 3}_{\sym}$ be the operator given by
\[
\M_{x'} \xi
:=
\begin{pmatrix} \begin{matrix} \xi \end{matrix} & \begin{matrix} \lambda^{x'}_1(\xi) \\ \lambda^{x'}_2(\xi) \end{matrix} \\ \begin{matrix} \lambda^{x'}_1(\xi) & \lambda^{x'}_2(\xi) \end{matrix} & \lambda^{x'}_3(\xi) \end{pmatrix}
\quad \text{for every }\xi \in \M^{2 \times 2}_{\sym},
\]
where for every $\xi \in \M^{2 \times 2}_{\sym}$ the tuple $(\lambda^{x'}_1(\xi),\lambda^{x'}_2(\xi),\lambda^{x'}_3(\xi))$ is the unique solution to the minimum problem
\begin{equation*}
	\min_{\lambda^{x'}_i \in \R} \C(x') \begin{pmatrix} \begin{matrix} \xi \end{matrix} & \begin{matrix} \lambda^{x'}_1 \\ \lambda^{x'}_2 \end{matrix} \\ \begin{matrix} \lambda^{x'}_1 & \lambda^{x'}_2 \end{matrix} & \lambda^{x'}_3 \end{pmatrix} : \begin{pmatrix} \begin{matrix} \xi \end{matrix} & \begin{matrix} \lambda^{x'}_1 \\ \lambda^{x'}_2 \end{matrix} \\ \begin{matrix} \lambda^{x'}_1 & \lambda^{x'}_2 \end{matrix} & \lambda^{x'}_3 \end{pmatrix}.
\end{equation*}
We observe that for every $\xi \in \M^{2 \times 2}_{\sym}$, the matrix $\M_{x'}\xi$ is given by the unique solution of the linear system
\begin{equation} \label{C tensor KKT condition}
	\C(x') \M_{x'} \xi : \begin{pmatrix} 0 & 0 & \lambda_1 \\ 0 & 0 & \lambda_2 \\ \lambda_1 & \lambda_2 & \lambda_3 \end{pmatrix} = 0 \quad \text{for every } \lambda_1, \lambda_2, \lambda_3 \in \R.
\end{equation}
This implies, in particular, for every $x' \in \omega$ that $\M_{x'}$ is a linear map and
$
(\C(x') \M_{x'} \xi)_{i3} = (\C(x') \M_{x'} \xi)_{3i} = 0.
$

We then define the \emph{reduced elasticity tensor} $\red{\C} : \omega \times \M^{2 \times 2}_{\sym} \to \M^{3 \times 3}_{\sym}$, given by
\[
\red{\C}(x') \xi := \C(x') \M_{x'} \xi \quad \text{for every } \xi \in \M^{2 \times 2}_{\sym}.
\]
We remark that by \eqref{C tensor KKT condition} we have
\begin{equation} \label{Cred 2x2 : 3x3}
	\red{\C}(x') \xi : \zeta = \C(x') \M_{x'} \xi : \begin{pmatrix} \zeta^{1,2} & 0 \\ 0 & 0 \end{pmatrix} = \C(x') \M_{x'} \xi : \M_{x'} \zeta^{1,2} \quad\text{for every } \xi\in \M^{2 \times 2}_{\sym},\,\zeta \in \M^{3 \times 3}_{\sym},
\end{equation}
and 
$
(\red{\C}(x') \xi)_{i3} = (\red{\C}(x') \xi)_{3i} = 0.
$
Consequently, the reduced elasticity tensor can be considered as a map from $\omega$ taking values in the set of $2 \times 2$ symmetric positive definite linear operators, i.e. by a slight abuse of notation we will, when needed, assume that $\red{\C}(x'): \M^{2 \times 2}_{\sym} \to \M^{2 \times 2}_{\sym}$.


\medskip
\noindent{\bf Assumptions on the loads and boundary datum.}

We will assume that the function $w$ whose trace on $\Gamma_D$ defines the boundary data is such that 
\begin{equation} \label{assbd} 
	w \in W^{1,1} (0,T;H^1_{KL}(\Omega)). 
\end{equation}
In particular, it is of the form 
\[
	w(x):=(\bar{w}_1(x')-{x_3}\partial_1 w_3(x'), \bar{w}_2(x')-x_3\partial_2 w_3(x'), w_3(x'))\quad\text{for a.e.\ }x \in \Omega, 
\]
where ${\PINK\bar{w}\BLACK}_\alpha \in W^{1,1}(0,T;H^1(\omega))$, $\alpha=1,2$, and ${\PINK\bar{w}\BLACK}_3 \in W^{1,1}(0,T;H^2(\omega))$.

We also assume that 
\begin{equation} \label{assloads} 
	\fvol \in W^{1,1}(0,T;L^2(\Omega;\R^3)), \quad \fsurf \in W^{1,1}(0,T;L^2(\partial \Omega \setminus \Gamma_D;\R^3)). 
\end{equation}

\subsubsection{Compactness result}
We first recall a compactness result for sequences of non-oscillating fields with uniformly bounded scaled symmetric gradients.

\begin{proposition} \label{two-scale weak limit of scaled strains - 2x2 submatrix}
Let $\{u^h\}_{h>0} \subset H^1(\Omega)$ be a sequence such that there exists a constant $C>0$ for which 
\begin{equation*}
	\|u^h\|_{L^2(\Omega;\R^3)}+\|\Lambda_h {\E}u^h\|_{L^2(\Omega;\M^{3 \times 3}_{\sym})} \leq C.
\end{equation*}
Then, there exist functions $\bar{u} = (\bar{u}_1, \bar{u}_2) \in H^1(\omega;\R^2)$ and $u_3 \in H^2(\omega)$ such that, up to subsequences, it holds 
\begin{align*}
	u^h_{\alpha} & \strong \bar{u}_{\alpha}-x_3 \partial_{x_\alpha}u_3, \quad \text{strongly in }L^2(\Omega), \quad \alpha \in \{1, 2\},\\
	u^h_3 & \strong u_3, \quad \text{strongly in }L^2(\Omega),\\
	{\E}u^h & \weak \begin{pmatrix} {\E}_{x'}\bar{u} - x_3 {\D}^2_{x'}u_3 & 0 \\ 0 & 0 \end{pmatrix} \quad \text{weakly in }L^2(\Omega;\M^{3 \times 3}_{\sym}).
\end{align*}
\end{proposition}

\begin{proof}
See the proof of \cite[Theorem 1.4.-1]{ciarlet1997mathematical}.
\end{proof} 

\subsubsection{Lower semicontinuity of energy functionals}
Motivated by the result in the previous section, we introduce the limiting class
\begin{align*}
	\calA_{0}^{\rm hard}(w) := \Big\{ 
	& (u, e, p, \alpha) \in H^1_{KL}(\Omega) \times L^2(\Omega;\M^{2 \times 2}_{\sym}) \times L^2(\Omega;\M^{3 \times 3}_{\dev}) \times L^2(\Omega) :
	\\& {\E}_{x'}\bar{u} - x_3 {\D}^2_{x'}u_3 = e + p^{1,2} \quad \text{ in } \Omega,
	\\& u = w \quad {\PINK\text{ on } \Gamma_\Dir\BLACK},
	\\& \disspot(x', p(x)) \leq \alpha(x) \quad \text{a.e.} \text{ in } \Omega 
	\Big\},
\end{align*}
where $\bar{u} \in H^1(\omega;\R^2)$ and $u_3 \in H^2(\omega)$ are the Kirchhoff--Love components of $u$.

For $(u, e, p, \alpha) \in \calA_{0}^{\rm hard}(w)$ we now define
\begin{align}
	\calQ^{\rm el}_{0}(e) & := \frac{1}{2} \int_{\Omega} \red{\C}(x') e(x) : e(x) \,dx, \label{definition Q^el_0}\\
	\calQ^{\rm hard}_{0}(p, \alpha) & := \frac{1}{2} \int_{\Omega} \Hkin(x') p(x) : p(x) \,dx + \frac{1}{2} \int_{\Omega} \Hiso(x') \alpha(x) \cdot \alpha(x) \,dx, \label{definition Q^hard_0}
\end{align}
and
\begin{equation} \label{definition R_0}
	\calR_{0}(p, \alpha) := \int_{\Omega} \gendisspot(x', p(x), \alpha(x)) \,dx.
\end{equation}
We again note that the condition $\disspot(x', p(x)) \leq \alpha(x)$ a.e. in $\Omega$ implies that $\calR_{0}(p, \alpha) = \int_{\Omega} \disspot(x', p(x)) \,dx$.
The following theorem gives us lower semicontinuity result. 
\begin{theorem} \label{lower semicontinuity of energies - h to 0}
Let $w \in H^1_{KL}(\Omega)$ and let $(u^h, e^h, p^h, \alpha^h) \in \calA_{h}^{\rm hard}( w)$ be such that
\begin{align}
	\label{lower semicontinuity convergence to u} & u^h \weak u \quad \text{weakly in $H^1(\Omega;\R^3)$},\\
	\label{lower semicontinuity convergence to e} & \Lambda_h e^h \weak \M_{x'} e \quad \text{weakly in $L^2(\Omega;\M^{3 \times 3}_{\sym})$},\\
	\label{lower semicontinuity convergence to p}& \Lambda_h p^h \weak p \quad \text{weakly in $L^2(\Omega;\M^{3 \times 3}_{\dev})$},\\
	\label{lower semicontinuity convergence to alpha}& \alpha^h \weak \alpha \quad \text{weakly in $L^2(\Omega)$},
\end{align}
for some $(u, e, p{\RED,\BLACK} \alpha)\in H^1(\Omega;\R^3) \times L^2(\Omega;\M^{2 \times 2}_{\sym}) \times L^2(\Omega;\M^{3 \times 3}_{\dev}) \times L^2(\Omega) $. 
Then $(u, e, p, \alpha) \in \calA_{0}^{\rm hard}(w)$ and we have
\begin{equation} \label{lower semicontinuity Eh}
	\calQ^{\rm el}_{0}(e) + \calQ^{\rm hard}_{0}(p, \alpha) - \big\langle \ell(t), u \big\rangle 
	\leq \liminf\limits_{h} \left( \calQ^{\rm el}_{h}(e^h) + \calQ^{\rm hard}_{h}(p^h, \alpha^h) - \big\langle \ell(t), u^h \big\rangle \right)
\end{equation}
and
\begin{equation} \label{lower semicontinuity Rh}
	\calR_{0}(p, \alpha) \leq \liminf\limits_{h} \calR_{h}(p^h, \alpha^h).
\end{equation}
\end{theorem}

\begin{proof}
The fact that $u \in H_{KL}^1(\Omega)$ follows from \Cref{two-scale weak limit of scaled strains - 2x2 submatrix}.
Immediately from the convexity of the quadratic forms we conclude
\begin{equation} \label{doleprva} 
	\calQ^{\rm hard}_{0}(p, \alpha) \leq \liminf\limits_{h} \calQ^{\rm hard}_{h}(p^h, \alpha^h). 
\end{equation}
Since $\C(x') \xi:\xi \geq \red{\C}(x') \xi:\xi$ for every $x' \in \omega$, $\xi \in \M^{3 \times 3}_{\sym})$ and from the convexity of the quadratic forms we conclude that 
\begin{equation} \label{doledruga}
	\calQ^{\rm el}_{0}(e) \leq \liminf\limits_{h} \calQ^{\rm el}_{h}(e^h). 
\end{equation}
By using \eqref{lower semicontinuity convergence to u} and the compactness of the trace we conclude that for every $t \in [0,T]$ we have 
\begin{equation} \label{doletreca} 
\lim_{h} \left\langle \ell(t), u^h\right \rangle = \left\langle \ell(t), u\right \rangle. 
\end{equation}
From \eqref{doleprva}, \eqref{doledruga} and \eqref{doletreca} we conclude \eqref{lower semicontinuity Eh}, while \eqref{lower semicontinuity Rh} directly follows from the convexity and \eqref{lower semicontinuity convergence to p} and \eqref{lower semicontinuity convergence to alpha}.
The fact that $(u, e, p, \alpha) \in \calA_{0}^{\rm hard}(w)$ follows from the fact that $u \in H_{KL}^1(\Omega)$ and the convexity of the dissipation potential $\disspot$ and Mazur's theorem. 
\end{proof}

\subsubsection{Construction of the recovery sequence}
The following theorem provides us the recovery sequence. 
\begin{theorem} \label{recovery sequence}
Let $w \in H^1_{KL}(\Omega)$ and let $(u, e, p, \alpha) \in \calA_{0}^{\rm hard}(w)$.
Then, there exists a sequence of tuples $(u^h, e^h, p^h, \alpha^h) \in \calA_{h}^{\rm hard}(w)$ such that
\begin{align}
	& u^h \weak u \quad \text{weakly in $H^1(\Omega;\R^3)$}, \label{recovery sequence convergence u}\\
	& \Lambda_h e^h \strong \M_{x'} e \quad \text{strongly in $L^2(\Omega;\M^{3 \times 3}_{\sym})$}, \label{recovery sequence convergence e}\\
	& \Lambda_h p^h \strong p \quad \text{strongly in $L^2(\Omega;\M^{3 \times 3}_{\dev})$}, \label{recovery sequence convergence p}\\
	& \alpha^h \strong \alpha \quad \text{strongly in $L^2(\Omega)$}, \label{recovery sequence convergence alpha}\\
	& \calR_{h}(p^h, \alpha^h) \to \calR_{0}(p, \alpha). \label{recovery sequence convergence dissipation}
\end{align}
\end{theorem}

\begin{proof}
For $i \in \{1, 2, 3\}$, we argue as in \cite[Proposition 4.1]{Liero.Mielke.2011} and we approximate functions $e_{i3}$ and $p_{i3}$ by means of elliptic regularizations. 
Let $\eta^h_i, \pi^h_i \in H^1_0(\Omega)$ be such that
\begin{align}
	& \eta^h_i \strong (\M_{x'} e)_{i3}\quad \text{strongly in $L^2(\Omega)$}, \label{elliptic regularizations convergence e}\\ 
	& \pi^h_i \strong p_{i3}\quad \text{strongly in $L^2(\Omega)$}, \label{elliptic regularizations convergence p}
\end{align}
as $h \to 0$, and
\begin{equation} \label{elliptic regularizations bounds}
	\|\nabla \eta^h_i\|_{L^2(\Omega)} \leq C h^{-\frac{1}{2}}, \quad \|\nabla \pi^h_i\|_{L^2(\Omega)} \leq C h^{-\frac{1}{2}}.
\end{equation}
Let functions $V^h \in L^2(\Omega;\R^3)$ be defined by
\begin{equation*}
	V^h_i(x',x_3) := \int_{-\frac{1}{2}}^{x_3} \left( \eta^h_i(x',\zeta) + \pi^h_i(x',\zeta) \right) \,d\zeta, \quad i \in \{1, 2, 3\}.
\end{equation*}
Then we construct the recovery sequence by setting
\begin{equation*}
	u^h(x) := \left(\begin{array}{c} \bar{u}(x') - x_3 \nabla_{x'} u_3(x') \\ u_3(x') \end{array}\right) + \left(\begin{array}{c} h\,(V^h)^{1,2}(x) \\ h^2\,V^h_3(x) \end{array}\right),
\end{equation*}
\begin{align*}
	e^h := \begin{pmatrix} \begin{matrix} e \end{matrix} & \begin{matrix} h\,\eta^h_1 \\ h\,\eta^h_2 \end{matrix} \\ \begin{matrix} \sym \end{matrix} & h^2\,\eta^h_3 \end{pmatrix} + \begin{pmatrix} h\,{\E}_{x'}(V^h)^{1,2} & \frac{h^2}{2}\,\nabla_{x'}V^h_3 \\ \sym & h^2\,(\pi^h_3 - p_{33}) \end{pmatrix}, \quad
	p^h := \begin{pmatrix} \begin{matrix} p^{1,2} \end{matrix} & \begin{matrix} h\,\pi^h_1 \\ h\,\pi^h_2 \end{matrix} \\ \begin{matrix} \sym \end{matrix} & h^2\,p_{33} \end{pmatrix},
\end{align*}
and
\begin{align*}
	\alpha^h := \alpha + \disspot(x', \Lambda_h p^h - p).
\end{align*}
{\PINK Above we used the abbreviation "$\sym$" in the last row to indicate that the entries of the last column are also used here (so that the matrices are symmetric).\BLACK}

Since $u = w$ on $\Gamma_\Dir$ and $\eta^h_i = \pi^h_i = 0$ on $\partial\Omega$, we have that $u^h = w$  on $\Gamma_\Dir$.
From \eqref{elliptic regularizations convergence p} we obtain \eqref{recovery sequence convergence p}, while \eqref{elliptic regularizations convergence e}--\eqref{elliptic regularizations bounds} imply \eqref{recovery sequence convergence e}.
From the calculation
\begin{equation*}
	{\E}u^h = \begin{pmatrix} {\E}_{x'}\bar{u} - x_3 {\D}^2_{x'}u_3 & 0 \\ 0 & 0 \end{pmatrix} + \begin{pmatrix} h\,{\E}_{x'}(V^h)^{1,2} & h\,\partial_{x_3}(V^h)^{1,2} + \frac{h^2}{2}\,\nabla_{x'}V^h_3 \\ \sym & h^2\,\partial_{x_3}V^h_3 \end{pmatrix}
\end{equation*}
it is easy to check that ${\E}u^h = e^h+p^h$ and \eqref{recovery sequence convergence u} holds. 
Furthermore, for a.e. $x \in \Omega$ we have 
\begin{align*}
	\disspot(x', \Lambda_h p^h(x))
	& \leq \disspot(x',p(x)) + \disspot(x', \Lambda_h p^h(x) - p(x)) 
	\\& \leq \alpha(x) + \disspot(x', \Lambda_h p^h(x) - p(x)) 
	= \alpha^h(x),
\end{align*}
from which we deduce $(u^h, e^h, p^h, \alpha^h) \in \calA_{h}^{\rm hard}(w)$.
Lastly, since from the definition of $\alpha^h$ and \eqref{coercivity of R} we have
\begin{align*}
	| \alpha^h(x) - \alpha(x) | = \disspot(x', \Lambda_h p^h(x) - p(x)) \;\text{ and }\; | \disspot(x', \Lambda_h p^h(x)) - \disspot(x', p(x)) | \leq R_K | \Lambda_h p^h(x) - p(x)) |,
\end{align*}
from \eqref{recovery sequence convergence p} we infer \eqref{recovery sequence convergence alpha} and \eqref{recovery sequence convergence dissipation}.
\end{proof}

\subsection{Limiting quasistatic evolution} \label{seclimquasiev} 

In this section we provide the definition of the limiting quasistatic evolution and prove the convergence statement in \Cref{main result 1} and \Cref{improved convergence} which are the main results of the first part of the paper. 

We recall the definition of energy functionals $\calQ^{\rm el}_{0}$, $\calQ^{\rm hard}_{0}$ and $\calR_{0}$ given in \eqref{definition Q^el_0}, \eqref{definition Q^hard_0} and \eqref{definition R_0}. 
The associated $\calR_{0}$-variation of a map $(p, \alpha) : [0, T] \to L^2(\Omega;\M^{3 \times 3}_{\dev}) \times L^2(\Omega)$ on $[a,b]$ is then defined as
\begin{equation*}
	\calV_{\calR_{0}}(p, \alpha; a, b) := \sup\left\{ \sum_{i = 1}^{n} \calR_{0}\!\left(p(t_{i+1}) - p(t_i), \alpha(t_{i+1}) - \alpha(t_i)\right) : a = t_1 < t_2 < \ldots < t_n = b,\ n \in \N \right\}.
\end{equation*}

We now give the notion of the limiting quasistatic elasto-plastic evolution.

\begin{definition} \label{limiting quasistatic evolution}
A \emph{limiting quasistatic evolution} for the boundary datum $w(t)$ is a function $t \mapsto (u(t), e(t), p(t), \alpha(t))$ from $[0, T]$ into $H^1_{KL}(\Omega) \times L^2(\Omega;\M^{2 \times 2}_{\sym}) \times L^2(\Omega;\M^{3 \times 3}_{\dev}) \times L^2(\Omega)$ which satisfies the following conditions:
\begin{enumerate}[label=(qs\arabic*)$_{0}$]
	\item \label{0-qs S} for every $t \in [0, T]$ we have $(u(t), e(t), p(t), \alpha(t)) \in \calA_{0}^{\rm hard}(w(t))$ and
	\begin{align*}
		& \calQ^{\rm el}_{0}(e(t)) + \calQ^{\rm hard}_{0}(p(t), \alpha(t)) 
		- \big\langle \ell(t), u(t) \big\rangle
		\\& \leq \calQ^{\rm el}_{0}(\eta) + \calQ^{\rm hard}_{0}(\pi, \beta) + \calR_{0}(\pi - p(t), \beta - \alpha(t)) 
		- \big\langle \ell(t), \upsilon \big\rangle,
	\end{align*}
	for every $(\upsilon, \eta, \pi, \beta) \in \calA_{0}^{\rm hard}(w(t))$.
	\item \label{0-qs E} 
	for every $t \in [0, T]$
	\begin{align*}
		& \calQ^{\rm el}_{0}(e(t)) + \calQ^{\rm hard}_{0}(p(t), \alpha(t)) 
		+ \calV_{\calR_{0}}(p, \alpha; 0, t) 
		- \big\langle \ell(t), u(t) \big\rangle
		\\& = \calQ^{\rm el}_{0}(e(0)) + \calQ^{\rm hard}_{0}(p(0), \alpha(0)) 
		- \big\langle \ell(0), u(0) \big\rangle 
		\\& + \int_0^t \left( \int_{\Omega} \red{\C}(x') e(s) : {\E}\dot{w}(s) \,dx - \big\langle \ell(s), \dot{w}(s) \big\rangle - \big\langle \dot{\ell}(s), u(s) \big\rangle \right) \,ds.
	\end{align*}
\end{enumerate}
\end{definition}

Recalling the definition of a $h$-quasistatic evolution for the boundary datum $w(t)$ given in \Cref{h-quasistatic evolution}, we are in a position to formulate the main result of this section.

\begin{theorem} \label{main result 1}
{\PINK Let the loads and boundary data of the problem satisfy \eqref{assbd} and \eqref{assloads}.\BLACK}
Assume that there exists a sequence of tuples $(u^h_0, e^h_0, p^h_0, \alpha^h_0) \in \calA_{h}^{\rm hard}(w(0))$ such that
\begin{align}
	& u^h_0 \weak u_0 \quad \text{weakly in $H^1(\Omega;\R^3)$}, \label{main result u^h_0 condition}\\
	& \Lambda_h e^h_0 \strong \M_{x'}e_0 \quad \text{strongly in $L^2(\Omega;\M^{3 \times 3}_{\sym})$}, \label{main result e^h_0 condition}\\
	& \Lambda_h p^h_0 \strong p_0 \quad \text{strongly in $L^2(\Omega;\M^{3 \times 3}_{\dev})$}, \label{main result p^h_0 condition}\\
	& \alpha^h_0 \strong \alpha_0 \quad \text{strongly in $L^2(\Omega)$}, \label{main result alpha^h_0 condition}
\end{align}
for some $(u_0, e_0, p_0, \alpha_0) \in \calA_{0}^{\rm hard}(w(0))$. 
For every $h > 0$, let 
\begin{equation*}
	t \mapsto (u^h(t), e^h(t), p^h(t), \alpha^h(t))
\end{equation*}
be a $h$-quasistatic evolution for the boundary datum $w(t)$ such that $u^h(0) = u^h_0$, $e^h(0) = e^h_0$, $p^h(0) = p^h_0$ and $\alpha^h(0) = \alpha_0$. 
Then, there exists a limiting quasistatic evolution
\vspace{-0.25\baselineskip}
\begin{equation*}
	t \mapsto (u(t), e(t), p(t), \alpha(t)) \vspace{-0.25\baselineskip}
\end{equation*}
for the boundary datum $w(t)$ such that $u(0) = u_0$,\, $e(0) = e_0$, $p(0) = p_0$, and $\alpha(0) = \alpha_0$, and such that on a subsequence
for every $t \in [0, T]$ we have 
\begin{align}
	& u^h(t) \weak u(t) \quad \text{weakly in $H^1(\Omega;\R^3)$}, \label{main result u^h(t)}\\
	& \Lambda_h e^h(t) \weak \M_{x'} e(t) \quad \text{weakly in $L^2(\Omega;\M^{3 \times 3}_{\sym})$}, \label{main result e^h(t)}\\
	& \Lambda_h p^h(t) \weak p(t) \quad \text{weakly in $L^2(\Omega;\M^{3 \times 3}_{\dev})$}, \label{main result p^h(t)}\\
	& \alpha^h(t) \weak \alpha(t) \quad \text{weakly in $L^2(\Omega)$}. \label{main result alpha^h(t)}
\end{align}
\end{theorem}

\begin{proof}
The proof is divided into steps, in the spirit of evolutionary $\Gamma$--convergence. 

\medskip
\noindent{\bf Step 1: \em Compactness.} \nopagebreak

Firstly, we can prove that that there exists a constant $C$, depending only on the initial and boundary data, such that
\begin{align} \label{boundness in time 1}
\begin{split}
    \sup_{t \in [0, T]} \left\|\Lambda_h e^h(t)\right\|_{L^2(\Omega;\M^{3 \times 3}_{\sym})} \leq C,
	\qquad \sup_{t \in [0, T]} \left\|\Lambda_h p^h(t)\right\|_{L^2(\Omega;\M^{3 \times 3}_{\dev})} \leq C,\\ \sup_{t \in [0, T]} \left\|\alpha^h(t)\right\|_{L^2(\Omega)} \leq C 
	\;\text{ and }\; \sup_{t \in [0, T]} \left\|u^h(t)\right\|_{H^1(\Omega;\M^{3 \times 3}_{\sym})} \leq C, 
\end{split}
\end{align}
for every $h>0$. 
Indeed, we have that
\begin{align}
\nonumber	\left\|u^h(t)\right\|_{H^1(\Omega;\R^3)} & \leq C \left(\left\|u^h(t)\right\|_{L^2(\Gamma_\Dir;\R^3)} + \left\|{\E}u^h(t)\right\|_{L^2(\Omega;\M^{3 \times 3}_{\sym})}\right)\\
\nonumber
	& \leq C \left(\left\|w(t)\right\|_{L^2(\Gamma_\Dir;\R^3)} + \left\|e^h(t)\right\|_{L^2(\Omega;\M^{3 \times 3}_{\sym})} + \left\|p^h(t)\right\|_{L^2(\Omega;\M^{3 \times 3}_{\sym})}\right)\\
	\label{mmarin1}
	& \leq C \left(\left\|w(t)\right\|_{L^2(\Omega;\R^3)} + \left\|\Lambda_h e^h(t)\right\|_{L^2(\Omega;\M^{3 \times 3}_{\sym})} + \left\|\Lambda_h p^h(t)\right\|_{L^2(\Omega;\M^{3 \times 3}_{\dev})}\right).
\end{align} 
By using the energy balance of the $h$-quasistatic evolution \ref{h-qs E} 
and the coercivity of the quadratic energies $\calQ^{\rm el}_{h}$ and $\calQ^{\rm hard}_{h}$
\begin{align*}
	& r_{\C} \left\|\Lambda_h e^h(t)\right\|_{L^2(\Omega;\M^{3 \times 3}_{\sym})}^2 + r_{\Hkin} \left\|\Lambda_h p^h(t)\right\|_{L^2(\Omega;\M^{3 \times 3}_{\dev})}^2 + r_{\Hiso} \left\|\alpha^h(t)\right\|_{L^2(\Omega)}^2 
	\\& \leq R_{\C} \left\|\Lambda_h e^h(0)\right\|_{L^2(\Omega;\M^{3 \times 3}_{\sym})}^2 + R_{\Hkin} \left\|\Lambda_h p^h(0)\right\|_{L^2(\Omega;\M^{3 \times 3}_{\dev})}^2 + R_{\Hiso} \left\|\alpha^h(0)\right\|_{L^2(\Omega)}^2 
	\\& + 2 R_{\C} \sup_{t \in [0, T]} \left\|\Lambda_h e^h(t)\right\|_{L^2(\Omega;\M^{3 \times 3}_{\sym})} \int_0^T \left\|{\E}\dot{w}(s)\right\|_{L^2(\Omega;\M^{3 \times 3}_{\sym})} \,ds\\ &+\sup_{t \in [0,T]}\|u(t) \|_{H^{1}(\Omega;\R^3)}\int_0^T \|\dot{l}(s)\|_{H^{-1}(\Omega;\R^3)} \,ds,
\end{align*} 
where the last two integral are well defined as $t \mapsto {\E}\dot{w}(t)$ belongs to $L^1([0, T];L^2(\Omega;\M^{3 \times 3}_{\sym}))$ and $t \mapsto \dot{l}(t)$ belongs to $L^1(0,T;H^{-1}(\Omega;\R^3))$.
In view of the boundedness of $\Lambda_h e^h_0$, $\Lambda_h p^h_0$ and $\alpha^h_0$ that is implied by \eqref{main result e^h_0 condition}--\eqref{main result alpha^h_0 condition}, property \eqref{boundness in time 1} for the first three terms now follows using the Cauchy-Schwarz inequality and \eqref{mmarin1}. The last property of \eqref{boundness in time 1} easily follows from \eqref{mmarin1}.

We infer from the absolute continuity of $w$ with respect to $t$ and \Cref{nakcor} 
that the families $\{\Lambda_h e^h\}_{h>0} \subset W^{1,1}(0,T;L^2(\Omega;\M^{3 \times 3}_{\sym}))$, $\{\Lambda_h p^h\}_{h>0} \subset W^{1,1}(0,T;L^2(\Omega;\M^{3 \times 3}_{\dev}))$, $\{\alpha^h\}_{h>0} \subset W^{1,1}(0,T;L^2(\Omega))$ are also equicontinuous.
Therefore, by the Ascoli--Arzel\`{a} theorem and the metrizability and compactness of weak topologies on the bounded sets{\RED,\BLACK} there exist three subsequences {\RED(not relabeled)\BLACK} and three absolutely continuous functions $\tilde{e} : [0, T] \to L^2(\Omega;\M^{3 \times 3}_{\sym})$, $p : [0, T] \to L^2(\Omega;\M^{3 \times 3}_{\dev})$ and $\alpha : [0, T] \to L^2(\Omega)$ such that
\begin{align}
	\label{compactness e^h(t)} & \Lambda_h e^h(t) \weak \tilde{e}(t) \quad \text{weakly in $L^2(\Omega;\M^{3 \times 3}_{\sym})$},\\
	\label{compactness p^h(t)} & \Lambda_h p^h(t) \weak p(t) \quad \text{weakly in $L^2(\Omega;\M^{3 \times 3}_{\dev})$},\\
	\label{compactness alpha^h(t)} & \alpha^h(t) \weak \alpha(t) \quad \text{weakly in $L^2(\Omega)$},
\end{align}
for every $t \in [0, T]$.
{\RED From \eqref{mmarin1} it follows that the sequences $\{u^h(t)\}$ are bounded in $H^1(\Omega;\R^3)$ uniformly with respect to $t$.
We extract a further subsequence (not relabeled) such that
\begin{align}
	\label{compactness u^h(t)} & u^h(t) \weak u(t) \quad \text{weakly in $H^1(\Omega;\R^3)$},
\end{align}
for every $t \in [0, T]$.\BLACK}
Furthermore, since $\Lambda_h {\E}u^h(t) = \Lambda_h e^h(t) + \Lambda_h p^h(t)$ in $\Omega$ for every $h > 0$ and $t \in [0, T]$, from  \Cref{two-scale weak limit of scaled strains - 2x2 submatrix} we can conclude that $u(t) \in H^1_{KL}(\Omega)$ and that $(u(t), \tilde{e}^{1,2}(t), p(t), \alpha(t)) \in \calA_{0}^{\rm hard}(w(t))$.

\medskip
\noindent{\bf Step 2: \em Global stability.} \nopagebreak

Let us fix $t \in [0, T]$ and let $(\upsilon, \eta, \pi, \beta) \in \calA_{0}^{\rm hard}(0)$. 
By \Cref{recovery sequence} there exist a sequence of tuples $(\upsilon^h, \eta^h, \pi^h, \beta^h) \in \calA_{h}^{\rm hard}(0)$ such that
\begin{align}
	& \upsilon^h \weak \upsilon \quad \text{weakly in $H^1(\Omega;\R^3)$}, \label{recovery sequence convergence upsilon}\\
	& \Lambda_h \eta^h \strong \M_{x'} \eta \quad \text{strongly in $L^2(\Omega;\M^{3 \times 3}_{\sym})$}, \label{recovery sequence convergence eta}\\
	& \Lambda_h \pi^h \strong \pi \quad \text{strongly in $L^2(\Omega;\M^{3 \times 3}_{\dev})$}, \label{recovery sequence convergence pi}\\
	& \beta^h \strong \beta \quad \text{strongly in $L^2(\Omega)$}, \label{recovery sequence convergence beta}\\
	& \calR_{h}(\pi^h, \beta^h) \to \calR_{0}(\pi, \beta) \label{recovery sequence convergence R}.
\end{align}
By \Cref{stability condition equivalence}, the stability condition \ref{h-qs S} for $(u^h(t), e^h(t), p^h(t), \alpha^h(t))$ implies
\begin{align*}
	- \calR_{h}(\pi^h, \beta^h) \leq 
	\int_{\Omega} \left( \C\!\left(x'\right) \Lambda_h e^h(t) : \Lambda_h \eta^h + \Hkin\!\left(x'\right) \Lambda_h p^h(t) : \Lambda_h \pi^h + \Hiso\!\left(x'\right) \alpha^h(t) \cdot \beta^h \right) \,dx 
	- \big\langle \ell(t), \upsilon^h \big\rangle
\end{align*}
for every $h > 0$.
From the convergences \eqref{compactness e^h(t)}--\eqref{compactness alpha^h(t)} and \eqref{recovery sequence convergence upsilon}--\eqref{recovery sequence convergence R}, and recalling \eqref{Cred 2x2 : 3x3}, we obtain 
\begin{align*}
	- \calR_{0}(\pi, \beta) \leq 
	\int_{\Omega} \left( \red{\C}\!\left(x'\right) \tilde{e}^{1,2}(t) : \eta + \Hkin\!\left(x'\right) p(t) : \pi + \Hiso\!\left(x'\right) \alpha(t) \cdot \beta \right) \,dx 
	- \big\langle \ell(t), \upsilon \big\rangle.
\end{align*}
As $(\upsilon, \eta, \pi, \beta)$ were arbitrary, arguing as in the proof of \Cref{stability condition equivalence} we can conclude that the above inequality implies that $(u(t), \tilde{e}^{1,2}(t), p(t), \alpha(t)) \in \calA_{0}^{\rm hard}(w(t))$ {\RED satisfies\BLACK} the global stability of the limiting quasistatic evolution \ref{0-qs S}.

\medskip
\noindent{\bf Step 3: \em Identification of the limiting scaled elastic strain.} \nopagebreak

We claim that the limit in \eqref{compactness e^h(t)} satisfies
\begin{equation} \label{elastic strain minimality}
	\tilde{e}(t) = \M_{x'} e(t)
\end{equation}
for every $t \in [0, T]${\RED, where $e(t) = \tilde{e}^{1,2}(t)$\BLACK}. 
Indeed, if we consider test functions of the form $(\pm\psi, \pm{\E}\psi, 0, 0)$ in \eqref{minimality condition h}, where $\psi \in H^1(\Omega;\R^3)$ with $\psi = 0$ on $\Gamma_\Dir$, then we get the equality 
\[
	\int_{\Omega} \C\!\left(x'\right) \Lambda_h e^h(t) : \Lambda_h {\E}\psi \,dx = \big\langle \ell(t), \psi \big\rangle,
\]
for every $h > 0$. 
Let now $v \in C_c^{\infty}(\Omega;\R^3)$ and $V \in C^{\infty}(\closure{\Omega};\R^3)$ be defined by
\begin{equation*}
	V(x',x_3) := \int_{-\frac{1}{2}}^{x_3} v(x',\zeta) \,d\zeta.
\end{equation*}
By putting
\begin{equation*}
	\psi(x) = \left(\begin{array}{c} 2h\,V_1(x) \\ 2h\,V_2(x) \\ h^2\,V_3(x) \end{array}\right),
\end{equation*}
and passing to the limit, it is easy to see that
\begin{equation*}
	\int_{\Omega} \C\!\left(x'\right) \tilde{e}(t) : \left(\begin{array}{ccc} 0 & 0 & v_1(x) \\ 0 & 0 & v_2(x) \\ v_1(x) & v_2(x) & v_3(x) \end{array} \right) \,dx = \int_{\Omega} \C\!\left(x'\right) \tilde{e}(t) : \left(\begin{array}{ccc} 0 & 0 & \partial_{x_3}V_1(x) \\ 0 & 0 & \partial_{x_3}V_2(x) \\ \partial_{x_3}V_1(x) & \partial_{x_3}V_2(x) & \partial_{x_3}V_3(x) \end{array} \right) \,dx = 0.
\end{equation*}
Since $v$ was arbitrary, we can conclude the condition \eqref{elastic strain minimality} by the characterization \eqref{C tensor KKT condition}.

\medskip
\noindent{\bf Step 4: \em Energy balance.} \nopagebreak

In order to prove energy balance of the two-scale quasistatic evolution \ref{0-qs E}, it is enough (by arguing as in, e.g., \cite[Theorem 4.7]{DalMaso.DeSimone.Mora.2006} and \cite[Theorem 2.7]{Francfort.Giacomini.2012}) to prove the energy inequality
\begin{align} \label{main result dr - step 3 inequality}
\begin{split}
	& \calQ^{\rm el}_{0}(e(t)) + \calQ^{\rm hard}_{0}(p(t), \alpha(t)) 
	+ \calV_{\calR_{0}}(p, \alpha; 0, t) 
	- \big\langle \ell(t), u(t) \big\rangle
	\\& \leq \calQ^{\rm el}_{0}(e(0)) + \calQ^{\rm hard}_{0}(p(0), \alpha(0)) 
	- \big\langle \ell(0), u(0) \big\rangle 
	\\& + \int_0^t \left( \int_{\Omega} \red{\C}(x') e(s) : {\E}\dot{w}(s) \,dx - \big\langle \ell(s), \dot{w}(s) \big\rangle - \big\langle \dot{\ell}(s), u(s) \big\rangle \right) \,ds.
\end{split}
\end{align}

For a fixed $t \in [0, T]$, let us consider a subdivision $0 = t_1 < t_2 < \ldots < t_n = t$ of $[0,t]$.
Due to the lower semicontinuity from \Cref{lower semicontinuity of energies - h to 0}, we have 
{\allowdisplaybreaks
\begin{align*}
	& \calQ^{\rm el}_{0}(e(t)) + \calQ^{\rm hard}_{0}(p(t), \alpha(t)) + \sum_{i = 1}^{n} \calR_{0}\!\left(p(t_{i+1}) - p(t_i), \alpha(t_{i+1}) - \alpha(t_i)\right) - \big\langle \ell(t), u(t) \big\rangle
	\\*& \leq \liminf\limits_{h}\left( \calQ^{\rm el}_{h}(e^h(t)) + \calQ^{\rm hard}_{h}(p^h(t), \alpha^h(t)) + \sum_{i = 1}^{n} \calR_{h}\!\left(p^h(t_{i+1}) - p^h(t_i), \alpha^h(t_{i+1}) - \alpha^h(t_i)\right) - \big\langle \ell(t), u^h(t) \big\rangle \right)
	\\& \leq \liminf\limits_{h}\left( \calQ^{\rm el}_{h}(e^h(t)) + \calQ^{\rm hard}_{h}(p^h(t), \alpha^h(t)) + \calV_{\calR_{h}}(p^h, \alpha^h; 0, t) - \big\langle \ell(t), u^h(t) \big\rangle \right)
	\\& = \liminf\limits_{h}\Bigg( 
	\calQ^{\rm el}_{h}(e^h(0)) + \calQ^{\rm hard}_{h}(p^h(0), \alpha^h(0)) 
	- \big\langle \ell(0), u^h(0) \big\rangle
	\\*& \quad\quad\quad\quad\quad + \int_0^t \left( \int_{\Omega} \C(x') \Lambda_h e^h(s) : {\E}\dot{w}(s) \,dx 
	- \big\langle \ell(s), \dot{w}(s) \big\rangle - \big\langle \dot{\ell}(s), u^h(s) \big\rangle \right) \,ds
	\Bigg),
\end{align*}
}\noindent
where the last inequality follows from \ref{h-qs E}.
In view of strong convergence assumed in \eqref{main result u^h_0 condition}--\eqref{main result alpha^h_0 condition}, by the Lebesgue's dominated convergence theorem we get
\begin{align*}
	& \lim\limits_{h}\left( \calQ^{\rm el}_{h}(e^h(0)) + \calQ^{\rm hard}_{h}(p^h(0), \alpha^h(0)) - \big\langle \ell(0), u^h(0) \big\rangle + \int_0^t \left( \int_{\Omega} \C(x') \Lambda_h e^h(s) : {\E}\dot{w}(s) \,dx - \big\langle \dot{\ell}(s), u^h(s) \big\rangle \right) \right)\\
	& = \calQ^{\rm el}_{0}(e(0)) + \calQ^{\rm hard}_{0}(p(0), \alpha(0)) - \big\langle \ell(0), u(0) \big\rangle + \int_0^t \left( \int_{\Omega} \C(x') \M_{x'} e(s) : {\E}\dot{w}(s) \,dx - \big\langle \dot{\ell}(s), u(s) \big\rangle \right).
\end{align*}
Hence, we have
\begin{align*}
	& \calQ^{\rm el}_{0}(e(t)) + \calQ^{\rm hard}_{0}(p(t), \alpha(t)) 
	+ \sum_{i = 1}^{n} \calR_{0}\!\left(p(t_{i+1}) - p(t_i), \alpha(t_{i+1}) - \alpha(t_i)\right) 
	- \big\langle \ell(t), u(t) \big\rangle
	\\& \leq \calQ^{\rm el}_{0}(e(0)) + \calQ^{\rm hard}_{0}(p(0), \alpha(0)) 
	- \big\langle \ell(0), u(0) \big\rangle 
	\\& + \int_0^t \left( \int_{\Omega} \red{\C}(x') e(s) : {\E}\dot{w}(s) \,dx - \big\langle \ell(s), \dot{w}(s) \big\rangle - \big\langle \dot{\ell}(s), u(s) \big\rangle \right) \,ds.
\end{align*}
Taking the supremum over all partitions of $[0,t]$ yields \eqref{main result dr - step 3 inequality}, which concludes the proof.
\end{proof}

\begin{remark} \label{remark gamma convergence} 
By using {\RED the\BLACK} lower semicontinuity given in \Cref{lower semicontinuity of energies - h to 0} and {\RED the\BLACK} construction of {\RED a\BLACK} recovery sequence given in \Cref{recovery sequence}{\RED,\BLACK} it is not difficult to see that the sequence of functionals given by 
\[
	\mathcal{F}_h(u^h, e^h, p^h, \alpha^h)=\calQ^{\rm el}_{h}(e^h) + \calQ^{\rm hard}_{h}(p^h, \alpha^h) +\calR_{h}(p^h, \alpha^h)- \big\langle \ell, u^h \big\rangle, \quad (u^h, e^h, p^h, \alpha^h) \in \calA_{h}^{\rm hard}( w),
\]
for some $w \in H^1_{KL}(\Omega)$ and $\ell$ of the form \eqref{definition ell} 
$\Gamma$--converges to the functional 
\[
	\mathcal{F}_0 (u, e, p, \alpha)=	\calQ^{\rm el}_{0}(e) + \calQ^{\rm hard}_{0}(p, \alpha) + \calR_{0}(p, \alpha)-\big\langle \ell, u \big\rangle,\quad (u, e, p, \alpha) \in \calA_{0}^{\rm hard}(w).
\]
Moreover we have that if 
\[
	(u^h_{\rm min}, e^h_{\rm min}, p^h_{\rm min}, \alpha^h_{\rm min})=\displaystyle\argmin_{(\upsilon, \eta, \pi, \beta) \in \calA_{h}^{\rm hard}(w)}
	\mathcal{F}_h (\upsilon, \eta, \pi, \beta),
\]
then 
\begin{align}
	& u^h_{\rm min} \weak u^0_{\rm min} \quad \text{weakly in $H^1(\Omega;\R^3)$},\\
	& \Lambda_h e^h_{\rm min} \strong \M_{x'}e_{\rm min}^0 \quad \text{strongly in $L^2(\Omega;\M^{3 \times 3}_{\sym})$},\\
	& \Lambda_h p^h_{\rm min} \strong p_{\rm min}^0 \quad \text{strongly in $L^2(\Omega;\M^{3 \times 3}_{\dev})$},\\
	& \alpha^h_{\rm min} \strong \alpha_{\rm min}^0 \quad \text{strongly in $L^2(\Omega)$},
\end{align}
where 
\[
	(u^0_{\rm min}, e^0_{\rm min}, p^0_{\rm min}, \alpha^0_{\rm min})=\displaystyle\argmin_{(\upsilon, \eta, \pi, \beta) \in \calA_{0}^{\rm hard}(w)}
	\mathcal{F}_0 (\upsilon, \eta, \pi, \beta).
\]
In particular this implies, using \Cref{remark stable states}, that for any $w(0) \in H^1_{KL}(\Omega)$ there exists $(u^0, e^0, p^0, \alpha^0) \in \calA_{0}^{\rm hard}(w(0))$ that satisfies global stability{\RED,\BLACK} as well as $(u^h, e^h, p^h, \alpha^h) \in \calA_{h}^{\rm hard}(w(0))$ that also satisfy global stability such that convergences \eqref{main result u^h_0 condition}--\eqref{main result alpha^h_0 condition} are valid.
\end{remark}

The following corollary improves the statement of \Cref{main result 1}.

\begin{corollary} \label{improved convergence} 
Under the assumption of \Cref{main result 1}{\RED,\BLACK} the convergences \eqref{main result u^h(t)}--\eqref{main result alpha^h(t)} are strong and valid as $h \to 0$ (not only on the subsequence). Moreover, the limit quasistatic evolution has a unique solution. The solution of the limit quasistatic evolution is absolutely continuous in time with values in $H^1_{KL}(\Omega) \times L^2(\Omega;\M^{3 \times 3}_{\sym})\times L^2(\Omega;\M^{3 \times 3}_{\dev}) \times L^2(\Omega)$. 
\end{corollary}

\begin{proof}
Absolute continuity of the solution of the limit quasistatic evolution can be proven by following the proof \cite[Theorem 5.2]{DalMaso.DeSimone.Mora.2006} (cf. \Cref{improved convergence 1}). By using \cite[Theorem 3.5.2]{Mielke.Roubicek.2015}) we conclude that the solution of the limit quasistatic evolution is unique.
Since the limit quasistatic evolution has a unique we have that the convergences \eqref{main result u^h(t)}--\eqref{main result alpha^h(t)} are valid on the whole sequence. It remains to prove that they are strong. Note that from the computations done in Step 4 of \Cref{main result 1}, the energy equality \ref{0-qs E} and the proof of \Cref{lower semicontinuity of energies - h to 0} we conclude that for $t \in [0,T]$
\[
	\calQ^{\rm el}_{h}(e^h(t))\to \calQ^{\rm el}_{0}(e(t)), \qquad \calQ^{\rm hard}_{h}(p^h(t), \alpha^h(t)) \to \calQ^{\rm hard}_{0}(p(t), \alpha(t)).
\]
From weak convergences \eqref{main result e^h(t)}, \eqref{main result p^h(t)}, \eqref{main result alpha^h(t)}, using the strict convexity of the above quadratic forms we conclude the strong convergences in \eqref{main result e^h(t)}, \eqref{main result p^h(t)}, \eqref{main result alpha^h(t)}. 
Since $\Lambda_h {\E}u^h =\Lambda_h e^h+\Lambda_h p^h$ we have that ${\E}u^h$ converges strongly to ${\E}u$. 
By using Korn's inequality we conclude the strong convergence in \eqref{main result u^h(t)}.
\end{proof}

\subsection{Admissible stress configurations} \label{Admissible stress configurations}

In \Cref{limit of quasistatic evolutions - eps to 0}, instead of using \Cref{stability condition equivalence} and \Cref{recovery sequence} for proving global stability of the limit problem, we will follow a stress-strain approach which was introduced in \cite{Francfort.Giacomini.2014} and further explored in \cite{Buzancic.Davoli.Velcic.2024.1st} and \cite{Buzancic.Davoli.Velcic.2024.2nd}.
To that end, we need to derive the balance equations of the limiting quasistatic evolution in terms of the corresponding stresses, which will be done in this section.

For every $e^h \in L^2(\Omega;\M^{3 \times 3}_{\sym})$, $p^h \in L^2(\Omega;\M^{3 \times 3}_{\dev})$ and $\alpha^h \in L^2(\Omega)$, $\alpha^h \geq 0$, we denote 
\begin{equation*}
	\sigma^h(x) := \C(x') \Lambda_h e^h(x),
	\quad \chikin^h(x) := - \Hkin(x') \Lambda_h p^h(x)
	\;\text{ and }\; \chiiso^h(x) := - \Hiso(x') \alpha^h(x).
\end{equation*}
Then we introduce
\begin{align*}
	\calK_{h}(f,g) := \bigg\{& (\sigma^h,\, \chikin^h,\, \chiiso^h) \in L^2(\Omega;\M^{3 \times 3}_{\sym}) \times L^2(\Omega;\M^{3 \times 3}_{\dev}) \times L^2(\Omega;(-\infty,0]) : 
	\\& -\div_{h}\sigma^h = \scaling_h\fvol \text{ in } \Omega,\quad \sigma^h\,\nu^h = \scaling_h\fsurf \text{ on } \partial{\Omega} \setminus {\Gamma_\Dir},
	\\& \sigma^h_{\dev}(x',x_3)+\chikin^h(x',x_3) \in \left(1-\chiiso^h(x',x_3)\right) K(x') \,\text{ for a.e. } x' \in \omega,\, x_3 \in I\bigg\}.
\end{align*}
Here $\nu^h = \scaling_{\frac{1}{h}}\nu$. 
Note that from the conditions 
$-\div_{h}\sigma^h = \scaling_h\fvol \text{ in } \Omega$ and $\sigma^h\,\nu^h = \scaling_h\fsurf \text{ in } \partial{\Omega} \setminus {\Gamma_\Dir}$,
for every $\varphi \in H^1(\Omega;\R^3)$ with $\varphi = 0$ on $\Gamma_\Dir$ we have
\begin{equation} \label{div sigma^h = 0}
	\int_{\Omega} \sigma^h(x) : \Lambda_h {\E}\varphi(x) \,dx 
	= \int_{\Omega} f(x) \cdot \varphi(x) \,dx 
	+ \int_{\partial{\Omega} \setminus {\Gamma_\Dir}} \fsurf(x) \cdot \varphi(x) \,d\calH^{2}.
\end{equation}
If we consider the weak limit $(\sigma, \chikin, \chiiso) \in L^2(\Omega;\M^{3 \times 3}_{\sym})$ of the sequence $(\sigma^h, \chikin^h, \chiiso^h) \in \calK_{h}(f,g)$ as $h \to 0$, then we can conclude $\sigma_{i3} = 0$ for $i=1,2,3$ form \eqref{div sigma^h = 0} by arguing as in Step 3 of the proof of \Cref{main result 1}.
Secondly, using Mazur's theorem and convexity we easily conclude that 
\begin{equation}
	\sigma_{\dev}(x',x_3)+\chikin(x',x_3) \in \left(1-\chiiso(x',x_3)\right) K(x') \,\text{ for a.e. } x' \in \omega,\, x_3 \in I. 
\end{equation} 
Lastly, let $\vartheta \in C^{\infty}(\closure{\omega};\R^3)$ such that $\vartheta=\nabla \vartheta=0$ on $\gamma_D$.
If we choose the function
\begin{equation*}
	\varphi(x) = \left(\begin{array}{c} \vartheta_1(x') - x_3\,\partial_{x_1}\vartheta_3(x') \\ \vartheta_2(x') - x_3\,\partial_{x_2}\vartheta_3(x') \\ \vartheta_3(x') \end{array}\right),
\end{equation*}
we deduce from \eqref{div sigma^h = 0} that
\begin{equation*}
\begin{split}
	\int_{\Omega} \sigma^h(x) : \begin{pmatrix} {\E}_{x'}\vartheta^{1,2}(x') - x_3 {\D}^2_{x'}\vartheta_3(x') & 0 \\ 0 & 0 \end{pmatrix} \,dx 
	=& \int_{\Omega} 
	\begin{pmatrix}
		\fvol^{1,2}(x) \\[3pt]
		\fvol_{3}(x)
	\end{pmatrix}
	\cdot 
	\begin{pmatrix}
		\vartheta^{1,2}(x') - x_3 \nabla_{x'}\vartheta_3(x') \\[3pt]
		\vartheta_3(x')
	\end{pmatrix}
	\,dx 
	\\& + \int_{\partial{\Omega} \setminus \Gamma_\Dir} \begin{pmatrix}
		\fsurf^{1,2}(x) \\[3pt]
		\fsurf_{3}(x)
	\end{pmatrix}
	\cdot 
	\begin{pmatrix}
		\vartheta^{1,2}(x') - x_3 \nabla_{x'}\vartheta_3(x') \\[3pt]
		\vartheta_3(x')
	\end{pmatrix} \,d\calH^{1}
	.
\end{split}
\end{equation*}
Passing to the limit,
and dividing by the effect of test functions $\vartheta^{1,2}$ and $\vartheta_3$, the equation above becomes
\begin{equation*}
\begin{split}
	\int_{\omega} \zerothI{\sigma}(x') : {\E}_{x'}\vartheta^{1,2}(x') \,dx' 
	=& \int_{\omega} \zerothI{\fvol^{1,2}}(x') \cdot \vartheta^{1,2}(x') \,dx' 
	\\& + \int_{\partial{\omega} \setminus \gamma_\Dir} \zerothI{\fsurf^{1,2}}(x') \cdot \vartheta^{1,2}(x') \,d\calH^{{\RED2\BLACK}}
	\\& + \int_{\omega \times \partial{I}} \fsurf^{1,2}(x) \cdot \vartheta^{1,2}(x') \,dx'
	,
\end{split}
\end{equation*}
\begin{equation*}
\begin{split}
	-\frac{1}{12} \int_{\omega} \firstI{\sigma}(x{\RED'\BLACK}) : {\D}^2_{x'}\vartheta_3(x') \,dx' 
	=& -\frac{1}{12} \int_{\omega} \firstI{\fvol^{1,2}}(x') \cdot \nabla_{x'}\vartheta_3(x') \,dx'
	+ \int_{\omega} \zerothI{\fvol_{3}}(x') \, \vartheta_3(x') \,dx'
	\\& -\frac{1}{12} \int_{\partial{\omega} \setminus \gamma_\Dir} \firstI{\fsurf^{1,2}}(x') \cdot \nabla_{x'}\vartheta_3(x') \,d\calH^{1}
	+ \int_{\partial{\omega} \setminus \gamma_\Dir} \zerothI{\fsurf_{3}}(x') \, \vartheta_3(x') \,d\calH^{1}
	\\& - \int_{\omega \times \partial{I}} x_3 \fsurf^{1,2}(x) \cdot \nabla_{x'}\vartheta_3(x') \,dx'
	+ \int_{\omega \times \partial{I}} \fsurf_{3}(x) \, \vartheta_3(x') \,dx'
	.
\end{split}
\end{equation*}
Using the integration by parts formulas,
as well as decomposing the gradient $\nabla_{x'}$ into the normal derivative $\partial_{\nu}$ and tangential derivative $\partial_{\tau}$,
we get that
\begin{equation} \label{stress membrane equation} 
\begin{split}
	-\div_{x'}\zerothI{\sigma} & = f^{\rm memb} 
	\quad \text{ in } \omega,
	\\ \zerothI{\sigma}\,\nu & = g^{\rm memb} \quad \text{ on } \partial{\omega} \setminus \gamma_\Dir,
\end{split}
\end{equation}
and
\begin{equation}\label{stress bending equation} 
\begin{split}
	-\frac{1}{12} \div_{x'}\div_{x'}\firstI{\sigma} & = f^{\rm bend} \quad \text{ in } \omega, 
	\\ \frac{1}{12} \div_{x'}\firstI{\sigma}\cdot\nu + \frac{1}{12} \partial_{\tau}(\firstI{\sigma}\,\nu\cdot\tau) & = g^{\rm bend}_0 \quad \text{ on } \partial{\omega} \setminus \gamma_\Dir,
	\\ - \firstI{\sigma}\,\nu\cdot\nu & = g^{\rm bend}_1 \quad \text{ on } \partial{\omega} \setminus \gamma_\Dir,
\end{split}
\end{equation}
where
\begin{equation} \label{limit loads expresions}
\begin{split}
	f^{\rm memb} \,:=\ & \zerothI{\fvol^{1,2}} 
	+ \fsurf^{1,2}(\cdot,\tfrac{1}{2}) + \fsurf^{1,2}(\cdot,-\tfrac{1}{2}), 
	\\ g^{\rm memb} \,:=\ & \zerothI{\fsurf^{1,2}}, 
	\\ f^{\rm bend} \,:=\ & \zerothI{\fvol_{3}} + \frac{1}{12} \div_{x'}\firstI{\fvol^{1,2}} 
	+ \fsurf_{3}(\cdot,\tfrac{1}{2}) + \fsurf_{3}(\cdot,-\tfrac{1}{2}) + {\RED\frac{1}{2}\BLACK} \div_{x'}\fsurf^{1,2}(\cdot,\tfrac{1}{2}) - {\RED\frac{1}{2}\BLACK} \div_{x'}\fsurf^{1,2}(\cdot,-\tfrac{1}{2}), 
	\\ g^{\rm bend}_0 \,:=\ & - \frac{1}{12} \firstI{\fvol^{1,2}}\cdot\nu + \zerothI{\fsurf_{3}} - {\RED\frac{1}{2}\BLACK} \fsurf^{1,2}(\cdot,\tfrac{1}{2})\cdot\nu + {\RED\frac{1}{2}\BLACK} \fsurf^{1,2}(\cdot,-\tfrac{1}{2})\cdot\nu + \frac{1}{12} \partial_{\tau}(\firstI{\fsurf^{1,2}}\cdot\tau),
	\\ g^{\rm bend}_1 \,:=\ & - \firstI{\fsurf^{1,2}}\cdot\nu.
\end{split}
\end{equation}
\begin{remark} 
{\PINK It is not surprising that the derivatives of the loads (and their boundary terms) appear on the right-hand side of the strong formulation.
This is a direct consequence of the Kirchhoff--Love ansatz and also appears in the linear elastic plate theory (see, e.g., \cite{ciarlet1997mathematical}).
The alternative to assuming additional regularity of the loads would simply be to interpret the terms on the right-hand side of \eqref{stress membrane equation} and  \eqref{stress bending equation} in the weak sense (i.e. as functionals), thereby making the right-hand sides of the strong and weak formulations equivalent. \BLACK}
\end{remark} 
\begin{remark} 
\label{remark comparison} 
We compare our limit model with the one obtained in \cite{Liero.Roche.2012}. 
There, the authors did not consider a heterogeneous plate, nor did they include isotropic hardening. 
Instead, they used a special form of the hardening and dissipation potentials, which enabled them to conclude that, in the limiting quasistatic evolution, {\PINK the plastic strain is such that \BLACK}$p_{\alpha 3} = 0$ for $\alpha = 1, 2$. 
More precisely, {\PINK in their work\BLACK} they assume that $\Hkin(x') \xi \cdot \xi=\frac{k_{\rm hard}}{2} |\xi|^2$ for every $\xi \in \M^{3 \times 3}_{\sym}${\RED,\BLACK} $x' \in \omega$ {\RED and\BLACK} for some $k_{\rm hard}>0$, $\Hiso(x') \equiv 0${\RED,\BLACK} and $\disspot(x',p)=\sigma_{\rm yield} |p|$ for every $p \in \M^{3 \times 3}_{\dev}$, $x' \in \omega$ and some $\sigma_{\rm yield}>0$.
\end{remark}

\begin{remark}
From \cite[Section 3.5]{Mielke.Roubicek.2015} it is easy to see that the limit quasistatic evolution implies the following system of equations:
\begin{align}
	-\div_{x'}\Big(\red{\C}(x')\left({\E}_{x'}\bar{u}-\zerothI{p}^{1,2}\right)\Big) 
	= f^{\rm memb}(t,x')
	\quad \text{ on } \omega, \label{dimension reduction equation 1}\\
	\frac{1}{12} \div_{x'}\div_{x'}\Big(\red{\C}(x')\left({\D}^2_{x'}u_3+\firstI{p}^{1,2}\right)\Big)
	= f^{\rm bend}(t,x')
	\quad \text{ on } \omega, \label{dimension reduction equation 2} 
\end{align}
\begin{equation} \label{dimension reduction equation 3}
	\begin{pmatrix}
		0 \\ 0
	\end{pmatrix}
	\in
	\partial\gendisspot(x', \dot{p}, \dot{\alpha})
	+ 
	\begin{pmatrix}
		\left[ \red{\C}(x')(p^{1,2}-{\E}_{x'}\bar{u}+x_3{\D}^2_{x'}u_3) \right]_{\dev} + \Hkin(x')p\\
		\Hiso(x')\alpha
	\end{pmatrix}
	\quad \text{ on } \Omega,
\end{equation}
which is a generalized version of the second-order membrane equation, fourth-order Kirchhoff's plate equation, and plastic flow law. 
{\PINK Here $f^{\rm memb}$, $f^{\rm bend}$ are defined as in \eqref{limit loads expresions}.\BLACK}
Note that the equations \eqref{dimension reduction equation 1} and \eqref{dimension reduction equation 2} are actually the first equations of \eqref{stress membrane equation} and \eqref{stress bending equation}. 
It can be shown that if $t \mapsto (u(t), p(t), \alpha(t))$ is absolutely continuous with values in $H^1_{KL}(\Omega) \times L^2(\Omega;\M^{3 \times 3}_{\dev}) \times L^2(\Omega)$ and we add to the system \eqref{dimension reduction equation 1}--\eqref{dimension reduction equation 3} boundary conditions given in the expressions \eqref{stress membrane equation}, \eqref{stress bending equation} as well as the Dirichlet boundary condition $\bar{u}(t,\cdot)=\bar{w}(t,\cdot)$, $ u_3=w_3$, $\nabla_{x'} u_3=\nabla_{x'} w_3$ on $\gamma_D$, then the system \eqref{dimension reduction equation 1}--\eqref{dimension reduction equation 3} is equivalent to \ref{0-qs S} and \ref{0-qs E}.
\end{remark}

\begin{remark} 
It is not difficult to generalize the results given in this section to the case when the elasticity tensor, hardening tensor and dissipation potential also depend on the $x_3$ variable. 
We avoided this in order to simplify the exposition.
\end{remark} 

\section{Simultaneous homogenization and vanishing of hardening} \label{limit of quasistatic evolutions - eps to 0}

In this section we want to obtain the limit model when the material changes on the fine scale $\eps \to 0$ letting at the same time the hardening tensors tend to zero (see below for the details). 
In this way in the limit we obtain the model of the homogenized plate in the context of perfect plasticity. 
Unlike in the case of hardening, the plastic strain belongs to the space of Radon measures. 
In this case one needs to define the dissipation potential at each interface, since the measure can concentrate at the interface. 
The limit model we justify by the convergence result. 
In order to obtain the convergence result one needs to define the limit model, prove the compactness statements with respect to two-scale convergence and prove that the limit satisfies the global stability and energy equality. 
In order to obtain global stability, we will use the approach introduced in \cite{Francfort.Giacomini.2014} and later explored in \cite{Buzancic.Davoli.Velcic.2024.1st,Buzancic.Davoli.Velcic.2024.2nd}. 
It consists in defining the set of limit two-scale stresses and then proving the lower bound for the two-scale dissipation functional (see \Cref{two-scale Hill's principle - regime zero} and \Cref{two-scale dissipation and plastic work inequality}). 
This immediately implies the global stability for the limit and then the energy equality is proved through global stability and lower semicontinuity. 
The main result in this section is the convergence statement given in \Cref{main result 2} and improved convergence given in \Cref{cor improved 2}. 

Here we work under the assumption that on $\gamma_\Dir=\partial \omega$ and that there are no applied loads (they are zero), i.e. the evolution is only driven by time-dependent boundary conditions.
As in \eqref{assbd} we assume that $w \in W^{1,1}(0,T;H^1_{KL}(\Omega))$ is given. 

The analysis presented here can also be {\RED extended\BLACK} to obtain the homogenized plate model with hardening (i.e. when the hardening tensors are not set to zero). 
{\PINK This model is, in fact, simpler to analyze, as it does not require the definition of a dissipation functional at the phase interfaces.\BLACK}

Next we give the outline of this section. 
In \Cref{section phase decomposition} we give the basic assumptions on the material and associated tensors. 
In \Cref{A priori estimates with vanishing hardening} we prove the a priori estimates that are used in \Cref{section two-scale limits} to obtain the necessary compactness statements with respect to two-scale convergence. 
In \Cref{semicontinuity of the dissipation functional}, relying on the compactness result we obtained, we define the homogenized elastic energy and dissipation functional and prove lower semicontinuity result. In \Cref{section two-scale limit stresses} we characterize two-scale limit stresses, define the stress-strain two-scale duality and prove lower bound for the dissipation functional.
Finally, in \Cref{section two-scale quasistatic evolution} we define the limit two-scale quasistatic evolution and prove the convergence result as well as the regularity in time of the limit. 

\subsection{Phase decomposition} \label{section phase decomposition} 

Here we give basic assumptions on the material and consequently on the tensors. 
We also recall here some basic notation and assumptions from \cite{Francfort.Giacomini.2014}. 

Recall that $\calY = \R^2/\Z^2$ is the $2$-dimensional torus, let $Y := [0, 1)^2$ be its associated periodicity cell, and denote by $\calI : \calY \to Y$ their canonical identification.
We denote by $\calC$ the set 
\begin{equation} \label{torus boundary}
	\calC := \calI^{-1}(\partial Y).
\end{equation}
Using this canonical identification, we can consider
\begin{equation} \label{torus without the boundary}
	\mathring{Y} := (0 ,1)^2
\end{equation}
as a subset of $\calY$ (i.e., $\mathring{Y} \equiv \calY \setminus \calC$).

For any $\calZ \subset \calY$, we denote
\begin{equation} \label{periodic set notation}
	\calZ_\eps := \left\{ x \in \R^2 : \frac{x}{\eps} \in \Z^2+\calI(\calZ) \right\},
\end{equation}
and for any function $F : \calY \to X$ we associate the $\eps$-periodic function $F_\eps : \R^2 \to X$, given by
\begin{equation*}
	F_\eps(x) := F\left(y_\eps\right), \;\text{ for }\; \frac{x}{\eps}-\left\lfloor \frac{x}{\eps} \right\rfloor = \calI(y_\eps) \in Y.
\end{equation*}
With a slight abuse of notation we will also write $F_\eps(x) = F\left(\frac{x}{\eps}\right)$.

The torus $\calY$ is assumed to be made up of finitely many phases $\calY_i$ together with their interfaces. 
We assume that those phases are pairwise disjoint open sets with Lipschitz boundary.
Then we have $\calY = \bigcup_{i} \closure{\calY}_i$ and we denote the interfaces by 
\begin{equation*}
	\Gamma := \bigcup_{i,j} \partial\calY_i \cap \partial\calY_j.
\end{equation*}
We assume that there exists a compact set $S \subset \Gamma$ with $\calH^{1}(S) = 0$ such that
\begin{equation} \label{interface regularity condition}
	\Gamma \setminus S \quad \text{ is a $C^2$-hypersurface.}
\end{equation}
Furthermore, the interfaces are assumed to have a negligible intersection with the set $\calC$, i.e. for every $i$
\begin{equation} \label{transversality condition}
	\calH^{1}(\partial\calY_i \cap \calC) = 0.
\end{equation}
We will write
\begin{equation*}
	\Gamma := \bigcup_{i \neq j} \Gamma_{ij},
\end{equation*}
where $\Gamma_{ij}$ stands for the interface between $\calY_i$ and $\calY_j$.

We assume that $\omega$ is composed of the finitely many phases $(\calY_i)_\eps$. 
Note that, provided that $\eps>0$ is chosen such that $\calH^1(\cup_i(\partial \calY_i)_{\eps} \cap \partial \omega)=0$, then each point of $\Gamma_\Dir$ outside $\calH^2$ negligible set belongs to a well defined phase. 
Therefore $\Omega \cup \Gamma_\Dir$ is a geometrically admissible multi-phase domain in the sense of \cite[Subsection 1.2]{Francfort.Giacomini.2012}. 
Only these $\eps$'s will be considered from this point on. 

We say that a multi-phase torus $\calY$ is \emph{geometrically admissible} if it satisfies the above assumptions.

\medskip
\noindent{\bf Periodic elastic and plastic properties.} \nopagebreak

For every $i$, let $(\red{\C})_i$ and $(\Hkin)_i$ be a symmetric positive definite tensor on $\M^{2 \times 2}_{\sym}$ and $\M^{3 \times 3}_{\dev}$, respectively, and let $(\Hiso)_i$ be positive constant, such that there exist 
constants $r_{\C}$ and $R_{\C}$, with $0 < r_{\C} \leq R_{\C}$, 
as well as constants $r_{\Hkin}$, $r_{\Hiso}$, $R_{\Hkin}$ and $R_{\Hiso}$, with $0 \leq r_{\Hkin} \leq R_{\Hkin}$, $0 \leq r_{\Hiso} \leq R_{\Hiso}$ and $r_{\Hkin} + r_{\Hiso} > 0$,
satisfying
\begin{align} 
	& r_{\C} |\xi|^2 \leq (\red{\C})_i \xi : \xi \leq R_{\C} |\xi|^2 \quad \text{ for every }\xi \in \M^{2 \times 2}_{\sym}, \label{tensorassumption1}
	\\ & r_{\Hkin} |\zeta|^2 \leq (\Hkin)_i \zeta : \zeta \leq R_{\Hkin} |\zeta|^2 \quad \text{ for every }\zeta \in \M^{3 \times 3}_{\dev}, \label{tensorassumption2}
	\\ & r_{\Hiso} \leq (\Hiso)_i \leq R_{\Hiso}. \label{tensorassumption3}
\end{align}
Now let $\red{\C}$, $\Hkin$ and $\Hiso$ be considered as maps from $\calY$ taking corresponding values on each phase $\calY_i$,
i.e. $\red{\C}(y) := (\red{\C})_i$, $\Hkin(y) := (\Hkin)_i$ and $\Hiso(y) := (\Hiso)_i$ for every $y \in \calY_i$.

\medskip

We also assume there exist convex compact sets $K_i \in \M^{3 \times 3}_{\dev}$ for each phase $\calY_i$.
We further assume there exist two constants $r_K$ and $R_K$, with $0 < r_K \leq R_K$, such that for every $i$
\begin{equation*}
	\{ \xi \in \M^{3 \times 3}_{\dev} : |\xi| \leq r_K \} \subseteq K_i \subseteq\{ \xi \in \M^{3 \times 3}_{\dev}: |\xi| \leq R_K \}.
\end{equation*}
Finally, we define
\begin{equation*}
	K(y) := K_i, \quad \text{ for } y \in \calY_i.
\end{equation*}

For each $i$, let $\disspot_i : \M^{3 \times 3}_{\dev} \to [0,+\infty)$ be the support function of the set $K_i$, i.e.
\begin{equation*}
	\disspot_i(\xi) = \sup\limits_{\tau \in K_i} \tau : \xi.
\end{equation*}
It follows that $\disspot_i$ is convex, positively 1-homogeneous, and satisfies
\begin{equation} \label{coercivity of R_i}
	r_K |\xi| \leq \disspot_i(\xi) \leq R_K |\xi| \quad \text{for every}\, \xi \in \M^{3 \times 3}_{\dev}.
\end{equation}
Then we define the dissipation potential to be
\begin{equation*}
	\disspot(y, \xi) := \disspot_i(\xi),
\end{equation*}
for every $y \in \calY_i$.

\begin{remark} \label{remark on Reshetnyak theorem conditions - new disspot}
We point out that $\disspot$ is a Carath\'{e}odory function on $\cup_i\calY_i \times \M^{3 \times 3}_{\dev}$. 
Furthermore, for each $y \in \calY$, the function $\xi \mapsto \disspot(y, \xi)$ is positively 1-homogeneous and convex.
\end{remark}


\medskip
\noindent{\bf The reduced dissipation potential.} \nopagebreak

The set $\red{K}(y) \subset \M^{2 \times 2}_{\sym}$ represents the set of admissible stresses in the reduced problem and can be {\RED characterized \BLACK}as follows (see \cite[Section 3.2]{Davoli.Mora.2013}):
\begin{equation} \label{K_r characherization}
	\sigma \in \red{K}(y)
	\quad \iff \quad 
	\begin{pmatrix}
	\sigma_{11} & \sigma_{12} & 0 \\
	\sigma_{12} & \sigma_{22} & 0 \\
	0 & 0 & 0
	\end{pmatrix}
	- \frac{1}{3}(\tr \sigma)I_{3 \times 3} \in K(y),
\end{equation}
where $I_{3 \times 3}$ is the identity matrix in $\M^{3 \times 3}$.

The \emph{reduced perfectly-plastic dissipation potential} $\red{\disspot} : \cup_i\calY_i \times \M^{2 \times 2}_{\sym} \to [0,+\infty)$ is given by the support function of $\red{K}(y)$, i.e.
\begin{equation*}
	\red{\disspot}(y, p) := \sup\limits_{\sigma \in \red{K}(y)}\, \sigma : p \quad \text{for every}\, p \in \M^{2 \times 2}_{\sym}.
\end{equation*}
It follows that $\red{\disspot}(y,\cdot)$ is convex and positively 1-homogeneous, and there are two constants $0 < r_H \leq R_H$ such that 
\begin{equation*}
	r_H |p| \leq \red{\disspot}(y, p) \leq R_H |p| \quad \text{for every}\, p \in \M^{2 \times 2}_{\sym}.
\end{equation*}
Therefore $\red{\disspot}(y,\cdot)$ satisfies the triangle inequality 
\begin{equation*}
\red{\disspot}(y, p_1+p_2) \leq \red{\disspot}(y, p_1) + \red{\disspot}(y, p_2) \quad \text{for every}\, p_1,\, p_2 \in \M^{2 \times 2}_{\sym}.
\end{equation*}
For a fixed $y \in \cup_i\calY_i$, we can deduce the property
\begin{equation*}
	\red{K}(y) = \partial \red{\disspot}(y, 0).
\end{equation*}
{\RED 
By $\red{\disspot}_i$ and $\red{K}_i$ we denote 
$\red{\disspot}$ and $\red{K}$ restricted on $\calY_i$, respectively.
From \eqref{K_r characherization} we also have that (see again \cite[Section 3.2]{Davoli.Mora.2013})
\begin{equation} \label{reference1}
    \red{\disspot}_i(p) = \min_{\eta_1,\eta_2 \in \R} \disspot_i 
    \begin{pmatrix}
        p_{11} & p_{12} & \eta_1 \\
        p_{12} & p_{22} & \eta_2 \\
        \eta_1 & \eta_2 & -(p_{11}+p_{22})
    \end{pmatrix}.
\end{equation}\BLACK}\noindent
Since our limit plastic strain will be a measure we will need to define the plastic dissipation at the interface also. We will do this in \Cref{semicontinuity of the dissipation functional}. 

\subsection{A priori estimates for limiting admissible configurations with periodic microstructure} \label{A priori estimates with vanishing hardening}

In this section we recall the definition of the $\eps$-limiting quasistatic evolution in the context of vanishing hardening (see \Cref{limiting quasistatic evolution}) and prove the apriori estimates for the displacement and strains. 
Introducing periodic microstructure for the limiting quasistatic evolution derived in \Cref{limit of quasistatic evolutions - h to 0}, for $\eps > 0$ we consider the class
\begin{align*}
	\calA_{0}^{\rm hard,\eps}(w) := \Big\{ 
	& (u^\eps, e^\eps, p^\eps, \alpha^\eps) \in BD_{KL}(\Omega;\R^3) \times L^2(\Omega;\M^{2 \times 2}_{\sym}) \times L^2(\Omega;\M^{3 \times 3}_{\dev}) \times L^2(\Omega) :
	\\& {\E}_{x'}\bar{u}^\eps - x_3 {\D}^2_{x'}u_3^\eps = e^\eps + (p^\eps)^{1,2} \quad \text{ in } \Omega,
	\\& u^\eps = w \quad \calH^{2}\text{-a.e.}\text{ on } \Gamma_\Dir,
	\\& \disspot\!\left(\tfrac{x'}{\eps}, p^\eps(x)\right) \leq \alpha^\eps(x) \quad \text{a.e.} \text{ in } \Omega 
	\Big\},
\end{align*}
where $\bar{u}^\eps \in H^1(\omega;\R^2)$ and $u_3^\eps \in H^2(\omega)$ are the Kirchhoff--Love components of $u^\eps$.

For $(u^\eps, e^\eps, p^\eps, \alpha^\eps) \in \calA_{0}^{\rm hard,\eps}(w)$ we now define
\begin{align}
	\calQ^{\rm el,\eps}_{0}(e^\eps) & := \frac{1}{2} \int_{\Omega} \red{\C}\!\left(\tfrac{x'}{\eps}\right) e^\eps(x) : e^\eps(x) \,dx, \label{definition Q^el,eps_0}\\
	\calQ^{\rm hard,\eps}_{0}(p^\eps, \alpha^\eps) & := \frac{1}{2} \int_{\Omega} \Hkin\!\left(\tfrac{x'}{\eps}\right) p^\eps(x) : p^\eps(x) \,dx + \frac{1}{2} \int_{\Omega} \Hiso\!\left(\tfrac{x'}{\eps}\right) \alpha^\eps(x) \cdot \alpha^\eps(x) \,dx, \label{definition Q^hard,eps_0}
\end{align}
and
\begin{equation} \label{definition R^eps_0}
	\calR^{\eps}_{0}(p^\eps, \alpha^\eps) := \int_{\Omega} \gendisspot\!\left(\tfrac{x'}{\eps}, p^\eps(x), \alpha^\eps(x)\right) \,dx.
\end{equation}
As before we note that the condition $\disspot\!\left(\tfrac{x'}{\eps}, p^\eps(x)\right) \leq \alpha^\eps(x)$ a.e. in $\Omega$ implies that in fact $\calR^{\eps}_{0}(p^\eps, \alpha^\eps) = \int_{\Omega} \disspot\!\left(\tfrac{x'}{\eps}, p^\eps(x)\right) \,dx$.

The associated $\calR^{\eps}_{0}$-variation of a map 
$(p^\eps, \alpha^\eps) : [0, T] \to L^2(\Omega;\M^{3 \times 3}_{\dev}) \times L^2(\Omega)$ on $[a,b]$ is then defined as
\begin{equation*}
	\calV_{\calR^{\eps}_{0}}(p^\eps, \alpha^\eps; a, b) := \sup\left\{ \sum_{i = 1}^{n} \calR^{\eps}_{0}\!\left(p^\eps(t_{i+1}) - p^\eps(t_i), \alpha^\eps(t_{i+1}) - \alpha^\eps(t_i)\right) : a = t_1 < t_2 < \ldots < t_n = b,\ n \in \N \right\}.
\end{equation*}

Let $\delta(\eps) \to 0$ as $\eps \to 0$.
We now give the notion of the limiting quasistatic elasto-plastic evolution with periodic microstructure and vanishing hardening.

\begin{definition} \label{eps-limiting quasistatic evolution}
Let $\eps > 0$. 
A \emph{$\eps$-limiting quasistatic evolution} for the boundary datum $w(t)$ is a function $t \mapsto (u^\eps(t), e^\eps(t), p^\eps(t), \alpha^\eps(t))$ from $[0, T]$ into $H^1_{KL}(\Omega) \times L^2(\Omega;\M^{2 \times 2}_{\sym}) \times L^2(\Omega;\M^{3 \times 3}_{\dev}) \times L^2(\Omega)$ which satisfies the following conditions:
\begin{enumerate}[label=(qs\arabic*)$^{\eps}_{0}$]
	\item \label{eps-0-qs S} for every $t \in [0, T]$ we have $(u^\eps(t), e^\eps(t), p^\eps(t), \alpha^\eps(t)) \in \calA_{0}^{\rm hard,\eps}(w(t))$ and
	\begin{align*}
		\calQ^{\rm el,\eps}_{0}(e^\eps(t)) + \delta(\eps)\,\calQ^{\rm hard,\eps}_{0}(p^\eps(t), \alpha^\eps(t)) 
		\leq \calQ^{\rm el,\eps}_{0}(\eta) + \delta(\eps)\,\calQ^{\rm hard,\eps}_{0}(\pi, \beta) + \calR^{\eps}_{0}(\pi - p^\eps(t), \beta - \alpha^\eps(t)) 
		,
	\end{align*}
	for every $(\upsilon, \eta, \pi, \beta) \in \calA_{0}^{\rm hard,\eps}(w(t))$.
	\item \label{eps-0-qs E} 
	for every $t \in [0, T]$
	\begin{align*}
		& \calQ^{\rm el,\eps}_{0}(e^\eps(t)) + \delta(\eps)\,\calQ^{\rm hard,\eps}_{0}(p^\eps(t), \alpha^\eps(t)) 
		+ \calV_{\calR^{\eps}_{0}}(p^\eps, \alpha^\eps; 0, t) 
		\\& = \calQ^{\rm el,\eps}_{0}(e^\eps(0)) + \delta(\eps)\,\calQ^{\rm hard,\eps}_{0}(p^\eps(0), \alpha^\eps(0)) 
		+ \int_0^t \int_{\Omega} \red{\C}\!\left(\tfrac{x'}{\eps}\right) e^\eps(s) : {\E}\dot{w}(s) \,dx ds.
	\end{align*}
\end{enumerate}
\end{definition}

We recall that, as a consequence of \cite[Theorem 3.5.2]{Mielke.Roubicek.2015} we have that there exists a unique absolutely continuous solution $t \mapsto (u^{\eps}(t), e^{\eps}(t), p^{\eps}(t), \alpha^{\eps}(t)) $ of the problem given by \Cref{eps-limiting quasistatic evolution}. 
Next we give the a priori estimate. 
\begin{lemma} \label{lema apriori estimates} 
Let $t \mapsto (u^\eps(t), e^\eps(t), p^\eps(t), \alpha^\eps(t))$ for $t \in [0,T]$ be the solution according to \Cref{eps-limiting quasistatic evolution}. Then there exists $C>0$, independent of $\eps$ such that we have 
\begin{equation} \label{eps-boundness in time 1}
	\sup_{t \in [0,T]} \left\|e^\eps(t)\right\|_{L^2(\Omega;\M^{2 \times 2}_{\sym})} \leq C, \qquad \calV(p^\eps, 0, T) \leq C, \qquad \sup_{t \in [0,T]} \left\|u^\eps(t)\right\|_{BD(\Omega)}\leq C,
\end{equation}
\begin{equation} \label{eps-boundness in time 4}
	\sup_{t \in [0,T]} \left\|\sqrt{\delta(\eps)} p^\eps(t)\right\|_{L^2(\Omega;\M^{3 \times 3}_{\dev})} \leq C \;\text{ and }\; \sup_{t \in [0,T]} \left\|\sqrt{\delta(\eps)} \alpha^\eps(t)\right\|_{L^2(\Omega)} \leq C,
\end{equation}
\end{lemma}

\begin{proof} 
Firstly, we can prove that that there exists a constant $C$, depending only on the initial and boundary data, such that the first inequality in \eqref{eps-boundness in time 1} is valid as well as \eqref{eps-boundness in time 4} and 
\begin{equation} \label{eps-boundness in time 111}
	\calV_{\calR^{\eps}_{0}}(p^\eps, \alpha^\eps; 0, T) \leq C.
\end{equation}
Indeed, the energy balance of the $\eps$-limiting quasistatic evolution \ref{eps-0-qs E}, the estimate $\delta(\eps) \ll 1$ and coercivity of the quadratic energies $\calQ^{\rm el,\eps}_{0}$ and $\calQ^{\rm hard,\eps}_{0}$ imply
\begin{align*}
	& r_{\C} \left\|e^\eps(t)\right\|_{L^2(\Omega;\M^{2 \times 2}_{\sym})}^2 + \delta(\eps) r_{\Hkin} \left\|p^\eps(t)\right\|_{L^2(\Omega;\M^{3 \times 3}_{\dev})} + \delta(\eps) r_{\Hiso} \left\|\alpha^\eps(t)\right\|_{L^2(\Omega)} + \calV_{\calR^{\eps}_{0}}(p^\eps, \alpha^\eps; 0, t) \\
	& \leq R_{\C} \left\|e^\eps(0)\right\|_{L^2(\Omega;\M^{2 \times 2}_{\sym})}^2 + R_{\Hkin} \left\|p^\eps(0)\right\|_{L^2(\Omega;\M^{3 \times 3}_{\dev})}^2 + R_{\Hiso} \left\|\alpha^\eps(0)\right\|_{L^2(\Omega)}^2 
	\\& + 2 R_{\C} \sup_{t \in [0,T]} \left\|e^\eps(t)\right\|_{L^2(\Omega;\M^{2 \times 2}_{\sym})} \int_0^T \left\|{\E}\dot{w}(s)\right\|_{L^2(\Omega;\M^{3 \times 3}_{\sym})} \,ds.
\end{align*}
The claim now follows by the Cauchy-Schwarz inequality.
for every $t \in [0,T]$, which by the triangle inequality implies the second inequality in \eqref{eps-boundness in time 1}.
By using
\eqref{coercivity of R_i}
we can conclude that 
\begin{equation} \label{eps-boundness in time 3}
	\calV(p^\eps; 0, T) \leq C.
\end{equation}
	
Next, using a variant of Poincar\'{e}--Korn's inequality (see \cite[Chapter II, Proposition 2.4]{Temam.1985}) and the fact $(u^\eps(t), e^\eps(t), p^\eps(t), \alpha^\eps(t)) \in \calA_{0}^{\rm hard,\eps}(w(t))$, we can conclude that, for every $h > 0$ and $t \in [0,T]$,
\begin{align*}
	\left\|u^\eps(t)\right\|_{BD(\Omega)} & \leq C \left(\left\|u^\eps(t)\right\|_{L^1(\Gamma_\Dir;\R^3)} + \left\|{\E}u^\eps(t)\right\|_{L^1(\Omega;\M^{3 \times 3}_{\sym})}\right)\\
	& \leq C \left(\left\|w(t)\right\|_{L^1(\Gamma_\Dir;\R^3)} + \left\|e^\eps(t)\right\|_{L^2(\Omega;\M^{2 \times 2}_{\sym})} + \left\|(p^\eps(t))^{1,2}\right\|_{L^1(\Omega;\M^{2 \times 2}_{\sym})}\right)\\
	& \leq C \left(\left\|w(t)\right\|_{L^2(\Omega;\R^3)} + \left\|e^\eps(t)\right\|_{L^2(\Omega;\M^{2 \times 2}_{\sym})} + \left\|p^\eps(t)\right\|_{L^1(\Omega;\M^{3 \times 3}_{\dev})}\right).
\end{align*}
In view of the assumption on $w \in H^1(\Omega;\R^3)$ and from the first two inequalities in \eqref{eps-boundness in time 1} it follows that the sequences $\{u^\eps(t)\}$ are bounded in $BD(\Omega)$ uniformly with respect to $t$.
\end{proof}

Note that the boundedness of $u^{\eps}(t)$ in $BD(\Omega)$ is equivalent to the boundedness of $\bar{u}^{\eps}$ in $BD(\omega)$ and $u_3^{\eps}$ in $BH(\omega)$. 

\subsection{Two-scale limits and corrector spaces} \label{section two-scale limits} 

In this section we prove the basic compactness results with respect to two-scale convergence that follow from a priori estimates. 
The results obtained here motivate the definition of the reduced dissipation functional at the interface (see \eqref{definition R^hom_0}) and serve as a starting point for proving lower semicontinuity of the dissipation functional, which is done in \Cref{semicontinuity of the dissipation functional}. 
In \Cref{section two-scale convergence of measures} we recall the notion of two-scale convergence of measures. 
In \Cref{section two-scale limits and correctors} we recall the definition of corrector spaces for sequences with bounded symmetric gradients and bounded Hessians introduced in \cite[Section 4.1.2]{Buzancic.Davoli.Velcic.2024.2nd} and give the corresponding compactness statements.
We also introduce a new corrector space which will be used in \Cref{section unfolding for sequances}, where we introduce the unfolding for measures and prove new compactness results for sequences with bounded symmetric gradients and bounded Hessians. 
Unlike the compactness statements in \Cref{section two-scale limits and correctors} the ones obtained here contain the information how limit two-scale measures behave at the interface (more precisely at the boundary of the regular enough subset of $\calY$).

\subsubsection{Two-scale convergence of measures} \label{section two-scale convergence of measures}
We briefly recall some results and definitions from \cite[Section 4.1]{Francfort.Giacomini.2014}. 

\begin{definition}
\label{def:2-scale-meas}
Let $\omega \subset \R^2$ be an open set, $\Omega=\omega$ or $\Omega=\omega \times I$.
Let $\{\mu^\eps\}_{h>0}$ be a family in $\Mb(\Omega)$ and consider $\mu \in \Mb(\Omega \times \calY)$. 
We say that
\begin{equation*}
	\mu^\eps \weakstartwoscale \mu \quad \text{two-scale weakly* in }\Mb(\Omega \times \calY),
\end{equation*}
if for every $\chi \in C_0(\Omega \times \calY)$
\begin{equation*}
	\lim_{\eps \to 0} \int_{\Omega} \chi\left(x,\frac{x'}{\eps}\right) \,d\mu^\eps(x) = \int_{\Omega \times \calY} \chi(x,y) \,d\mu(x',y).
\end{equation*}
The convergence above is called \emph{two-scale weak* convergence}.
\end{definition}

\begin{remark} \label{transfertwoscale}
Notice that the family $\{\mu^\eps\}_{h>0}$ determines the family of measures $\{\nu^\eps\}_{h>0} \subset \Mb(\omega \times \calY)$ obtained by setting
\[
	\int_{\Omega \times \calY} \chi(x',y)\,d\nu^\eps(x',y) =\int_{\Omega} \chi \left(x',\frac{x'}{\eps}\right) \,d\mu^\eps (x)
\]
for every $\chi \in C_0(\Omega \times \calY)$. Thus $\mu$ is simply the weak* limit in $\Mb(\Omega \times \calY)$ of $\{\nu^\eps\}_{h>0}$. 
\end{remark} 

We collect some basic properties of two-scale convergence below:

\begin{proposition} \label{proposition properties two-scale convergence} 
The following statements hold:
\begin{enumerate}[label=(\roman*)]
	\item
	Any sequence that is bounded in $\Mb(\Omega)$ admits a two-scale weakly* convergent subsequence.
	\item 
	Let $\calD \subset \calY$ and assume that 
	$\supp(\mu^\eps) \subset \Omega \cap \calD_\eps$.
	If $\mu^\eps \weakstartwoscale \mu$ two-scale weakly* in $\Mb(\Omega \times \calY)$, then $\supp(\mu) \subset \Omega \times \closure{\calD}$.
\end{enumerate}
\end{proposition}

\subsubsection{Two-scale limit results for symmetrized gradients and Hessians} \label{section two-scale limits and correctors}
Here we will recall the corrector spaces 
related to spaces of measures which arise as two-scale limits of symmetrized gradients and Hessians in $BD$ and $BH$ spaces, respectively. 

Set
\begin{align*}
	\calXzero{\omega} := \Big\{\mu \in \Mb(\omega \times \calY;\R^2) : {\E}_{y}\mu \in \Mb(\omega \times \calY;\M^{2 \times 2}_{\sym}),& \\
	\mu(F \times \calY) = 0 \textrm{ for every Borel set } F \subseteq \omega & \Big\},
\end{align*}
\begin{align*}
	\calYzero{\omega} := \Big\{\kappa \in \Mb(\omega \times \calY) : {\D}^2_{y}\kappa \in \Mb(\omega \times \calY;\M^{2 \times 2}_{\sym}),& \\
	\kappa(F \times \calY) = 0 \textrm{ for every Borel set } F \subseteq \omega & \Big\}{\RED.\BLACK}
\end{align*}

In \cite[Section 4.1.2]{Buzancic.Davoli.Velcic.2024.2nd}, the spaces $\calXzero{\omega}$ and $\calYzero{\omega}$ can be characterized through disintegration and duality arguments.
Furthermore, $\calXzero{\omega}$ is the $2$-dimensional variant of the set introduced in \cite[Section 4.2]{Francfort.Giacomini.2014}, where the following structure result for two-scale weak* limits of symmetrized
gradients of functions was given.

\begin{proposition} \label{two-scale limit of symmetrized gradients}
Let $\{v^\eps\}_{\eps>0}$ be a bounded family in $BD(\omega)$ such that
\begin{equation*}
	v^\eps \weakstar v \quad \text{weakly* in $BD(\omega)$},
\end{equation*}
for some $v \in BD(\omega)$. 
Then there exists $\mu \in \calXzero{\omega}$ such that (on a further subsequence)
\begin{equation*}
	{\E}v^\eps \weakstartwoscale {\E}_{x'} v \otimes \calL^{2}_{y} + {\E}_{y}\mu \quad \text{two-scale weakly* in $\Mb(\omega \times \calY;\M^{2 \times 2}_{\sym})$}.
\end{equation*}
\end{proposition}

Similarly, we can prove the following structure result for two-scale weak* limits of Hessians:

\begin{proposition} \label{two-scale limit of Hessians}
Let $\{v^\eps\}_{\eps>0}$ be a bounded family in $BH(\omega)$ such that
\begin{equation*}
	v^\eps \weakstar v \quad \text{weakly* in $BH(\omega)$},
\end{equation*}
for some $v \in BH(\omega)$. 
Then there exists $\kappa \in \calYzero{\omega}$ such that (on a further subsequence)
\begin{equation*}
	{\D}^2v^\eps \weakstartwoscale {\D}^2_{x'}v \otimes \calL^{2}_{y} + {\D}^2_{y}\kappa \quad \text{two-scale weakly* in $\Mb(\omega \times \calY;\M^{2 \times 2}_{\sym})$}.
\end{equation*}
\end{proposition}

\begin{proof}
By compactness, the exists $\lambda \in \Mb(\omega \times \calY;\M^{2 \times 2}_{\sym})$ such that (up to a subsequence)
\begin{equation*}
	{\D}^2v^\eps \weakstartwoscale \lambda \quad \text{two-scale weakly* in $\Mb(\omega \times \calY;\M^{2 \times 2}_{\sym})$}.
\end{equation*}
Since $v^\eps \strong v$ strongly in $W^{1,1}(\omega)$, we have
\begin{align*}
	& v^\eps \weakstartwoscale v(x') \,\calL^{2}_{x'} \otimes \calL^{2}_{y} \quad \text{two-scale weakly* in $\Mb(\omega \times \calY)$},\\
	& \partial_{x_\alpha}v^\eps \weakstartwoscale \partial_{x_\alpha}v(x') \,\calL^{2}_{x'} \otimes \calL^{2}_{y} \quad \text{two-scale weakly* in $\Mb(\omega \times \calY)$}, \quad \alpha=1,2.
\end{align*}
Consider $\chi \in C_c^{\infty}(\omega \times \calY;\M^{2 \times 2}_{\sym})$ such that $\div_{y}\div_{y}\chi(x',y) = 0$. 
Then
{\allowdisplaybreaks
\begin{align*}
	& \lim_{\eps\to0} \int_{\omega} \chi\left(x',\frac{x'}{\eps}\right) : {\D}^2v^\eps(x') \,dx'\\*
	& = \lim_{\eps\to0} \int_{\omega} v^\eps(x')\,\div_{x'}\div_{x'}\left(\chi\left(x',\frac{x'}{\eps}\right)\right) \,dx'\\ 
	& = \lim_{\eps\to0} \sum_{\alpha,\beta=1,2} \int_{\omega} v^\eps(x')\,\bigg( \partial_{x_\alpha x_\beta}\chi_{\alpha \beta}\left(x',\frac{x'}{\eps}\right) + \frac{1}{\eps}\partial_{y_\alpha x_\beta}\chi_{\alpha \beta}\left(x',\frac{x'}{\eps}\right)\\*\nonumber& \hspace{10em} + \frac{1}{\eps}\partial_{x_\alpha y_\beta}\chi_{\alpha \beta}\left(x',\frac{x'}{\eps}\right) + \frac{1}{\eps^2}\partial_{y_\alpha y_\beta}\chi_{\alpha \beta}\left(x',\frac{x'}{\eps}\right) \bigg) \,dx'\\
	& = \lim_{\eps\to0} \sum_{\alpha,\beta=1,2} \int_{\omega} v^\eps(x')\,\bigg( \partial_{x_\alpha x_\beta}\chi_{\alpha \beta}\left(x',\frac{x'}{\eps}\right) + \frac{2}{\eps}\partial_{y_\alpha x_\beta}\chi_{\alpha \beta}\left(x',\frac{x'}{\eps}\right) \bigg) \,dx'\\
	& = \lim_{\eps\to0} \sum_{\alpha,\beta=1,2} \int_{\omega} v^\eps(x')\,\partial_{x_\alpha x_\beta}\chi_{\alpha \beta}\left(x',\frac{x'}{\eps}\right) \,dx' + 2 \int_{\omega} \bigg( \partial_{x_\alpha}\left( v^\eps(x')\,\partial_{x_\beta}\chi_{\alpha \beta}\left(x',\frac{x'}{\eps}\right) \right)\\*\nonumber& \hspace{10em} - \partial_{x_\alpha}v^\eps(x')\,\partial_{x_\beta}\chi_{\alpha \beta}\left(x',\frac{x'}{\eps}\right) - v^\eps(x')\,\partial_{x_\alpha x_\beta}\chi_{\alpha \beta}\left(x',\frac{x'}{\eps}\right) \bigg) \,dx'\\*
	& = \lim_{\eps\to0} \sum_{\alpha,\beta=1,2} \left\{- \int_{\omega} v^\eps(x')\,\partial_{x_\alpha x_\beta}\chi_{\alpha \beta}\left(x',\frac{x'}{\eps}\right) \,dx' - 2 \int_{\omega} \partial_{x_\alpha}v^\eps(x')\,\partial_{x_\beta}\chi_{\alpha \beta}\left(x',\frac{x'}{\eps}\right) \,dx'\right\},
\end{align*}
}\noindent
where in the last equality we used Green's theorem. 
Passing to the limit, we have
{\allowdisplaybreaks
\begin{align}
	\nonumber & \lim_{\eps\to0} \int_{\omega} \chi\left(x',\frac{x'}{\eps}\right) : {\D}^2v^\eps(x') \,dx'\\*
	\nonumber & = \sum_{\alpha,\beta=1,2} - \int_{\omega \times \calY} v(x')\,\partial_{x_\alpha x_\beta}\chi_{\alpha \beta}\left(x',y\right) \,dx' dy - 2 \int_{\omega \times \calY} \partial_{x_\alpha}v(x')\,\partial_{x_\beta}\chi_{\alpha \beta}\left(x',y\right) \,dx' dy,\\
	\nonumber & = \sum_{\alpha,\beta=1,2} - \int_{\omega \times \calY} v(x')\,\partial_{x_\alpha x_\beta}\chi_{\alpha \beta}\left(x',y\right) \,dx' dy \\*\nonumber& \hspace{4em}- 2 \int_{\omega \times \calY} \bigg( \partial_{x_\alpha}\left( v(x')\,\partial_{x_\beta}\chi_{\alpha \beta}\left(x',y\right) \right) - v(x')\,\partial_{x_\alpha x_\beta}\chi_{\alpha \beta}\left(x',y\right) \bigg) \,dx' dy,\\
	\nonumber & = \sum_{\alpha,\beta=1,2} \int_{\omega \times \calY} v(x')\,\partial_{x_\alpha x_\beta}\chi_{\alpha \beta}\left(x',y\right) \,dx' dy\\
	\nonumber & = \int_{\omega \times \calY} v(x')\,\div_{x'}\div_{x'}\left(\chi\left(x',y\right)\right) \,dx'\\* 
	& = \int_{\omega \times \calY} \chi(x',y) : d\left({\D}^2_{x'}v \otimes \calL^{2}_{y} \right).
\end{align}
}\noindent
By a density argument, we infer that
\begin{equation*}
	\int_{\omega \times \calY} \chi(x',y) : d\left(\lambda(x',y) - {\D}^2_{x'}v \otimes \calL^{2}_{y} \right) = 0,
\end{equation*}
for every $\chi \in C_0(\omega \times \calY;\M^{2 \times 2}_{\sym})$ with $\div_{y}\div_{y}\chi(x',y) = 0$ (in the sense of distributions). 
From this and 
the duality argument given in \cite[Proposition 4.8]{Buzancic.Davoli.Velcic.2024.2nd},
we conclude that there exists $\kappa \in \calYzero{\omega}$ such that 
\begin{equation*}
	\lambda - {\D}^2_{x'}v \otimes \calL^{2}_{y} = {\D}^2_{y}\kappa.
\end{equation*}
This implies the claim.
\end{proof}

Recalling $\mathring{Y}$ given in \eqref{torus without the boundary}, we introduce the following corrector space:
\begin{align*}
	\calZzero{\omega} := \Big\{\delta \in \Mb(\omega \times \mathring{Y}) : {\D}^2_{y}\delta \in \Mb(\omega \times \mathring{Y};\M^{2 \times 2}_{\sym}),& \\
	\delta(F \times \mathring{Y}) = 0 \textrm{ for every Borel set } F \subseteq \omega & \Big\}.
\end{align*}
As the notation suggests, this space can be seen as a variant of the space $\calYzero{\omega}$, and we can use the same arguments from \cite[Section 4.1]{Buzancic.Davoli.Velcic.2024.1st} to infer the following disintegration result. The analogous claims can be concluded for the spaces $\calXzero\omega$ and $\calYzero\omega$, but we will not use them here. 

\begin{proposition} \label{corrector main property - regime zero without boundary of Y}
Let $\delta \in \calZzero{\omega}$. 
Then there exist $\eta \in \Mb^+(\omega)$ and a Borel map $(x',y) \in \omega \times \mathring{Y} \mapsto \delta_{x'}(y) \in \R$ such that, for $\eta$-a.e. $x' \in \omega$,
\begin{align*}
	\delta_{x'} \in BH(\mathring{Y}), \qquad \int_{\mathring{Y}} \delta_{x'}(y) \,dy = 0, \qquad |{\D}^2_{y}\delta_{x'}|(\mathring{Y}) \neq 0,
\end{align*}
and
\begin{align*}
	\delta = \delta_{x'}(y) \,\eta \otimes \calL^{2}_{y}.
\end{align*}
Moreover, the map $x' \mapsto {\D}^2_{y}\delta_{x'} \in \Mb(\calY;\M^{2 \times 2}_{\sym})$ is $\eta$-measurable and
\begin{align*}
	{\D}^2_{y}\delta = \eta \genprod {\D}^2_{y}\delta_{x'}.
\end{align*}
\end{proposition}

\subsubsection{Unfolding for sequences in spaces of functions with measures as derivatives} \label{section unfolding for sequances}
We introduce the unfolding for measures. 
For every $\eps > 0$ and $i \in \Z^2$, let
\begin{equation*}
	Q_\eps^i := \left\{ x' \in \R^2 : \frac{x'-\eps i}{\eps} \in Y \right\}.
\end{equation*}
Given an open set $\omega \subseteq \R^2$, we will set
\begin{equation}\label{cubes inside} 
	I_\eps(\omega) := \left\{ i \in \Z^2 : Q_\eps^i \subset \omega \right\}.
\end{equation}
Given $\mu_\eps \in \Mb(\omega)$ and $Q_\eps^i \subset \omega$, we define $\mu_\eps^i \in \Mb(\calY)$ such that
\begin{equation*}
	\int_{\calY} \psi(y) \,d\mu_\eps^i(y) = \frac{1}{\eps^2} \int_{Q_\eps^i} \psi\left(\frac{x'}{\eps}\right) \,d\mu_\eps(x'), \quad \psi \in C(\calY).
\end{equation*}
\begin{definition}
Given $\eps > 0$, \emph{the unfolding measure} associated with $\mu_\eps$ is the measure $\tilde{\lambda}_\eps \in \Mb(\omega \times \calY)$ defined by
\begin{equation*}
	\tilde{\lambda}_\eps := \sum_{i \in I_\eps(\omega)} \left( \calL_{x'}^2\mres{Q_\eps^i} \right) \otimes \mu_\eps^i.
\end{equation*}
\end{definition}

The following proposition gives the relationship between the two-scale weakly* convergence and unfolding measures.
The proof is given in \cite[Proposition 4.11.]{Francfort.Giacomini.2014}.

\begin{proposition} \label{unfolding measure weak* convergence}
Let $\omega \subseteq \R^2$ be an open set and let $\{\mu_\eps\} \subset \Mb(\omega)$ be a bounded family such that
\begin{equation*}
	\mu_\eps \weakstartwoscale \mu_0 \quad \text{two-scale weakly* in $\Mb(\omega \times \calY)$}.
\end{equation*}
Let $\{\tilde{\lambda}_\eps\} \subset \Mb(\omega \times \calY)$ be the associated family of unfolding measure with $\{\mu_\eps\}$. 
Then 
\begin{equation*}
	\tilde{\lambda}_\eps \weakstar \mu_0 \quad \text{weakly* in $\Mb(\omega \times \calY)$}.
\end{equation*}
Also the converse statement is valid, i.e., if for a bounded family of measures its associated family of unfolding measures converge weakly*, then the family of measures converges two-scale weakly*.
\end{proposition}



We also recall the following result given in \cite[Theorem 5.7, Step 1]{Francfort.Giacomini.2014} which characterizes how the limit two-scale measure of symmetrized gradients behave at the boundary of regular enough subset of $\calY$. 

\begin{lemma} \label{two-scale rank-1 lemma BD}
Let $\calB \subseteq \calY$ be an open set with Lipschitz boundary, such that $\partial\calB \setminus \calT$ is a $C^1$-hypersurface, for some compact set $\calT$ with $\calH^{1}(\calT) = 0$. 
Additionally we assume that $\partial\calB \cap \calC \subseteq \calT$. 
Let $v^\eps \in BD(\omega)$ be a bounded sequence such that 
\begin{equation*}
	{\E}_{x'}v^\eps\mres{\omega \cap \calB_\eps} \weakstartwoscale \pi \quad \text{two-scale weakly* in $\Mb(\omega \times \calY;\M^{2 \times 2}_{\sym})$}.
\end{equation*}
Then $\pi$ is supported in $\omega \times {\RED\closure{\calB}\BLACK}$ and
\begin{equation} \label{two-scale rank-1 structure BD}
	\pi\mres{\omega \times (\partial\calB \setminus \calT)} = a(x',y) \odot \nu(y) \,\zeta,
\end{equation}
where $\zeta \in \Mb^+(\omega \times (\partial\calB \setminus \calT))$, $a : \omega \times (\partial\calB \setminus \calT) \to \R^2$ is a Borel map, and $\nu$ is the exterior normal to $\partial\calB$.
\end{lemma}

In the following, we will consider the Hessian of $BH$ functions and prove the corresponding results to those above. 
We will use unfolding (like it was used in \cite{Francfort.Giacomini.2014}). \Cref{associated unfolding measure BH} characterizes the unfolding measure associated to sequences of Hessians while \Cref{two-scale rank-1 lemma BH} characterizes the limit two-scale measure of Hessians at the boundary of the regular enough subset of $\calY$. 


\begin{proposition} \label{associated unfolding measure BH}
Let $\omega \subseteq \R^2$ be an open set and let $\calB \subseteq \calY$ be an open set with Lipschitz boundary. 
If $u^\eps \in BH(\omega)$, the unfolding measure associated with ${\D}^2_{x'}u^\eps\mres{(\calB_\eps \setminus \calC_\eps)}$ can be written as
\begin{equation} \label{unfolding Hessian}
	\sum_{i \in I_\eps(\omega)} \left( \calL_{x'}^2\mres{Q_\eps^i} \right) \otimes {\D}^2_{y}\hat{u}_\eps^i\mres{(\calB \setminus \calC)},
\end{equation}
where $\calC$ has been introduced in \eqref{torus boundary} and $\hat{u}_\eps^i \in BH(\mathring{Y})$ is such that
\begin{equation} \label{unfolding inequality BH}
	\int_{\mathring{Y}} \hat{u}_\eps^i(y)\,dy=0, \qquad \left\|\hat{u}_\eps^i\right\|_{BH(\mathring{Y})}
	\leq \frac{C}{\eps^2} |{\D}^2_{x'}u^\eps|\left(int(Q_\eps^i)\right),
\end{equation}
for some constant $C$ independent of $i$ and $\eps$.
\end{proposition}


\begin{proof}
Since $\calB_\eps$ has Lipschitz boundary, $\nabla_{x'}u^\eps \charfun{\calB_\eps} \in BV_{\rm loc}(\omega;\R^2)$ with
\begin{equation*}
	{\D}^2_{x'}u^\eps\mres{\calB_\eps} = D_{x'}\left(\nabla_{x'}u^\eps \charfun{\calB_\eps}\right) + \left[\nabla_{x'}u^\eps\mres{\partial\calB_\eps} \otimes \nu\right] \calH^{1}\mres{\partial\calB_\eps},
\end{equation*}
where $\nabla_{x'}u^\eps \mres{\partial\calB_\eps}$ denotes the trace of $\nabla_{x'}u^\eps \charfun{\calB_\eps}$ on $\partial\calB_\eps$, while $\nu$ is the exterior normal to $\partial\calB_\eps$. 

Remark that $\calC_\eps = \left( \cup _i \partial Q_\eps^i\right) \cap \omega$. 
Accordingly, for $i \in I_\eps(\omega)$ and $\psi \in C^1(\calY;\M^{2 \times 2}_{\sym})$,
\begin{align*}
	& \int_{Q_\eps^i} \psi\left(\frac{x'}{\eps}\right) : d\left({\D}^2_{x'}u^\eps\mres{(\calB_\eps \setminus \calC_\eps)}\right)(x')
	= \int_{int(Q_\eps^i)} \psi\left(\frac{x'}{\eps}\right) : d\left({\D}^2_{x'}u^\eps\mres{\calB_\eps}\right)(x')\\
	& = \int_{int(Q_\eps^i)} \psi\left(\frac{x'}{\eps}\right) : dD_{x'}\left(\nabla_{x'}u^\eps \charfun{\calB_\eps}\right)(x')
	+ \int_{int(Q_\eps^i)} \psi\left(\frac{x'}{\eps}\right) : \left[\nabla_{x'}u^\eps\mres{\partial\calB_\eps} \otimes \nu\right] d\calH^{1}\mres{\partial\calB_\eps}(x').
\end{align*}
We define 
\[
	\hat{u}_\eps^i(y) := \frac{1}{\eps^2} u^\eps\left(\eps i+\eps \calI(y)\right).
\]
Then $\hat{u}_\eps^i \in BH(\mathring{Y})$ (which essentially means that $\hat{u}_\eps^i$ is not necessarily $[0,1]^2$-periodic).
With change of variables we obtain
\begin{align} \label{BH total variation estimate}
	|{\D}^2_{y}\hat{u}_\eps^i|(\mathring{Y}) = \frac{1}{\eps^2} |{\D}^2_{x'}u^\eps|\left(int(Q_\eps^i)\right).
\end{align}
Furthermore,
\begin{equation*}
	\int_{int(Q_\eps^i)} \psi\left(\frac{x'}{\eps}\right) : dD_{x'}\left(\nabla_{x'}u^\eps \charfun{\calB_\eps}\right)(x') = \eps^2 \int_{\mathring{Y}} \psi(y) : dD_{y}\left(\nabla_{y}\hat{u}_\eps^i \charfun{\calB}\right)(y)
\end{equation*}
and
\begin{equation*}
	\int_{int(Q_\eps^i)} \psi\left(\frac{x'}{\eps}\right) : \left[\nabla_{x'}u^\eps\mres{\partial\calB_\eps} \otimes \nu\right] d\calH^{1}\mres{\partial\calB_\eps}(x') = \eps^2 \int_{\mathring{Y}} \psi(y) : \left[\nabla_{y}\hat{u}_\eps^i\mres{\partial\calB} \otimes \nu\right] d\calH^{1}(y).
\end{equation*}
So we get
\begin{align*}
	& \frac{1}{\eps^2} \int_{Q_\eps^i} \psi\left(\frac{x'}{\eps}\right) : d\left({\D}^2_{x'}u^\eps\mres{(\calB_\eps \setminus \calC_\eps)}\right)(x')\\
	& = \int_{\mathring{Y}} \psi(y) : dD_{x'}\left(\nabla_{y}\hat{u}_\eps^i \charfun{\calB}\right)(y)
	+ \int_{\mathring{Y}} \psi(y) : \left[\nabla_{y}\hat{u}_\eps^i\mres{\partial\calB} \otimes \nu\right] d\calH^{1}(y)\\
	& = \int_{\calY} \psi(y) : d{\D}^2_{y}\hat{u}_\eps^i\mres{(\calB \setminus \calC)}(y),
\end{align*}
from which \eqref{unfolding Hessian} follows. 
It remains to prove \eqref{unfolding inequality BH}. 
Notice that we can assume that
\begin{equation*}
	\int_{\mathring{Y}} \hat{u}_\eps^i(y)\,dy=0, \qquad \left\|\hat{u}_\eps^i\right\|_{BH(\mathring{Y})} \leq C |{\D}^2_{y}\hat{u}_\eps^i|(\mathring{Y}),
\end{equation*}
for some constant $C$ independent of $i$ and $\eps$, since by subtracting an affine term to $u^\eps$ on $Q_\eps^i$ corresponds to subtracting an affine term to $\hat{u}_\eps^i$, which does not modify the calculations done thus far. 
Hence, we can use version of Poincar\'{e} inequality in $BH$ to get the desired inequality.
Finally, from \eqref{BH total variation estimate} we have
\begin{align*}
	& \left\|\hat{u}_\eps^i\right\|_{BH(\mathring{Y})} \leq \frac{C}{\eps^2} |{\D}^2_{x'}u^\eps|\left(int(Q_\eps^i)\right).
\end{align*}
This concludes the proof of the theorem.
\end{proof}

The previous result can be used to prove the following lemma, which in turn (together with \Cref{two-scale rank-1 lemma BD}) will be used in the proof of the lower semicontinuity of dissipation functionals (see \Cref{lower semicontinuity of energies - eps to 0}).

\begin{lemma} \label{two-scale rank-1 lemma BH}
Let $\calB \subseteq \calY$ be an open set with Lipschitz boundary, such that $\partial\calB \setminus \calT$ is a union of
$C^2$ curves, 
for some compact set $\calT$ with $\calH^{1}(\calT) = 0$. 
Additionally we assume that $\partial\calB \cap \calC \subseteq \calT$. 
Let $v^\eps \in BH(\omega)$ be a bounded sequence such that 
\begin{equation*}
	{\D}^2_{x'}v^\eps\mres{\omega \cap \calB_\eps} \weakstartwoscale \pi \quad \text{two-scale weakly* in $\Mb(\omega \times \calY;\M^{2 \times 2}_{\sym})$}.
\end{equation*}
Then $\pi$ is supported in $\omega \times {\RED\closure{\calB}\BLACK}$ and

\begin{equation} \label{two-scale rank-1 structure BH - C^2 bdry}
	\pi\mres{\omega \times (\partial\calB \setminus \calT)} = a(x',y)\, \nu(y) \otimes \nu(y) \,\zeta,
\end{equation}
where $\zeta \in \Mb^+(\omega \times (\partial\calB \setminus \calT))$, $a : \omega \times (\partial\calB \setminus \calT) \to \R$ is a Borel map, and $\nu$ is the exterior normal to $\partial\calB$.
\end{lemma}

\begin{proof}

Denote by $\tilde{\pi} \in \Mb(\omega \times \calY;\M^{2 \times 2}_{\sym})$ the two-scale weak* limit (up to a subsequence) of
\begin{equation*}
	{\D}^2_{x'}v^\eps\mres{\omega \cap (\calB_\eps \setminus \calC_\eps)} \in \Mb(\omega;\M^{2 \times 2}_{\sym}).
\end{equation*}
Then it is enough to prove the analogue of \eqref{two-scale rank-1 structure BH - C^2 bdry} for $\tilde{\pi}$. 
Indeed, the two-scale weak* limit (up to a subsequence) of
\begin{equation*}
	{\D}^2_{x'}v^\eps\mres{\omega \cap (\calB_\eps \cap \calC_\eps)} \in \Mb(\omega;\M^{2 \times 2}_{\sym})
\end{equation*}
is supported on $\omega \times \closure{\calB \cap \calC}$. 
Since by assumption $\partial\calB \cap \calC \subseteq \calT$, we have that $\partial\calB \setminus \calT$ and $\closure{\calB \cap \calC}$ are disjoint sets, which implies
\begin{equation*}
	\pi\mres{\omega \times (\partial\calB \setminus \calT)} = \tilde{\pi}\mres{\omega \times (\partial\calB \setminus \calT)}.
\end{equation*}
Furthermore, we can, without loss of generality assume that $\calB$ has $C^2$ boundary and is the subset of $\mathring{Y}$, since the claim we prove can be obtained through localization, for every $C^2$ part of $\calB$. We continue the proof under this assumption. 

By \Cref{associated unfolding measure BH} we have that the unfolding measure associated with ${\D}^2_{x'}v^\eps\mres{(\calB_\eps \setminus \calC_\eps)}$ is given by
\begin{equation} \label{unfolding symmetrized gradient 2}
	\sum_{i \in I_\eps(\omega)} \left( \calL_{x'}^2\mres{Q_\eps^i} \right) \otimes {\D}^2_{y}\hat{v}_\eps^i\mres{(\calB \setminus \calC)},
\end{equation}
where $\hat{v}_\eps^i \in BH(\mathring{Y})$ are such that
\begin{equation} \label{unfolding inequality 2}
	\int_{\mathring{Y}} \hat{v}_\eps^i(y)\,dy=0, \qquad \left\|\hat{v}_\eps^i\right\|_{BH(\mathring{Y})}
	\leq \frac{C}{\eps^2} |{\D}^2_{x'}v^\eps|\left(int(Q_\eps^i)\right).
\end{equation}
Further, by \Cref{unfolding measure weak* convergence}, the family of associated measures in \eqref{unfolding symmetrized gradient 2} converge weakly* to $\tilde{\pi}$ in $\Mb(\omega \times \calY;\M^{2 \times 2}_{\sym})$. 
Then, for every $\chi \in C_0(\omega \times \mathring{Y};\M^{2 \times 2}_{\sym})$ 
we get
\begin{align*} 
\begin{split}
	\int_{\omega \times \calY} \chi(x', y) : d\tilde{\pi}(x', y)
	& = \lim_{\eps \to 0} \int_{\omega \times \calY} \chi(x', y) : d\left(\sum_{i \in I_\eps(\omega)} \left( \calL_{x'}^2\mres{Q_\eps^i} \right) \otimes {\D}^2_{y}\hat{v}_\eps^i\mres{(\calB \setminus \calC)}\right)
	\\& = \lim_{\eps \to 0} \sum_{i \in I_\eps(\omega)} \int_{Q_\eps^i} \left( \int_{\calB \setminus \calC} \chi(x', y) : d{\D}^2_{y}\hat{v}_\eps^i(y) \right) \,dx'
	\\& = \lim_{\eps \to 0} \sum_{i \in I_\eps(\omega)} \int_{Q_\eps^i} \left( \int_{\calB} \chi(x', y) : d{\D}^2_{y}\hat{v}_\eps^i(y) 
	\right) \,dx'.
\end{split}
\end{align*}

By 
the integration-by-parts formula for $H^2(\calB)$ functions given in \cite[Th\'{e}or\`{e}me 2.3]{Demengel.1983}) and extended to $BH(\calB)$ functions (see \cite[Chapter III, Proposition 2.11]{Temam.1985},
we have (under the assumption $\chi \in C_0^2(\omega \times \mathring{Y};\M^{2 \times 2}_{\sym})$)
\begin{align*} 
	& \sum_{i \in I_\eps(\omega)} \int_{Q_\eps^i} \left( \int_{\calB} \chi(x', y) : d{\D}^2_{y}\hat{v}_\eps^i(y) \right) \,dx'\\
	& = \sum_{i \in I_\eps(\omega)} \int_{Q_\eps^i} \left( \int_{\calB} \hat{v}_\eps^i\,\div_{y}\div_{y}\chi(x',y) \,dy - \int_{\partial\calB} b_0(\chi)\,\hat{v}_\eps^i \,d\calH^{1}(y) + \int_{\partial\calB} b_1(\chi)\,\frac{\partial \hat{v}_\eps^i}{\partial \nu} \,d\calH^{1}(y) \right) \,dx'\\
	& = \sum_{i \in I_\eps(\omega)} \int_{Q_\eps^i} \left( \int_{\calB} \hat{v}_\eps^i\,\div_{y}\div_{y}\chi(x',y) \,dy - \int_{\partial\calB} b_0(\chi)\,\hat{v}_\eps^i \,d\calH^{1}(y) + \int_{\partial\calB} \chi(x', y) : \left(\frac{\partial \hat{v}_\eps^i}{\partial\nu}\, \nu \otimes \nu\right) \,d\calH^{1}(y) \right) \,dx',
\end{align*}
where
\begin{align*}
	b_0(\chi) & = \div_{y}\chi(x',y)\cdot\nu(y) + \frac{\partial}{\partial \tau}\left(\chi(x',y)\,\nu(y)\cdot\tau(y)\right),\\
	b_1(\chi) & = \chi(x',y)\,\nu(y)\cdot\nu(y).
\end{align*}
We consider
\begin{equation*}
	\hat{v}^\eps(x',y) := \sum_{i \in I_\eps(\omega)} \charfun{Q_\eps^i}(x')\,\hat{v}_\eps^i(y), \qquad 
	\hat{w}^\eps(x',y) := \sum_{i \in I_\eps(\omega)} \charfun{Q_\eps^i}(x')\,\frac{\partial \hat{v}_\eps^i}{\partial\nu}\Big|_{\partial \calB \cap \mathring{Y}}(y).
\end{equation*}
In view of \eqref{unfolding inequality 2} we have that $\{\hat{v}^\eps\}$ is bounded in $L^1(\omega \times \mathring{Y})$.
Additionally, by the trace theorem in $BH(\mathring{Y})$, we have that $\{\hat{w}^\eps\}$ is bounded in $L^1(\omega \times \partial\calB)$.
Thus, we can infer that there exist $\delta \in \Mb(\omega \times \mathring{Y})$ and $\lambda \in \Mb(\omega \times \mathring{Y})$ such that (up to a subsequence)
\begin{align}
	\label{vhat^eps convergence in L^1} \hat{v}^\eps \,\calL_{x'}^2 \otimes \calL_{y}^2 & \weakstar \delta \quad \text{weakly* in $\Mb(\omega \times \mathring{Y})$},\\
	\label{what^eps convergence in L^1} \hat{w}^\eps \,\calL_{x'}^2 \otimes (\calH^{1}_{y}\mres\partial\calB ) & \weakstar \lambda \quad \text{weakly* in $\Mb(\omega \times \mathring{Y})$}.
\end{align}
As a consequence of \Cref{proposition properties two-scale convergence} (ii) the support of $\lambda$ is contained in $\omega \times \partial \calB$.
Note also that as a consequence of \eqref{unfolding inequality 2} 
\[
	{\D}^2_y\hat{v}^\eps(x',y) := \sum_{i \in I_\eps(\omega)} \charfun{Q_\eps^i}(x')\,\calL_{x'}^2 \otimes {\D}^2_y\hat{v}_\eps^i(y)
\]
is bounded in $\calM_b (\omega \times \mathring{Y};\M^{2 \times 2}_{\sym})$, and thus there exists $\xi \in \calM_b (\omega \times \mathring{Y};\M^{2 \times 2}_{\sym})$ such that 
\begin{equation} \label{boundedness of second derivatives} 
	{\D}^2_y \hat{v}^\eps \weakstar \xi \quad \text{weakly* in } \Mb(\omega \times \mathring{Y};\M^{2 \times 2}_{\sym}). 
\end{equation} 

From \eqref{vhat^eps convergence in L^1} and \eqref{boundedness of second derivatives} we easily conclude that $\delta \in \calZzero{\omega}$ and ${\D}^2_y \delta=\xi$. 
From \Cref{corrector main property - regime zero without boundary of Y} we conclude that there exists $\eta \in \calM_b^+(\omega)$ such that $\delta=\delta_{x'}(y) \eta \otimes \calL_{x'}^2$, where $\delta_{x'} \in BH(\mathring{Y})$ and $(x',y)\mapsto \delta_{x'}(y)$ is Borel measurable.

We take $\chi(x', y) = \phi(x')\psi(y)$ with $\phi \in C_c^{\infty}(\omega)$ and $\psi \in C_c^{\infty}(\mathring{Y};\M^{2 \times 2}_{\sym})$, since linear combinations of functions of this type are dense in $C_0(\omega \times \mathring{Y};\M^{2 \times 2}_{\sym})$.
We have that
\begin{align*} 
	& \sum_{i \in I_\eps(\omega)} \int_{Q_\eps^i} \left( \int_{\calB} \chi(x', y) : d{\D}^2_{y}\hat{v}_\eps^i(y) \right) \,dx'\\
	& = \int_{\omega \times \calB} \hat{v}^\eps\,\div_{y}\div_{y}\chi \,dx' dy - \int_{\omega \times \partial\calB} b_0(\chi)\,\hat{v}^\eps \,d\left(\calL_{x'}^2 \otimes \calH^{1}_{y}\right) + \int_{\omega \times \partial\calB} \chi : \hat{w}^\eps\, \nu(y) \otimes \nu(y) \,d\left(\calL_{x'}^2 \otimes \calH^{1}_{y}\right)\\
	& = \int_{\calB} \left( \int_{\omega} \hat{v}^\eps(x',y)\,\phi(x') \,dx' \right)\,\div_{y}\div_{y}\psi(y) dy - \int_{\partial\calB} b_0(\psi) \left( \int_{\omega} \hat{v}^\eps(x',y)\,\phi(x') \,dx' \right) \,d\calH^{1}_{y} 
	\\& + \int_{\omega \times \partial\calB} \chi : \hat{w}^\eps\, \nu(y) \otimes \nu(y) \,d\left(\calL_{x'}^2 \otimes \calH^{1}_{y}\right).
\end{align*}
For every fixed $\phi$, denote $z^\eps_\phi := \int_{\omega} \hat{v}^\eps(x',\cdot)\,\phi(x') \,dx'$. 
By \eqref{unfolding inequality 2} we have that $\left\{z^\eps_\phi\right\}$ is bounded in $BH(\calB)$. 
Furthermore, by the trace theorem in $BH(\calB)$, 
we have that the traces $\left\{z^\eps_\phi\right\}$ are bounded in $W^{1,1}(\partial\calB)$. 
Thus, there exists $z_\phi \in BH(\calB)$ such that, up to a subsequence,
\begin{equation} \label{z^eps_phi convergence}
	z^\eps_\phi \weakstar z_\phi \quad \text{weakly* in $BH(\calB)$},
	\quad \text{ and } \quad
	z^\eps_\phi|_{\partial \calB} \strong z_\phi|_{\partial \calB} \quad \text{strongly in $L^1(\partial\calB)$},
\end{equation}
where $z_{\phi}=\int_{\omega}\phi(x') \delta_{x'} \,d\eta(x').$
Furthermore, since $\nu \in C^1(\partial\calB;\R^2)$, we have $\frac{\partial z_{\phi}}{\partial \nu}|_{\partial \calB}=\int_{\omega} \phi(x') \frac{\partial \delta_{x'}}{\partial \nu}\, d\eta(x') |_{\partial \calB}$.
From 
\eqref{what^eps convergence in L^1},
\eqref{z^eps_phi convergence} and integration by parts for $BH(\calB)$ functions we obtain
{\allowdisplaybreaks
\begin{align*} 
	& \int_{\omega \times \calY} \chi(x', y) : d\tilde{\pi}(x', y)
	\\*& = \lim_{\eps \to 0} \left( \int_{\calB} z^\eps_\phi(y)\,\div_{y}\div_{y}\psi(y) \,dy - \int_{\partial\calB} b_0(\psi)\,z^\eps_\phi(y) \,d\calH^{1}_{y} + \int_{\omega \times \partial\calB} \chi : \hat{w}^\eps\, \nu(y) \otimes \nu(y) \,d\left(\calL_{x'}^2 \otimes \calH^{1}_{y}\right) \right)
	\\& = \int_{\calB} z_\phi\,\div_{y}\div_{y}\psi \,dy - \int_{\partial\calB} b_0(\psi)\,z_\phi \,d\calH^{1}_{y} + \int_{\omega \times \partial\calB} \chi : \left( \nu(y) \otimes \nu(y) \right) \,d\lambda
	\\& = \int_{\calB} \psi : d{\D}^2_{y}z_\phi - \int_{\partial\calB} \psi : \left(\frac{\partial z_\phi}{\partial\nu}\, \nu(y) \otimes \nu(y) \right) \,d\calH^{1}_{y} + \int_{\omega \times \partial\calB} \chi : \left(\frac{d\lambda}{d|\lambda|}\, \nu(y) \otimes \nu(y) \right) \,d|\lambda|
	\\*& = \int_{\omega \times \calB} \chi : d{\D}^2_{y}\delta - \int_{\omega \times \partial\calB} \chi : \left( \nu(y) \otimes \nu(y) \right) \,d\left(\eta \frac{\partial \delta_{x'}}{\partial\nu} \otimes \calH^1_y\right) + \int_{\omega \times \partial\calB} \chi : \left(\frac{d\lambda}{d|\lambda|}\, \nu(y) \otimes \nu(y) \right) \,d|\lambda|.
\end{align*}
}\noindent
We can conclude that
\begin{equation*}
	\tilde{\pi}\mres{\omega \times \partial\calB} = \frac{d\lambda}{d|\lambda|}\, \nu(y) \otimes \nu(y) \,|\lambda| - \nu(y) \otimes \nu(y) \,\eta \frac{\partial \delta_{x'}}{\partial\nu} \otimes \calH^1_y \mres \partial \calB.
\end{equation*}
Defining 
$\zeta := |\lambda| + \left|\frac{\partial \delta_{x'}}{\partial\nu}\right| \eta\otimes \calH^1\mres\partial \calB$ and 
$\tilde{a} := \dfrac{d\lambda}{d\zeta} - \dfrac{d \left(\frac{\partial \delta_{x'}}{\partial\nu} \eta\otimes \calH^1_y \mres \calB\right)}{d\zeta}$, 
we obtain the structure \eqref{two-scale rank-1 structure BH - C^2 bdry}.
\end{proof}

\subsection{Lower semicontinuity of the dissipation functional of multiphase materials} \label{semicontinuity of the dissipation functional}

In this section we define the dissipation functional in \eqref{definition R^hom_0} and prove lower semicontinuity result in \Cref{lower semicontinuity of energies - eps to 0}. 
Lower semicontinuity is one of the key ingredients for the proof of \Cref{main result 2}, the energy equality for the limit problem. The definition of dissipation functional is motivated by the compactness statement of \Cref{two-scale rank-1 lemma BD} and \Cref{two-scale rank-1 lemma BH}. 
Note that there is a fundamental difference between lower semicontinuity result used in \cite{Buzancic.Davoli.Velcic.2024.2nd} and the one proved here, since in \cite{Buzancic.Davoli.Velcic.2024.2nd} we analyzed the case when there is a special relation between the phases which implied that on the interface we could simply define the dissipation potential to take the values of the minimal one.

From now onward, we consider the open, Lipschitz and bounded $\ext{\omega} \subseteq \R^2$ such that $\omega$ is compactly contained in $\ext{\omega}$. 
We also denote by $\ext{\Omega} = \ext{\omega} \times I$ the associated reference domain.

The following definition will define the two-scale limit (where the limit of $(u^\eps, e^{\eps}, p^{\eps})$ will belong to). 
\begin{definition} \label{definition A^hom_0}
Let $w \in H^1_{KL}(\ext{\Omega})$. 
We define the class $\calA^{\rm hom}_{0}(w)$ of admissible two-scale configurations relative to the boundary datum $w$ as the set of triplets $(u,E,P)$ with
\begin{equation*}
	u \in BD_{KL}(\ext{\Omega}), \qquad E \in L^2(\ext{\Omega} \times \calY;\M^{2 \times 2}_{\sym}), \qquad P \in \Mb(\ext{\Omega} \times \calY;\M^{2 \times 2}_{\sym}),
\end{equation*}
such that
\begin{equation*}
	u = w, \qquad E = {\E}w, \qquad P = 0 \qquad \text{ on } (\ext{\Omega} \setminus \closure{\Omega}) \times \calY,
\end{equation*}
and also such that there exist $\mu \in \calXzero{\ext{\omega}}$, $\kappa \in \calYzero{\ext{\omega}}$ with
\begin{equation} \label{admissible two-scale configurations - regime zero}
	{\E}u \otimes \calL^{2}_{y} + {\E}_{y}\mu - x_3 {\D}^2_{y}\kappa = E \,\calL^{3}_{x} \otimes \calL^{2}_{y} + P \qquad \text{ in } \ext{\Omega} \times \calY.
\end{equation}
\end{definition}

For $(u^\eps, e^\eps, p^\eps, \alpha^\eps) \in \calA_{0}^{\rm hard,\eps}(w)$, we recall the definition of energy functionals given in \eqref{definition Q^el,eps_0}--\eqref{definition R^eps_0}. 
For $(u,E,P) \in \calA^{\rm hom}_{0}(w)$ we first define
\begin{equation} \label{definition Q^hom_0} 
	\calQ^{\rm hom}_{0}(E) := \frac{1}{2} \int_{\Omega \times \calY} \red{\C}(y) E(x,y) : E(x,y) \,{\PINK dx\,dy\BLACK}.
\end{equation}
Secondly, for 
$i \neq j$, $\nu \in \mathbb{S}^1$, $\bar{c} \in \R^2$, $\hat{c} \in \R$ define
\begin{equation} \label{R_ij definition}
	\disspot^{\rm red}_{ij}(\nu, \bar{c}, \hat{c}) := \inf\limits_{C(\bar{c}, \hat{c})} \left\{ \int_I \!\Big( \disspot_i^{\rm red}\!\big(\bar{c}^i \odot \nu + x_3 \hat{c}^i\, \nu \otimes \nu\big) + \disspot_j^{\rm red}\!\big(\bar{c}^j \odot \nu + x_3 \hat{c}^j\, \nu \otimes \nu\big) \!\Big) dx_3 \right\},
\end{equation}
where the infimum is taken over 
the set
\[
	C(\bar{c}, \hat{c}) := \left\{ (\bar{c}^i,\bar{c}^j,\hat{c}^i,\hat{c}^j) \in \mathbb{R}^2 \times \mathbb{R}^2 \times \mathbb{R} \times \mathbb{R} : \bar{c}^i+\bar{c}^j = \bar{c},\, \hat{c}^i+\hat{c}^j = \hat{c} \right\}.
\]
For any 
measure $P \in \Mb(\ext{\Omega} \times \calY;\M^{2 \times 2}_{\sym})$ for which there exist $\zeta \in \Mb^+(\ext{\omega} \times (\Gamma_{ij} \setminus S))$ and for every $i,j$ Borel maps $\bar{c} : \ext{\omega} \times (\Gamma_{ij} \setminus S) \to \R^2$, $\hat{c} : \ext{\omega} \times (\Gamma_{ij} \setminus S) \to \R$, such that
\begin{equation} \label{P form on the interface}
	P\mres{\ext{\Omega} \times (\Gamma_{ij} \setminus S)} = \big[ \bar{c}(x',y) \odot \nu(y) + x_3 \hat{c}(x',y) \, \nu(y) \otimes \nu(y) \big]\,\zeta \otimes \calL_{x_3}^1,
\end{equation}
we define $\disspot^{\rm red}(P) \in \Mb^+(\ext{\Omega} \times \calY)$ by
\begin{equation} \label{R^red measure definition}
	\disspot^{\rm red}(P) := \sum_{i} \disspot_i^{\rm red}(P) + \sum_{i < j} \disspot_{ij}^{\rm red}(P),
\end{equation}
where
\begin{align*}
	& \disspot_i^{\rm red}(P) := \disspot_i^{\rm red}\left(\frac{dP}{d|P|}\right)|P| \in \Mb^+(\ext{\Omega} \times \calY_i),
	\\& \disspot_{ij}^{\rm red}(P) := \disspot_{ij}^{\rm red}\left(\nu,\bar{c},\hat{c}\right) \zeta \otimes \calL_{x_3}^1 \in \Mb^+(\ext{\Omega} \times (\Gamma_{ij} \setminus S)).
\end{align*}
Finally, we define $\calR^{\rm hom}_{0}(P)$ as the total mass $\disspot^{\rm red}(P)(\tilde{\Omega}\times \mathcal{Y})$, i.e.
\begin{align} \label{definition R^hom_0}
\begin{split}
	\calR^{\rm hom}_{0}(P) & := \sum_{i} \int_{\ext{\Omega} \times \calY_i} \disspot^{\rm red}_i\!\left(\frac{dP}{d|P|}\right) \,d|P| 
	+ \sum_{i < j} \int_{\ext{\omega} \times (\Gamma_{ij} \setminus S)} \disspot^{\rm red}_{ij}(\nu(y), \bar{c}(x',y), \hat{c}(x',y)) \,d\zeta.
\end{split}
\end{align}

\begin{remark} \label{minimum attained} 
Note that the definition of $\disspot_{ij}^{\rm red}(P)$ is independent of the choice of the triple $(\bar{c}, \hat{c}, \zeta)$ in \eqref{P form on the interface}.
Furthermore, from \eqref{coercivity of R_i} we have that 
\begin{equation} \label{boundedness interface} 
	\disspot_{ij}^{\rm red}\left(\nu,\bar{c},\hat{c}\right) \leq C \left(|\bar{c}|+|\hat{c}|\right).
\end{equation}
By using the standard coercivity and convexity arguments it can be easily seen that the infimum in \eqref{R_ij definition} is actually a minumum. 
Moreover it is not difficult to show that 
\begin{equation} \label{coercivity interface} 
	\disspot_{ij}^{\rm red}\left(\nu,\bar{c},\hat{c}\right) \geq \alpha \left(|\bar{c}|+|\hat{c}|\right),
\end{equation}
for some $\alpha>0$. 
\end{remark}

\begin{remark} \label{remark interface potential}
The dissipation functional defined in \eqref{definition R^hom_0} {\RED does not\BLACK} have {\RED a\BLACK} dissipation potential on $\Mb(\ext{\Omega} \times \calY;\M^{2 \times 2}_{\sym})$ and is not of the form of pointwise inf-convolution since one integrates over $I$ before taking the infimum!
\end{remark}

\begin{remark} \label{marinrem} 
Using the identity \eqref{admissible two-scale configurations - regime zero} and the disintegration of measures in $\calXzero\omega$ and $\calYzero\omega$ (see \cite{Buzancic.Davoli.Velcic.2024.1st}) one can actually see that $P$ is more regular on the interface of phases (it is {\RED a\BLACK} measure in $x'$, but {\RED a\BLACK} $L^1$ function in $y$). However, for the lower semicontinuity result we will need to define the plastic dissipation for the more general measures at the interface, like it is done here.
\end{remark}

Next we give {\RED the\BLACK} lower semicontinuity result. 
\begin{theorem} \label{lower semicontinuity of energies - eps to 0}
Let $(u^\eps, e^\eps, p^\eps, \alpha^\eps) \in \calA_{0}^{\rm hard,\eps}(w)$ be such that
\begin{align}
	& u^\eps \weakstar u \quad \text{weakly* in $BD(\ext{\Omega})$},\\
	& e^\eps \weaktwoscale E \quad \text{two-scale weakly in $L^2(\ext{\Omega} \times \calY;\M^{2 \times 2}_{\sym})$}, \label{e^eps two-scale weakly}\\
	& p^\eps \weakstartwoscale P \quad \text{two-scale weakly* in $\Mb(\ext{\Omega} \times \calY;\M^{3 \times 3}_{\dev})$},
\end{align}
with $(u, E, P^{1,2}) \in \calA^{\rm hom}_{0}(w)$. 
Then we have
\begin{equation} \label{lower semicontinuity Eeps}
	\calQ^{\rm hom}_{0}(E) \leq \liminf\limits_{\eps} \calQ^{\rm el,\eps}_{0}(e^\eps) 
\end{equation}
and
\begin{equation} \label{lower semicontinuity Reps}
	\calR^{\rm hom}_{0}(P^{1,2}) \leq \liminf\limits_{\eps} \calR^{\eps}_{0}(p^\eps, \alpha^\eps).
\end{equation}
\end{theorem}

\begin{proof}
\eqref{lower semicontinuity Eeps} is the consequence of convexity. 

To prove \eqref{lower semicontinuity Reps}, let us first assume without loss of generality that 
\begin{equation} \label{WLOG finite liminf}
	\liminf\limits_{\eps} \calR^{\eps}_{0}(p^\eps, \alpha^\eps) < \infty.
\end{equation}
Identifying $p^\eps \in L^2(\ext{\Omega};\M^{3 \times 3}_{\dev})$ with elements in $\Mb(\ext{\Omega};\M^{3 \times 3}_{\dev})$, 
we can write
$p^\eps = \sum_{i} p^\eps_{i}$,
where $p^\eps_{i} := p^\eps\mres{\ext{\Omega} \cap ((\calY_i)_\eps \times I)}$. 
Up to a subsequence,
\begin{align*}
	& p^\eps_{i} \weakstartwoscale P_{i} \quad \text{two-scale weakly* in $\Mb(\ext{\Omega} \times \calY;\M^{3 \times 3}_{\dev})$}. 
\end{align*}
Clearly, 
$P = \sum_{i} P_{i}$,
with $\supp(P_{i}) \subseteq \ext{\Omega} \times \closure{\calY}_i$. 
Furthermore, considering \eqref{e^eps two-scale weakly}, we can concluded that
\begin{equation*}
	{\E}u^\eps\mres{\ext{\Omega} \cap ((\calY_i)_\eps \times I)} \weakstartwoscale E^{1,2} \,\charfun{\ext{\Omega} \times \calY_i} \,\calL^{3}_{x} \otimes \calL^{2}_{y} + P_{i}^{1,2} \quad \text{two-scale weakly* in $\Mb(\ext{\Omega} \times \calY;\M^{2 \times 2}_{\sym})$}
\end{equation*}
Recalling \eqref{transversality condition}, we can additionally assume that $\Gamma_{ij} \cap \calC \subseteq S$. 
Then, with a normal $\nu$ on $\Gamma_{ij}$ that points from $\calY_j$ to $\calY_i$ for every $j \neq i$, \Cref{two-scale rank-1 lemma BD} and \Cref{two-scale rank-1 lemma BH} imply that
\begin{equation} \label{rank-1 lemma implication 0th moment}
	\zerothI{P}_{i}^{1,2}\mres{\ext{\omega} \times (\Gamma_{ij} \setminus S)} = \bar{a}_{ij}(x',y) \odot \nu(y) \,\bar{\eta}_{ij}.
\end{equation}
\begin{equation} \label{rank-1 lemma implication 1st moment}
	\firstI{P}_{i}^{1,2}\mres{\ext{\omega} \times (\Gamma_{ij} \setminus S)} = \hat{a}_{ij}(x',y)\, \nu(y) \otimes \nu(y) \,\hat{\eta}_{ij},
\end{equation}
for suitable $\bar{\eta}_{ij},\, \hat{\eta}_{ij} \in \Mb^+ (\ext{\omega} \times (\Gamma_{ij} \setminus S))$, 
and Borel maps $\bar{a}_{ij} : \ext{\omega} \times (\Gamma_{ij} \setminus S) \to \R^2$ and $\hat{a}_{ij} : \ext{\omega} \times (\Gamma_{ij} \setminus S) \to \R$. We also have that $(P_i^{1,2})^{\perp}\mres{\ext{\omega} \times (\Gamma_{ij} \setminus S)}=-(E^{1,2})^{\perp} \mres{\ext{\omega} \times (\Gamma_{ij} \setminus S)}=0$. 

Using a version of Reshetnyak's lower semicontinuity theorem adapted for two-scale convergence (see \cite[Lemma 4.6]{Francfort.Giacomini.2014}), we deduce
\begin{align} \label{R liminf}
	\nonumber & \liminf\limits_{\eps} \int_{\Omega} \disspot\!\left(\frac{x'}{\eps}, p^\eps_{i}\right) \,dx = \liminf\limits_{\eps} \int_{\ext{\Omega}} \disspot\!\left(\frac{x'}{\eps}, \frac{dp^\eps_{i}}{d|p^\eps_{i}|}\right) \,d|p^\eps_{i}|\\
	\nonumber & = \liminf\limits_{\eps} \int_{\ext{\Omega}} \disspot_i\!\left(\frac{dp^\eps_{i}}{d|p^\eps_{i}|}\right) \,d|p^\eps_{i}|
	\geq \int_{\ext{\Omega} \times \calY} \disspot_i\!\left(\frac{dP_{i}}{d|P_{i}|}\right) \,d|P_{i}|\\
	\nonumber & = \int_{\ext{\Omega} \times \calY_i} \disspot_i\!\left(\frac{dP_{i}}{d|P_{i}|}\right) \,d|P_{i}| + \int_{\ext{\Omega} \times \Gamma} \disspot_i\!\left(\frac{dP_{i}}{d|P_{i}|}\right) \,d|P_{i}|\\
	& \geq \int_{\ext{\Omega} \times \calY_i} \disspot_i\!\left(\frac{dP_{i}}{d|P_{i}|}\right) \,d|P_{i}| + \sum_{j \neq i} \int_{\ext{\Omega} \times (\Gamma_{ij} \setminus S)} \disspot_i\!\left(\frac{dP_{i}}{d|P_{i}|}\right) \,d|P_{i}|.
\end{align}
From \eqref{rank-1 lemma implication 0th moment} and \eqref{rank-1 lemma implication 1st moment} we get
\begin{align*}
	P_{i}^{1,2}\mres{\ext{\Omega} \times (\Gamma_{ij} \setminus S)} 
	& = \zerothI{P}_{i}^{1,2}\mres{\ext{\omega} \times (\Gamma_{ij} \setminus S)} \otimes \calL^{1}_{x_3} + x_3 \firstI{P}_{i}^{1,2}\mres{\ext{\omega} \times (\Gamma_{ij} \setminus S)} \otimes \calL^{1}_{x_3}
	\\& = \bar{a}_{ij} \odot \nu(y) \,\bar{\eta}_{ij} \otimes \calL^{1}_{x_3} + x_3 \hat{a}_{ij}\, \nu(y) \otimes \nu(y) \,\hat{\eta}_{ij} \otimes \calL^{1}_{x_3}
	\\& = \!\left( \bar{b}_{ij} \odot \nu(y) + x_3 \hat{b}_{ij}\, \nu(y) \otimes \nu(y) \right) \eta_{ij} \otimes \calL^{1}_{x_3},
\end{align*}
where $\eta_{ij} := \bar{\eta}_{ij} + \hat{\eta}_{ij} \in \Mb^+ (\ext{\omega} \times (\Gamma_{ij} \setminus S))$, 
$\bar{b}_{ij} := \bar{a}_{ij}\dfrac{d\bar{\eta}_{ij}}{d\eta_{ij}}$ and $\hat{b}_{ij} := \hat{a}_{ij}\dfrac{d\hat{\eta}_{ij}}{d\eta_{ij}}$.
Since $P^{1,2} = P_{i}^{1,2} + P_{j}^{1,2}$ on $\ext{\Omega} \times (\Gamma_{ij} \setminus S)$, we can write
\begin{align} \label{P decomposition of form on interface}
	\nonumber & P^{1,2}\mres{\ext{\Omega} \times (\Gamma_{ij} \setminus S)}
	\\ \nonumber & = \!\left( \bar{b}_{ij} \odot \nu(y) + x_3 \hat{b}_{ij}\, \nu(y) \otimes \nu(y) \right) \eta_{ij} \otimes \calL^{1}_{x_3}
	+ \!\left( -\bar{b}_{ji} \odot \nu(y) + x_3 \hat{b}_{ji}\, \nu(y) \otimes \nu(y) \right) \eta_{ji} \otimes \calL^{1}_{x_3}
	\\& = \!\left( \left[\bar{c}_{ij}^i+\bar{c}_{ij}^j\right] \odot \nu(y) + x_3 \left[\hat{c}_{ij}^i+\hat{c}_{ij}^j\right]\, \nu(y) \otimes \nu(y) \right) \zeta_{ij} \otimes \calL^{1}_{x_3},
\end{align}
where $\zeta_{ij} := \eta_{ij} + \eta_{ji}$, 
and $\bar{c}_{ij}^i, \bar{c}_{ij}^j : \ext{\omega} \times (\Gamma_{ij} \setminus S) \to \R^2$, $\hat{c}_{ij}^i, \hat{c}_{ij}^j : \ext{\omega} \times (\Gamma_{ij} \setminus S) \to \R$ are suitable Borel functions.

Thus, by \eqref{R liminf} {\RED and recalling \eqref{reference1}\BLACK}, we have 
{\allowdisplaybreaks
\begin{align*}
	& \liminf\limits_{\eps} \calR^{\eps}_{0}(p^\eps)
	\\*& \geq \sum_{i} \int_{\ext{\Omega} \times \calY_i} \disspot^{\rm red}_i\!\left(\frac{dP^{1,2}}{d|P|}\right) \,d|P| + \sum_{i \neq j} \int_{\ext{\Omega} \times (\Gamma_{ij} \setminus S)} \disspot_i^{\rm red}\!\left(\bar{b}_{ij} \odot \nu(y) + x_3 \hat{b}_{ij}\, \nu(y) \otimes \nu(y) \right) \,d\eta_{ij} dx_3
	\\& = \sum_{i} \int_{\ext{\Omega} \times \calY_i} \disspot^{\rm red}_i\!\left(\frac{dP^{1,2}}{d|P^{1,2}|}\right) \,d|P^{1,2}| 
	\\*& \hspace{1em}+ \sum_{i < j} \int_{\ext{\Omega} \times (\Gamma_{ij} \setminus S)} \!\Big( \disspot_i^{\rm red}\!\big(\bar{c}_{ij}^i \odot \nu + x_3 \hat{c}_{ij}^i\, \nu \otimes \nu\big) + \disspot_j^{\rm red}\!\big(\bar{c}_{ij}^j \odot \nu + x_3 \hat{c}_{ij}^j\, \nu \otimes \nu\big) \!\Big) \,d\zeta_{ij} dx_3
	\\& \geq \sum_{i} \int_{\ext{\Omega} \times \calY_i} \disspot^{\rm red}_i\!\left(\frac{dP^{1,2}}{d|P^{1,2}|}\right) \,d|P^{1,2}| 
	\\*& \hspace{1em}+ \sum_{i < j} \int_{\ext{\omega} \times (\Gamma_{ij} \setminus S)} \inf\limits_{C_{(x',y)}}\left\{ \int_I \!\Big( \disspot_i^{\rm red}\!\big(\bar{c}^i \odot \nu + x_3 \hat{c}^i\, \nu \otimes \nu\big) + \disspot_j^{\rm red}\!\big(\bar{c}^j \odot \nu + x_3 \hat{c}^j\, \nu \otimes \nu\big) \!\Big) dx_3 \right\} \,d\zeta_{ij},
\end{align*}
}\noindent
where $C_{(x',y)} := \left\{ (\bar{c}^i,\bar{c}^j,\hat{c}^i,\hat{c}^j) : \bar{c}^i+\bar{c}^j = \bar{c}_{ij}^i(x',y)+\bar{c}_{ij}^j(x',y),\, \hat{c}^i+\hat{c}^j = \hat{c}_{ij}^i(x',y)+\hat{c}_{ij}^j(x',y) \right\}$.
Considering \eqref{P decomposition of form on interface} and the definition \eqref{definition R^hom_0}, we infer that the right-hand side in the above inequality is in fact equal to $\calR^{\rm hom}_{0}(P^{1,2})$,
which in turn concludes the proof.
\end{proof}

\subsection{Two-scale limit stresses and lower bound for the dissipation functional} \label{section two-scale limit stresses}

The goal of this section is to prove the lower bound for the plastic dissipation given in \Cref{two-scale Hill's principle - regime zero}, whose corollary will be used for proving global stability of the two-scale limit in \Cref{main result 2} (see \Cref{two-scale dissipation and plastic work inequality})).
In this way we avoid using recovery sequence argument for global stability used in \cite{Davoli.Mora.2013}.
This approach is introduced in \cite{Francfort.Giacomini.2014} and further developed in dimension reduction problems in \cite{Buzancic.Davoli.Velcic.2024.1st,Buzancic.Davoli.Velcic.2024.2nd}.
In order to use that approach one needs to define the limit stresses and prove certain compactness results as well as version of stress-strain duality on $I \times \calY$.
One then proves lower bound for plastic dissipation on a torus (see \Cref{cell Hill's principle - regime zero}) and uses disintegration of measures to prove the lower bound for two-scale plastic dissipation.
Again the fundamental difference with respect to the approach used in \cite{Buzancic.Davoli.Velcic.2024.2nd} is that here we have to take into account the new definition of the dissipation potential at the interface. 

In \Cref{section two-scale statics and duality} we introduce two-scale limit stresses and prove the compactness result, i.e. that $\eps$-stresses converge two-scale to the element of the class of limit stresses. 
In \Cref{section stress-plastic strain duality on the cell} we define the class of stresses on $I \times \calY$, state and prove stress-strain duality on $I \times \calY$ and prove lower bound for plastic dissipation $I \times \calY$. Finally in \Cref{subs:dis} we use disintegration of measures to prove the lower bound for two-scale plastic dissipation.

\subsubsection{Two-scale statics and duality} \label{section two-scale statics and duality} 
For every $e^\eps \in L^2(\Omega;\M^{2 \times 2}_{\sym})$, $p^\eps \in L^2(\Omega;\M^{3 \times 3}_{\dev})$ and $\alpha^\eps \in L^2(\Omega)$, $\alpha^\eps \geq 0$, we denote 
\begin{equation*}
	\sigma^\eps(x) := \red{\C}\!\left(\frac{x'}{\eps}\right) e^\eps(x),
	\quad \chikin^\eps(x) := - \delta(\eps)\,\Hkin\!\left(\frac{x'}{\eps}\right) p^\eps(x)
	\;\text{ and }\; \chiiso^\eps(x) := - \delta(\eps)\,\Hiso\!\left(\frac{x'}{\eps}\right) \alpha^\eps(x).
\end{equation*}
Considering the results in \Cref{Admissible stress configurations}, we introduce
\begin{align*}
	\calK^\eps_{0} = \bigg\{& (\sigma^\eps,\, \chikin^\eps,\, \chiiso^\eps) \in L^2(\Omega;\M^{3 \times 3}_{\sym}) \times L^2(\Omega;\M^{3 \times 3}_{\dev}) \times L^2(\Omega;(-\infty,0]) : 
	\\& \sigma^\eps_{i3} = 0,\quad \div_{x'}{\RED\zerothI{\sigma}\BLACK}^\eps = 0 \text{ in } \omega,\quad \div_{x'}\div_{x'}{\RED\firstI{\sigma}\BLACK}^\eps = 0 \text{ in } \omega,
	\\& \sigma^\eps_{\dev}(x',x_3)+\chikin^\eps(x',x_3) \in \!\left(1-\chiiso^\eps(x',x_3)\right) K\!\left(\frac{x'}{\eps}\right) \,\text{ for a.e. } x' \in \omega,\, x_3 \in I\bigg\},
\end{align*}
where ${\RED\zerothI{\sigma}\BLACK}^\eps,\, {\RED\firstI{\sigma}\BLACK}^\eps \in L^2(\omega;\M^{2 \times 2}_{\sym})$ are the zeroth and first order moments of the $2 \times 2$ minor of $\sigma^\eps$.
The following definition introduces the class of two-scale limit stresses (see also \cite{Buzancic.Davoli.Velcic.2024.2nd}).
\begin{definition} \label{definition K^hom_0}
The set $\calK^{\rm hom}_{0}$ is the set of all elements $\Sigma \in L^{\infty}(\Omega \times \calY;\M^{3 \times 3}_{\sym})$ satisfying:
\begin{enumerate}[label=(\roman*)]
	\item \label{definition K^hom_0 (i)} $\Sigma_{i3}(x,y) = 0 \,\text{ for } i=1,2,3$,
	\item \label{definition K^hom_0 (ii)} $\Sigma_{\dev}(x,y) \in K(y) \,\text{ for } \calL^{3}_{x} \otimes \calL^{2}_{y}\text{-a.e. } (x,y) \in \Omega \times \calY$,
	\item \label{definition K^hom_0 (iii)} $\div_{y}{\PINK\zerothI{\Sigma}\BLACK}(x',\cdot) = 0 \text{ in } \calY \,\text{ for a.e. } x' \in \omega$,
	\item \label{definition K^hom_0 (iv)} $\div_{y}\div_{y}{\PINK\firstI{\Sigma}\BLACK}(x',\cdot) = 0 \text{ in } \calY \,\text{ for a.e. } x' \in \omega$,
	\item \label{definition K^hom_0 (v)} $\div_{x'}{\PINK\zerothI{\sigma}\BLACK} = 0 \text{ in } \omega$,
	\item \label{definition K^hom_0 (vi)} $\div_{x'}\div_{x'}{\PINK\firstI{\sigma}\BLACK} = 0 \text{ in } \omega$,
\end{enumerate}
where ${\PINK\zerothI{\Sigma}\BLACK},\, {\PINK\firstI{\Sigma}\BLACK} \in L^{\infty}(\omega \times \calY;\M^{2 \times 2}_{\sym})$ are the zeroth and first order moments of the $2 \times 2$ minor of $\Sigma$, \
$\sigma := \int_{\calY} \Sigma(\cdot,y) \,dy$, and ${\PINK\zerothI{\sigma}\BLACK},\, {\PINK\firstI{\sigma}\BLACK} \in L^{\infty}(\omega;\M^{2 \times 2}_{\sym})$ are the zeroth and first order moments of the $2 \times 2$ minor of $\sigma$.
\end{definition}

\begin{remark} \label{rem Linfty} 
Note that in the definition we could demand only that $\Sigma \in L^2(\Omega \times \calY;\M^{3 \times 3}_{\sym})$, since \ref{definition K^hom_0 (i)} and \ref{definition K^hom_0 (ii)} immediately imply that $\Sigma \in L^\infty(\Omega \times \calY;\M^{3 \times 3}_{\sym})$.
\end{remark}

The above definition is motivated by the following proposition.

\begin{proposition} \label{two-scale weak limit of admissible stress - regime zero}
Let $(\sigma^\eps,\, \chikin^\eps,\, \chiiso^\eps) \in \calK^\eps_{0}$ be such that
\begin{align*}
	\sigma^\eps \weaktwoscale \Sigma & \quad \text{two-scale weakly in $L^2(\Omega \times \calY;\M^{3 \times 3}_{\sym})$},
	\\ \chikin^\eps \strong 0 & \quad \text{strongly in $L^2(\Omega;\M^{3 \times 3}_{\dev})$},
	\\ \chiiso^\eps \strong 0 & \quad \text{strongly in $L^2(\Omega)$}.
\end{align*}
Then $\Sigma \in \calK^{\rm hom}_{0}$.
\end{proposition}

\begin{proof}
Property \ref{definition K^hom_0 (i)} directly follows from the condition $\sigma^\eps_{i3} = 0$, while properties \ref{definition K^hom_0 (v)} and \ref{definition K^hom_0 (vi)} follow since ${\PINK\firstI{\sigma}\BLACK}$ and ${\PINK\zerothI{\sigma}\BLACK}$ are weak limits the zeroth and first order moments (respectively) of the $2 \times 2$ minor of $\sigma^\eps$. 

To prove \ref{definition K^hom_0 (ii)} we define (recall \eqref{cubes inside})
\begin{equation*} \label{approximating sequence for K^hom}
	\Sigma^\eps(x,y) = \sum_{i \in I_\eps(\ext{\omega})} \charfun{Q_\eps^i}\!\left(\frac{x'}{\eps}\right)\,\frac{\sigma^\eps(\eps i + \eps\calI(y),x_3)+\chikin^\eps(\eps i + \eps\calI(y),x_3)}{1-\chiiso^\eps(\eps i + \eps\calI(y),x_3)},
\end{equation*}
and consider the set
\begin{equation*}
	S = \{ \Xi \in L^2(\Omega \times \calY;\M^{3 \times 3}_{\sym}) : \Xi_{\dev}(x,y) \in K(y) \,\text{ for } \calL^{3}_{x} \otimes \calL^{2}_{y}\text{-a.e. } (x,y) \in \Omega \times \calY \}.
\end{equation*}
The construction of $\Sigma^\eps$ from $(\sigma^\eps,\, \chikin^\eps,\, \chiiso^\eps) \in \calK^\eps_{0}$ ensures that $\Sigma^\eps \in S$ and that $\Sigma^\eps \weak \Sigma \;\text{ weakly in } L^2(\Omega \times \calY;\M^{3 \times 3}_{\sym})$. 
Since compactness of $K(y)$ implies that $S$ is convex and weakly closed in $L^2(\Omega \times \calY;\M^{3 \times 3}_{\sym})$, we have that $\Sigma \in S$, which concludes the proof.
Now, properties \ref{definition K^hom_0 (i)} and \ref{definition K^hom_0 (ii)} imply that $\Sigma \in L^{\infty}$ (see \Cref{rem Linfty}).

Finally to prove \ref{definition K^hom_0 (iii)} and \ref{definition K^hom_0 (iv)} let $\phi \in C_c^{\infty}(\omega;C^{\infty}(\calY;\R^3))$ and consider the test function 
\begin{equation*}
	\varphi(x) 
	= \eps \!\left(\begin{array}{c} \phi_1(x',\frac{x'}{\eps}) \\ \phi_2(x',\frac{x'}{\eps}) \end{array}\right) 
	+ \eps^2 \!\left(\begin{array}{c} - x_3\,\partial_{x_1}\phi_3(x',\frac{x'}{\eps}) - \frac{x_3}{\eps}\partial_{y_1}\phi_3(x',\frac{x'}{\eps}) \\ - x_3\,\partial_{x_2}\phi_3(x',\frac{x'}{\eps}) - \frac{x_3}{\eps}\partial_{y_2}\phi_3(x',\frac{x'}{\eps}) \end{array}\right).
\end{equation*}
By a direct computation we infer
\begin{equation*}
	{\E}_{x'}\varphi(x) \strongtwoscale {\E}_{y}\phi^{1,2}(x',y) - x_3 {\D}^2_{y}\phi_3(x',y) \quad \text{two-scale strongly in } L^2(\Omega \times \calY;\M^{3 \times 3}_{\sym})
\end{equation*}
and
\begin{align*}
	& \int_{\Omega} \sigma^\eps(x) : \begin{pmatrix} {\E}_{x'}\varphi(x) & 0 \\ 0 & 0 \end{pmatrix} \,dx
	= \int_{\Omega} (\sigma^\eps)^{1,2}(x) : {\E}_{x'}\varphi(x) \,dx
	\\& = - \eps \int_{\omega} \div_{x'}{\PINK\zerothI{\sigma}\BLACK}^\eps(x') \cdot (\phi)^{1,2}\!\left(x',\frac{x'}{\eps}\right) \,dx' - \frac{1}{12} \eps^2 \int_{\omega} \div_{x'}\div_{x'}{\PINK\firstI{\sigma}\BLACK}^\eps(x') \cdot \phi_3\!\left(x',\frac{x'}{\eps}\right) \,dx'= 0.
\end{align*}
Hence, taking such a test function and passing to the limit, we get
\begin{equation*}
	\int_{\Omega \times \calY} \Sigma(x,y) : \begin{pmatrix} {\E}_{y}\phi' - x_3 {\D}^2_{y}\phi_3 & 0 \\ 0 & 0 \end{pmatrix} \,{\PINK dx\,dy\BLACK} = 0.
\end{equation*}
Suppose now that $\phi\!\left(x',y\right) = \psi^{(1)}(x')\,\psi^{(2)}(y)$ for $\psi^{(1)} \in C_c^{\infty}(\omega)$ and $\psi^{(2)} \in C^{\infty}(\calY;\R^3)$.
Then
\begin{equation*}
	\int_{\omega} \psi^{(1)}(x') \!\left(\int_{I \times \calY} \Sigma(x,y) : \begin{pmatrix} {\E}_{y}(\psi^{(2)})'(y) - x_3 {\D}^2_{y}\psi^{(2)}_3(y) & 0 \\ 0 & 0 \end{pmatrix} \,dx_3 dy\right) \,dx' = 0,
\end{equation*}
from which we deduce that, for a.e. $x' \in \omega$,
\begin{align*}
	0 & = \int_{I \times \calY} \Sigma(x,y) : \begin{pmatrix} {\E}_{y}(\psi^{(2)})'(y) - x_3 {\D}^2_{y}\psi^{(2)}_3(y) & 0 \\ 0 & 0 \end{pmatrix} \,dx_3 dy\\
	& = \int_{\calY} {\PINK\zerothI{\Sigma}\BLACK}(x',y) : {\E}_{y}(\psi^{(2)})'(y) \,dy - \frac{1}{12} \int_{\calY} {\PINK\firstI{\Sigma}\BLACK}(x',y) : {\D}^2_{y}\psi^{(2)}_3(y) \,dy\\
	& = - \int_{\calY} \div_{y}{\PINK\zerothI{\Sigma}\BLACK}(x',y) \cdot (\psi^{(2)})'(y) \,dy - \frac{1}{12} \int_{\calY} \div_{y}\div_{y}{\PINK\firstI{\Sigma}\BLACK}(x',y) \cdot \psi^{(2)}_3(y) \,dy.
\end{align*}
Thus, $\div_{y}{\PINK\zerothI{\Sigma}\BLACK}(x',\cdot) = 0 \text{ in } \calY$ and $\div_{y}\div_{y}{\PINK\firstI{\Sigma}\BLACK}(x',\cdot) = 0 \text{ in } \calY$. This finishes the proof. 
\end{proof}

\subsubsection{Stress-plastic strain duality on the cell} \label{section stress-plastic strain duality on the cell} 
The main result in this section is to prove the lower bound of plastic dissipation on the cell proved in \Cref{cell Hill's principle - regime zero}. 

Following the approach of \cite{Buzancic.Davoli.Velcic.2024.2nd} we introduce the set of admissible stresses on $I \times \calY$. 

\begin{definition} \label{definition K_0}
The set $\calK_{0}$ of admissible stresses is defined as the set of all elements $\Sigma \in L^2(I \times \calY;\M^{3 \times 3}_{\sym})$ satisfying:
\begin{enumerate}[label=(\roman*)]
	\item \label{definition K_0 (i)} $\Sigma_{i3}(x_3,y) = 0 \,\text{ for } i=1,2,3$,
	\item \label{definition K_0 (ii)} $\Sigma_{\dev}(x_3,y) \in K(y) \,\text{ for } \calL^{1}_{x_3} \otimes \calL^{2}_{y}\text{-a.e. } (x_3,y) \in I \times \calY$,
	\item \label{definition K_0 (iii)} $\div_{y}{\PINK\zerothI{\Sigma}\BLACK} = 0 \text{ in } \calY$,
	\item \label{definition K_0 (iv)} $\div_{y}\div_{y}{\PINK\firstI{\Sigma}\BLACK} = 0 \text{ in } \calY$,
\end{enumerate}
where ${\PINK\zerothI{\Sigma}\BLACK},\, {\PINK\firstI{\Sigma}\BLACK} \in L^2(\calY;\M^{2 \times 2}_{\sym})$ are the zeroth and the first order moments of the $2 \times 2$ minor of $\Sigma$.
\end{definition}

Recalling \eqref{K_r characherization}, by conditions \ref{definition K_0 (i)} and \ref{definition K_0 (ii)} we may identify $\Sigma \in \calK_{0}$ with an element of $L^{\infty}(I \times \calY;\M^{2 \times 2}_{\sym})$ such that $\Sigma(x_3,y) \in \red{K}(y) \,\text{ for } \calL^{1}_{x_3} \otimes \calL^{2}_{y}\text{-a.e. } (x_3,y) \in I \times \calY$.
Thus, in this regime it will be natural to define the family of admissible configurations by means of conditions formulated on $\M^{2 \times 2}_{\sym}$.

Next we introduce the set of admissible configurations on the cell $I \times \calY$.

\begin{definition} \label{definition A_0}
The family $\calA_{0}$ of admissible configurations is given by the set of quadruplets
\begin{equation*}
	\bar{u} \in BD(\calY), \qquad u_3 \in BH(\calY), \qquad E \in L^2(I \times \calY;\M^{2 \times 2}_{\sym}), \qquad P \in \Mb(I \times \calY;\M^{2 \times 2}_{\sym}),
\end{equation*}
such that
\begin{equation} \label{disss1}
	{\E}_{y}\bar{u} - x_3 {\D}^2_{y}u_3 = E \,\calL^{1}_{x_3} \otimes \calL^{2}_{y} + P \quad \textit{ in } I \times \calY.
\end{equation}
\end{definition}

\begin{remark} \label{dissipation on acell} 
Let $P \in \Mb(I \times \calY;\M^{2 \times 2}_{\sym})$ be any measure satisfying \eqref{disss1}. 
Then, from the structure theorems for $BD$ and $BH$ functions, we can see that, for any $C^1$-hypersurface $\calD \subseteq \calY$, there exist $\zeta \in \Mb^+(\calD)$ \footnote{Again{\RED,\BLACK}  the measure is actually {\RED a\BLACK} $L^1$ function on the interface, cf. \Cref{marinrem}{\RED.\BLACK}} and Borel maps $(\bar{c}, \hat{c}) \in L^1_{d\zeta}(\calD;\R^2) \times L^1_{d\zeta}(\calD)$, such that
\begin{equation} \label{P form on the interface without x'}
	P\mres{I \times \calD} = \big[ \bar{c}(y) \odot \nu(y) + x_3 \hat{c}(y) \, \nu(y) \otimes \nu(y) \big]\,\zeta \otimes \calL_{x_3}^1.
\end{equation}
Here $\nu$ denotes a continuous unit normal vector field to $\calD$. 
Recalling \eqref{R_ij definition}, we can then give meaning to the measure on $I \times (\Gamma_{ij} \setminus S)$ 
\begin{align*}
	\disspot_{ij}^{\rm red}(P) := \disspot_{ij}^{\rm red}\left(\nu,\bar{c},\hat{c}\right) \zeta \otimes \calL_{x_3}^1 \in \Mb^+(I \times (\Gamma_{ij} \setminus S)), 
\end{align*}
and consequently the measure $\disspot^{\rm red}(P)$ given by \eqref{R^red measure definition} is a well defined element of $\Mb^+(I \times \calY)$.
\end{remark}

Recalling the definitions of zeroth and first order moments of functions and measures (see \Cref{moments of functions} and \Cref{moments of measures}), we introduce the following analogue of the duality between moments of stresses and plastic strains (see also \cite{Buzancic.Davoli.Velcic.2024.2nd}). 

\begin{definition}
Let 
$\Sigma \in \calK_{0}$
and let 
$(\bar{u}, u_3, E, P) \in \calA_{0}$.
We define the distributions $[ {\PINK\zerothI{\Sigma}\BLACK} : {\PINK\zerothI{P}\BLACK} ]$ and $[ {\PINK\firstI{\Sigma}\BLACK} : {\PINK\firstI{P}\BLACK} ]$ on $\calY$ by
\begin{gather}
\begin{split} \label{cell stress-strain duality zero moment - regime zero}
	[ {\PINK\zerothI{\Sigma}\BLACK} : {\PINK\zerothI{P}\BLACK} ](\varphi) :=
	- \int_{\calY} \varphi\,{\PINK\zerothI{\Sigma}\BLACK} : {\PINK\zerothI{E}\BLACK} \,dy 
	- \int_{\calY} {\PINK\zerothI{\Sigma}\BLACK} : \big( \bar{u} \odot \nabla_{y}\varphi \big) \,dy,
\end{split}\\
\begin{split} \label{cell stress-strain duality first moment - regime zero}
	[ {\PINK\firstI{\Sigma}\BLACK} : {\PINK\firstI{P}\BLACK} ](\varphi) :=
	- \int_{\calY} \varphi\,{\PINK\firstI{\Sigma}\BLACK} : {\PINK\firstI{E}\BLACK} \,dy 
	+ 2 \int_{\calY} {\PINK\firstI{\Sigma}\BLACK} : \big( \nabla_{y}u_3 \odot \nabla_{y}\varphi \big) \,dy 
	+ \int_{\calY} u_3\,{\PINK\firstI{\Sigma}\BLACK} : {\D}^2_{y}\varphi \,dy,
\end{split}
\end{gather}
for every $\varphi \in C^{\infty}(\calY)$.
\end{definition}

\begin{remark}
Note that the second integral in \eqref{cell stress-strain duality zero moment - regime zero} is well defined since $BD(\calY)$ is embedded into $L^2(\calY;\R^2)$. 
Similarly, the second and third integrals in \eqref{cell stress-strain duality first moment - regime zero} are well defined since $BH(\calY)$ is embedded into $H^1(\calY)$. 
Moreover, the definitions are independent of the choice of $(u, E)$, so \eqref{cell stress-strain duality zero moment - regime zero} and \eqref{cell stress-strain duality first moment - regime zero} define a meaningful distributions on $\calY$ (this is valid for arbitrary ${\PINK\zerothI{\Sigma}\BLACK}, {\PINK\firstI{\Sigma}\BLACK} \in L^\infty(\calY;\mathbb{M}_{\sym}^{2 \times 2})$ that satisfy the properties \ref{definition K_0 (iii)} and \ref{definition K_0 (iv)} of \Cref{definition K_0}).

Arguing as in \cite[Section 7]{Davoli.Mora.2013}, one can prove that $[ {\PINK\zerothI{\Sigma}\BLACK} : {\PINK\zerothI{P}\BLACK} ]$ and $[ {\PINK\firstI{\Sigma}\BLACK} : {\PINK\firstI{P}\BLACK} ]$ are bounded Radon measures on $\calY$. 
For ${\PINK\zerothI{\Sigma}\BLACK}$ of class $C^1$ and ${\PINK\firstI{\Sigma}\BLACK}$ of class $C^2$ it can be shown by integration by parts (see, e.g., \cite{Francfort.Giacomini.2012} and \cite[Remark 7.1, Remark 7.4]{Davoli.Mora.2015} that 
\begin{equation} \label{dualityadded1} 
	\int_{\calY} \varphi \,d[{\PINK\zerothI{\Sigma}\BLACK}:{\PINK\zerothI{P}\BLACK}] = \int_{\calY} \varphi {\PINK\zerothI{\Sigma}\BLACK} : d{\PINK\zerothI{P}\BLACK}, \quad 
	\int_{\calY} \varphi \,d[{\PINK\firstI{\Sigma}\BLACK}:{\PINK\firstI{P}\BLACK}] = \int_{\calY} \varphi {\PINK\firstI{\Sigma}\BLACK} : d{\PINK\firstI{P}\BLACK}. 
\end{equation} 
From this it follows that for ${\PINK\zerothI{\Sigma}\BLACK}$ of class $C^1$ and ${\PINK\firstI{\Sigma}\BLACK}$ of class $C^2$ we have 
\begin{equation} \label{dualityadded2} 
	\left|[{\PINK\zerothI{\Sigma}\BLACK}:{\PINK\zerothI{P}\BLACK}]\right| \leq \|{\PINK\zerothI{\Sigma}\BLACK}\|_{L^\infty}|{\PINK\zerothI{P}\BLACK}|, \quad 
	\left|[{\PINK\firstI{\Sigma}\BLACK}:{\PINK\firstI{P}\BLACK}]\right| \leq \|{\PINK\firstI{\Sigma}\BLACK}\|_{L^\infty}|{\PINK\firstI{P}\BLACK}|, \quad \varphi \in C(\calY).
\end{equation} 
Through the approximation by convolution \eqref{dualityadded1} then extends to arbitrary continuous ${\PINK\zerothI{\Sigma}\BLACK}$, ${\PINK\firstI{\Sigma}\BLACK}$ and \eqref{dualityadded2} applies to arbitrary ${\PINK\zerothI{\Sigma}\BLACK}, {\PINK\firstI{\Sigma}\BLACK} \in L^\infty(\calY;\mathbb{M}_{\sym}^{2 \times 2})$ satisfying the properties \ref{definition K_0 (iii)} and \ref{definition K_0 (iv)} of \Cref{definition K_0}.
\end{remark}

\begin{remark} \label{remadd1}
If $\alpha$ is a simple $C^2$ curve in $\calY$, then
\begin{equation} \label{dualityadded3} 
	[{\PINK\zerothI{\Sigma}\BLACK}:{\PINK\zerothI{P}\BLACK}] = {\PINK\zerothI{\Sigma}\BLACK}\nu \cdot [[\bar{u}]] \,\calH^{1}_{y} \quad \textrm{ on } \alpha, 
\end{equation} 
where $\nu$ is a unit normal on the curve $\alpha$, while $[[\bar{u}]] = \bar{u}^{+}-\bar{u}^{-}$ denotes the jump of $\bar{u}$ over $\alpha$, with $\bar{u}^{+}$ and $\bar{u}^{-}$ the traces on $\alpha$ of $\bar{u}$ ($\bar{u}^{+}$ is from the side toward which normal is pointing, $\bar{u}^{-}$ is from the opposite side). This can be obtained from \eqref{cell stress-strain duality zero moment - regime zero}, \eqref{dualityadded1}, \eqref{dualityadded2} and approximation by convolution, see e.g. \cite[Lemma 3.8]{Francfort.Giacomini.2012}. 

From \eqref{traces of the stress} it follows that if $U$ is an open set in $\calY$ whose boundary is of class $C^2$ (and that intersects $\alpha$) and ${\PINK\zerothI{\Sigma}\BLACK}_n \in L^\infty(U;\mathbb{M}_{\sym}^{2 \times 2})$ a bounded sequence such that ${\PINK\zerothI{\Sigma}\BLACK}_n \to {\PINK\zerothI{\Sigma}\BLACK}$ almost everywhere (and thus in $L^p(U)$, for every $p<\infty$) and $\div_y {\PINK\zerothI{\Sigma}\BLACK}_n\to 0$ strongly in $L^2(U)$, then ${\PINK\zerothI{\Sigma}\BLACK}_n \nu \weakstar {\PINK\zerothI{\Sigma}\BLACK}\nu$, weakly* in $L^{\infty}(K \cap \alpha)$ for any compact set $K \subset U$.
\end{remark}

\begin{remark} \label{remadd2}
It can be shown that if $\alpha \subset \calY$ is simple $C^2$ closed or non-closed $C^2$ curve with endpoints $\{a,b\}$ that there exists $b_1({\PINK\firstI{\Sigma}\BLACK}) \in L^{\infty} (\alpha) $ such that
\begin{equation} \label{curve1} 
	[{\PINK\firstI{\Sigma}\BLACK}:{\PINK\firstI{P}\BLACK}] = b_1({\PINK\firstI{\Sigma}\BLACK}) [[\partial_{\nu}u_3]] \,\calH^{1}_{y} \quad \textrm{ on } \alpha,
\end{equation} 
where $\nu$ is a unit normal on the curve $\alpha$, while $[[\partial_{\nu}u_3]] = \frac{\partial u_3^{+}}{\partial\nu} - \frac{\partial u_3^{-}}{\partial\nu}$ is a jump in the normal derivative of $u_3$ (from the side in the opposite direction of the normal), which is an $L^1 (\alpha)$ function. This
can be obtained by approximation of ${\PINK\firstI{\Sigma}\BLACK}$ by convolution from \eqref{cell stress-strain duality first moment - regime zero}, \eqref{dualityadded1}, \eqref{dualityadded2} 
and \cite[Th\'{e}oreme 2.1]{Demengel.1984}.

From \cite[Section 7.1]{Davoli.Mora.2013} and \cite[Th\'{e}oreme 2.1 and Appendice, Th\'{e}oreme 1]{Demengel.1984} it follows that if $U$ is an open set in $\calY$ whose boundary is of class $C^2$ (and that intersects $\alpha$) and ${\PINK\firstI{\Sigma}\BLACK}_n \in L^\infty(U;\mathbb{M}_{\sym}^{2 \times 2})$ a bounded sequence such that ${\PINK\firstI{\Sigma}\BLACK}_n \to {\PINK\firstI{\Sigma}\BLACK}$ almost everywhere (and thus in $L^p(U)$, for every $p<\infty$) and $\div_y \div_y {\PINK\firstI{\Sigma}\BLACK}_n\to 0$ strongly in $L^2(U)$, then $b_1({\PINK\firstI{\Sigma}\BLACK}_n) \weakstar b_1({\PINK\firstI{\Sigma}\BLACK})$, weakly* in $L^{\infty}(K \cap \alpha)$ for any compact set $K \subset U$. 
\end{remark}

We are now in a position to introduce a duality pairing between admissible stresses and plastic strains. 

\begin{definition}
Let 
$\Sigma \in \calK_{0}$
and let 
$(\bar{u}, u_3, E, P) \in \calA_{0}$. 
Then we can define a bounded Radon measure $[ \Sigma : P ]$ on $I \times \calY$ by setting
\begin{equation*}
	[ \Sigma : P ] := 
	[ {\PINK\zerothI{\Sigma}\BLACK} : {\PINK\zerothI{P}\BLACK} ] \otimes \calL^{1}_{x_3} 
	+ \frac{1}{12} [ {\PINK\firstI{\Sigma}\BLACK} : {\PINK\firstI{P}\BLACK} ] \otimes \calL^{1}_{x_3} 
	- \Sigma^\perp : E^\perp,
\end{equation*}
so that
\begin{align} \label{cell stress-strain duality - regime zero}
\begin{split}
	\int_{I \times \calY} \varphi \,d[ \Sigma : P ] 
	& = - \int_{I \times \calY} \varphi\,\Sigma : E \,dx_3 dy 
	- \int_{\calY} {\PINK\zerothI{\Sigma}\BLACK} : \big( \bar{u} \odot \nabla_{y}\varphi \big) \,dy\\
	& \,\quad + \frac{1}{6} \int_{\calY} {\PINK\firstI{\Sigma}\BLACK} : \big( \nabla_{y}u_3 \odot \nabla_{y}\varphi \big) \,dy 
	+ \frac{1}{12} \int_{\calY} u_3\,{\PINK\firstI{\Sigma}\BLACK} : {\D}^2_{y}\varphi \,dy,
\end{split}
\end{align}
for every $\varphi \in C^2(\calY)$.
\end{definition}

\begin{remark} \label{remmarin1} 
Notice that 
\[
	\zerothI{[ \Sigma : P ]} := 
	[ {\PINK\zerothI{\Sigma}\BLACK} : {\PINK\zerothI{P}\BLACK} ] 
	+ \frac{1}{12} [ {\PINK\firstI{\Sigma}\BLACK} : {\PINK\firstI{P}\BLACK} ] 
	- \zerothI{\Sigma^\perp : E^\perp}.
\]
\end{remark}

The following proposition will be used in \Cref{two-scale Hill's principle - regime zero}. 
Its proof requires to use the definition of the dissipation potential at the interface and thus one needs to use additional arguments when compared with \cite{Buzancic.Davoli.Velcic.2024.2nd}. We recall \Cref{dissipation on acell}.

\begin{proposition} \label{cell Hill's principle - regime zero}
Let $\Sigma \in \calK_{0}$ and $(\bar{u}, u_3, E, P) \in \calA_{0}$. 
If $\calY$ is a geometrically admissible multi-phase torus, we have 
\begin{equation} \label{inequality1}
	\zerothI{\,\disspot^{\rm red}(P)} \geq \zerothI{[\Sigma : P ]}.
\end{equation}
\end{proposition}

\begin{proof}
Inequality \eqref{inequality1} on an arbitrary phase $\calY_i$ can be written as
\begin{equation} \label{inequality on the phases}
	\zerothI{\,\disspot^{\rm red}_i\left(\frac{dP}{d|P|}\right) |P|} \geq \zerothI{[\Sigma : P ]},
\end{equation}
which follows from Step 1 of the proof of \cite[Proposition 5.9.]{Buzancic.Davoli.Velcic.2024.2nd}. 
We quickly sketch the proof, and note that this step is independent on the assumption on the ordering of the phases, assumed in that paper, thus the same proof works in our more general case as well.

Regularizing $\Sigma$ just by convolution with respect to $y$, we obtain a sequence $\{\Sigma_n\}$ of smooth tensors that converges strongly in $L^2(I \times \calY; \M^{2 \times 2}_{\sym})$ to $\Sigma$ and, for sufficiently large $n$, $\Sigma_n$ satisfies conditions from \Cref{definition K_0} for almost every point that is sufficiently far from the boundary $\partial\calY_i$. 
Next, the measure $P$ is decomposed as
\begin{equation*}
	P = {\PINK\zerothI{P}\BLACK} \otimes \calL^{1}_{x_3} + {\PINK\firstI{P}\BLACK} \otimes x_3 \calL^{1}_{x_3} + P^\perp,
\end{equation*}
where ${\PINK\zerothI{P}\BLACK}, {\PINK\firstI{P}\BLACK} \in \Mb(\calY;\M^{2 \times 2}_{\sym})$ and $P^\perp \in L^2(I \times \calY;\M^{2 \times 2}_{\sym})$. 
It follows that the total variation measure $|P|$ is absolutely continuous with respect to the measure
\begin{equation*}
	\Pi := |{\PINK\zerothI{P}\BLACK}| \otimes \calL^{1}_{x_3} + |{\PINK\firstI{P}\BLACK}| \otimes \calL^{1}_{x_3} + \calL^{3}_{x_3,y}.
\end{equation*}
Consequently, for $|\Pi|$-almost every $(x_3,y) \in I \times \calY_i$ sufficiently far from the boundary $\partial\calY_i$ and for sufficiently large $n$, the inequality
\begin{equation}
	\disspot^{\rm red}_i\left(\frac{dP}{d|\Pi|}\right) \geq \Sigma_n : \frac{dP}{d|\Pi|}
\end{equation}
holds. 
Finally, integrating over $I \times \calY_i$ with respect to $d|\Pi|$ and passing to $d|P|$, as well as using the convergence properties of moments of $\Sigma_n$ and duality formulas, it is easily seen that for every non-negative test function $\varphi \in C_c(\calY_i)$ we obtain the desired inequality
\begin{equation*}
	\int_{I \times \calY_i} \varphi(y)\,\disspot^{\rm red}_i\left(\frac{dP}{d|P|}\right) \,d|P| \geq \int_{I \times \calY_i} \varphi \,d[ \Sigma : P ].
\end{equation*}
which proves the inequality \eqref{inequality on the phases} on $\calY_i$.

In order to prove the inequality \eqref{inequality1} on the interfaces, we consider a curve $\alpha$ that is of class $C^2$ (together with its possible endpoints) which is a connected component of $\Gamma \backslash S$. 
The points on $\alpha$ (with the exception of the possible endpoints) belong to the intersection of the boundary of exactly two phases $\partial \calY_i \cap \partial \calY_j$.
Since $P = {\PINK\zerothI{P}\BLACK} +x_3 {\PINK\firstI{P}\BLACK}$ on $\alpha$, from the continuity of $u_3$ we find 
\begin{align}
	\label{bar P on curve alpha} & {\PINK\zerothI{P}\BLACK} = [[\bar{u}]] \odot \nu \,\calH^{1}_{y} \quad \textrm{ on } \alpha, 
	\\\label{hat P on curve alpha} & {\PINK\firstI{P}\BLACK} = [[\partial_{\nu}u_3]] \nu \odot \nu \,\calH^{1}_{y} \quad \textrm{ on } \alpha,
\end{align}
where $\nu$ is a normal at the interface, $[[\bar{u}]]$ is the jump of $\bar{u}$ and $[[\partial_{\nu}u_3]]$ is a jump in the normal derivative of $u_3$ on $\alpha$ from $\calY_i$ and $\calY_j$, respectively. 
From \eqref{dualityadded3} and \eqref{curve1} (cf. \Cref{remmarin1}) we deduce 
\begin{equation} 
\label{dualcurve} 
	\zerothI{[\Sigma:P]} = \left( {\PINK\zerothI{\Sigma}\BLACK}\nu \cdot [[\bar{u}]] + \frac{1}{12} b_1({\PINK\firstI{\Sigma}\BLACK}) [[\partial_{\nu}u_3]] \right) \,\calH^{1}_{y} \quad \textrm{ on } \alpha. 
\end{equation}
In order to prove the claim we will approximate the stress tensor $\Sigma$ with smooth ones locally around $\alpha$ in two different ways, once using the values from one phase, the other time from the other phase, obtaining in that way the approximations that take values only in $K^{\rm red}_i$, i.e. $K^{\rm red}_j$, where $i$ and $j$ are two different phases that meet on $\alpha$. 

Since $\calY_i$ and $\calY_j$ are bounded open sets with piecewise $C^2$ boundary (in particular, with Lipschitz boundary) by \cite[Proposition 2.5.4]{Carbone.DeArcangelis.2002} there exist finite open coverings $\{\calU^{(i)}_k\}$ of $\closure{\calY}_i$ and $\{\calU^{(j)}_l\}$ of $\closure{\calY}_j$ such that $\calY_i \cap \calU^{(i)}_k$ and $\calY_j \cap \calU^{(j)}_l$ are (strongly) star-shaped with Lipschitz boundary
(moreover, those $\calU^{(i)}_k$ and $\calU^{(j)}_l$ that intersect the boundary have cylindrical form up to rotation). 
We take only those members of the covering that have non-empty intersection with $\alpha$, and we can easily modify these cylindrical sets to be of class $C^2$. 
Let $\{\psi^{(i)}_k\}$ be a smooth partition of unity of $\alpha$ subordinate to the covering $\{\calU^{(i)}_k\}$, i.e. $0 \leq \psi^{(i)}_k\leq 1$, $\supp(\psi_k^{(i)}) \subset \calU_k^{(i)}$ and $\sum_{k} \psi^{(i)}_k = 1$ on $\alpha$ and let $\varphi \in C_c(\alpha)$ be an arbitrary non-negative function. 
For each $k$ we define an approximation of the stress $\Sigma$ on $\calY_i \cap \calU^{(i)}_k$ 
by
\begin{equation} \label{prop4}
	\Sigma^{(i)}_{n,k}(x_3,y) := \big(\left( \Sigma \circ d_{n,k}^{(i)} \right)(x_3,\cdot) \ast \rho_{\frac{1}{n+1}}\big)(y),
\end{equation}
where 
$\rho$ is the standard mollifier,
$d^{(i)}_{n,k}(x_3,y) = \left( x_3,\tfrac{n}{n+1}(y-y_k^{(i)})+y_k^{(i)}\right)$ and $y_k^{(i)}$ is the point with respect to which $\calY_i \cap \calU^{(i)}_k$ is star shaped. 
Obviously one has for every $k$ 
\begin{enumerate}[label=(\roman*)]
\item $\Sigma_{n,k}^{(i)} \in (K_i)^{\rm red}$ for a.e. $(x_3,y) \in I \times (\closure{\calY}_i\cap \calU^{(i)}_k)$,
\item $\|\Sigma^{(i)}_{n,k}\|_{L^\infty} \leq \|\Sigma\|_{L^{\infty}(\calY_i \cap \calU^{(i)}_k)} $, 
\item $\Sigma^{(i)}_{n,k} \to \Sigma$,\, ${\PINK\zerothI{\Sigma}\BLACK}^{(i)}_{n,k} \to {\PINK\zerothI{\Sigma}\BLACK}$,\, ${\PINK\firstI{\Sigma}\BLACK}^{(i)}_{n,k} \to {\PINK\firstI{\Sigma}\BLACK}$ strongly in $L^2(\closure{\calY}_i\cap \calU^{(i)}_k;\mathbb{M}_{\sym}^{2 \times 2})$,
\item $\div_y {\PINK\zerothI{\Sigma}\BLACK}_{n,k}^{(i)}=0$,\, $\div_y \div_y{\PINK\firstI{\Sigma}\BLACK}_{n,k}^{(i)}=0$. 
\end{enumerate}
Similarly, let $\{\psi^{(j)}_l\}$ be a partition of unity of $\alpha$ subordinate to the covering $\{\calU^{(j)}_l\}$, and for each $l$ we define an approximation $\Sigma^{(j)}_{n,l}$ of the stress $\Sigma$ on $\calY_j \cap \calU^{(j)}_l$ which satisfies the analogue properties to (i)-(iv).

Using \eqref{bar P on curve alpha} and \eqref{hat P on curve alpha}, we conclude for every $k$ and $l$
\begin{align*}
	& \int_{I \times \alpha} \psi_k^{i}(y)\psi_l^{j}(y)\varphi(y)\,d\disspot^{\rm red}(P) 
	= \int_{\alpha} \psi_k^{i}(y)\psi_l^{j}(y)\varphi(y) \disspot^{\rm red}_{ij}\!\left(\nu,[[\bar{u}]], [[\partial_{\nu}u_3]]\right) \,d\calH^{1}_{y}
	\\& = \int_{I \times \alpha} \psi_k^{i}(y)\psi_l^{j}(y)\varphi(y) \!\Big( \disspot_i^{\rm red}\!\big(\bar{c}^i(y) \odot \nu + x_3 \hat{c}^i(y)\, \nu \otimes \nu\big) + \disspot_j^{\rm red}\!\big(\bar{c}^j(y) \odot \nu + x_3 \hat{c}^j(y)\, \nu \otimes \nu\big) \!\Big) \,d\calH^{2}_{x_3,y},
\end{align*}
where $(\bar{c}^i(y),\bar{c}^j(y),\hat{c}^i(y),\hat{c}^j(y))$ are such that the minimum is attained in \eqref{R_ij definition} (see \Cref{minimum attained}, it is not difficult to argument measurability). 
In particular, $\bar{c}^i+\bar{c}^j = [[\bar{u}]]$ and $\hat{c}^i+\hat{c}^j = [[\partial_{\nu^i}u_3]]$ on $\supp\psi_k^{i} \cap \supp\psi_l^{j}$.

From $\Sigma_{n,k}^{(i)} \in (K_i)^{\rm red}$ and $\Sigma_{n,l}^{(j)} \in (K_j)^{\rm red}$ a.e. on $\supp\psi_k^{i} \cap \supp\psi_l^{j}$, it follows
\begin{align*}
	\disspot_i^{\rm red}\!\big(\bar{c}^i(y) \odot \nu + x_3 \hat{c}^i(y)\, \nu \otimes \nu\big) \geq \Sigma_{n,k}^{(i)} : \big(\bar{c}^i(y) \odot \nu + x_3 \hat{c}^i(y)\, \nu \otimes \nu\big),\\
	\disspot_j^{\rm red}\!\big(\bar{c}^j(y) \odot \nu + x_3 \hat{c}^j(y)\, \nu \otimes \nu\big) \geq \Sigma_{n,l}^{(j)} : \big(\bar{c}^j(y) \odot \nu + x_3 \hat{c}^j(y)\, \nu \otimes \nu\big).
\end{align*}
Next, by using \Cref{remadd1}, \Cref{remadd2} and \eqref{dualcurve} we have
\begin{align*}
	& \int_{I \times \alpha} \psi_k^{i}(y)\psi_l^{j}(y)\varphi(y)\,d\disspot^{\rm red}(P) 
	\\& \geq \int_{I \times \alpha} \psi_k^{i}(y)\psi_l^{j}(y)\varphi(y) \Big( \Sigma_{n,k}^{(i)} : \big(\bar{c}^i(y) \odot \nu + x_3 \hat{c}^i(y)\, \nu \otimes \nu\big) + \Sigma_{n,l}^{(j)} : \big(\bar{c}^j(y) \odot \nu + x_3 \hat{c}^j(y)\, \nu \otimes \nu\big) \!\Big) \,d\calH^{2}_{x_3,y}
	\\& = \int_{\alpha} \psi_k^{i}(y)\psi_l^{j}(y)\varphi(y) \Big( {\PINK\zerothI{\Sigma}\BLACK}_{n,k}^{(i)}\nu \cdot \bar{c}^i(y) + \frac{1}{12} b_1\left({\PINK\firstI{\Sigma}\BLACK}_{n,k}^{(i)}\right) \hat{c}^i(y) + {\PINK\zerothI{\Sigma}\BLACK}_{n,l}^{(j)}\nu \cdot \bar{c}^j(y) + \frac{1}{12} b_1\left({\PINK\firstI{\Sigma}\BLACK}_{n,l}^{(j)}\right) \hat{c}^j(y) \!\Big) \,d\calH^{1}_{y}
	\\& \to \int_{\alpha} \psi_k^{i}(y)\psi_l^{j}(y)\varphi(y) \left( {\PINK\zerothI{\Sigma}\BLACK}\nu \cdot [[\bar{u}]] + \frac{1}{12} b_1({\PINK\firstI{\Sigma}\BLACK}) [[\partial_{\nu}u_3]] \right) d\calH^{1}_{y}
	= \int_{\alpha} \psi_k^{i}(y)\psi_l^{j}(y)\varphi(y) \,d\zerothI{[\Sigma:P]}.
\end{align*}
Note that in the above computations we used that for both of these approximations we have 
\begin{eqnarray*} 
	& &{\PINK\zerothI{\Sigma}\BLACK}_{n,k}^{(i)} \nu \weakstar {\PINK\zerothI{\Sigma}\BLACK}\nu, \qquad {\PINK\zerothI{\Sigma}\BLACK}_{n,l}^{(j)} \nu \weakstar {\PINK\zerothI{\Sigma}\BLACK}\nu, \\
	& & b_1({\PINK\firstI{\Sigma}\BLACK}_{n,k}^{(i)}) \weakstar b_1({\PINK\firstI{\Sigma}\BLACK}), \quad b_1({\PINK\firstI{\Sigma}\BLACK}_{n,l}^{(j)}) \weakstar b_1({\PINK\firstI{\Sigma}\BLACK}) \quad \text{weakly* in $L^{\infty}(\calU_k^{(i)} \cap \calU_l^{(j)} \cap K\cap \alpha)$}, 
\end{eqnarray*} 
for every compact $K \subset \calU_k^{(i)} \cap \calU_l^{(j)} $ as a consequence of \Cref{remadd1} and \Cref{remadd2}. 
By summing over $k$ and $l$ we infer \eqref{inequality1} on $\alpha$.

The final claim goes by using the fact that both measures in \eqref{inequality1} are zero on $\mathcal{S}$ as a consequence of \eqref{disss1} and \eqref{dualityadded2}. 
\end{proof}

\subsubsection{Disintegration of admissible configurations and the principle of maximum plastic work} \label{subs:dis}
The main result of this section is \Cref{two-scale Hill's principle - regime zero} and \Cref{two-scale dissipation and plastic work inequality}. 
\Cref{two-scale dissipation and plastic work inequality} is the fundamental part of \Cref{main result 2} when proving global stability of two-scale limit. 
In order to prove \Cref{two-scale Hill's principle - regime zero} one needs the notion of two-scale duality between limit stresses given in \Cref{definition K^hom_0} and limit two-scale plastic strains given in \Cref{definition A^hom_0}. 
Equipped with the appropriate duality on the cell $I \times \calY$, one is then tempted to do this through disintegration given in \Cref{disintegration result - regime zero}. 
However, the problem is that $\Sigma$ {\RED does not\BLACK} have enough regularity in $x'$ variable (i.e. it is not continuous in $x'$) while $\eta$ from \Cref{disintegration result - regime zero} can concentrate somewhere on $\omega$. 
Thus we need to define this duality through proper mollification of $\Sigma$ and then by letting the limit. This is stated in \Cref{two-scale stress-strain duality - regime zero}. The strategy of this section corresponds to the strategy of \cite{Buzancic.Davoli.Velcic.2024.1st,Buzancic.Davoli.Velcic.2024.2nd} which relies on the approach of \cite{Francfort.Giacomini.2014} with proper modifications. 
 
Let us recall the disintegration result given in \cite[Lemma 5.18.]{Buzancic.Davoli.Velcic.2024.2nd}.

\begin{lemma} \label{disintegration result - regime zero}
Let $(u,E,P) \in \calA^{\rm hom}_{0}(w)$ with the associated $\mu \in \calXzero{\ext{\omega}}$, $\kappa \in \calYzero{\ext{\omega}}$, and let $\bar{u} \in BD(\ext{\omega})$ and $u_3 \in BH(\ext{\omega})$ be the Kirchhoff--Love components of $u$. 
Then there exists $\eta \in \Mb^+(\ext{\omega})$ such that the following disintegrations hold true:
\begin{align}
	\label{disintegration result 1 - regime zero} {\E}u \otimes \calL^{2}_{y} & = \left( A_1(x') + x_3 A_2(x') \right) \eta \otimes \calL^{1}_{x_3} \otimes \calL^{2}_{y},\\
	\label{disintegration result 2 - regime zero} E \,\calL^{3}_{x} \otimes \calL^{2}_{y} & = C(x') E(x,y) \,\eta \otimes \calL^{1}_{x_3} \otimes \calL^{2}_{y}\\
	\label{disintegration result 3 - regime zero} P & = \eta \genprod P_{x'}.
\end{align}
Above, $A_1, A_2 : \ext{\omega} \to \M^{2 \times 2}_{\sym}$ and $C : \ext{\omega} \to [0, +\infty]$ are respective Radon--Nikodym derivatives of ${\E}_{x'}\bar{u}$, $-{\D}^2_{x'}u_3$ and $\calL^{2}_{x'}$ with respect to $\eta$, $E(x,y)$ is a Borel representative of $E$, and $P_{x'} \in \Mb(I \times \calY;\M^{2 \times 2}_{\sym})$ for $\eta$-a.e. $x' \in \ext{\omega}$.
Furthermore, we can choose Borel maps $(x',y) \in \ext{\omega} \times \calY \mapsto \mu_{x'}(y) \in \R^2$ and $(x',y) \in \ext{\omega} \times \calY \mapsto \kappa_{x'}(y) \in \R$ such that, for $\eta$-a.e. $x' \in \ext{\omega}$,
\begin{equation} \label{disintegration result 4a - regime zero} 
	\mu = \mu_{x'}(y) \,\eta \otimes \calL^{2}_{y}, \quad {\E}_{y}\mu = \eta \genprod E_{y}\mu_{x'},
\end{equation}
\begin{equation} \label{disintegration result 4b - regime zero} 
	\kappa = \kappa_{x'}(y) \,\eta \otimes \calL^{2}_{y}, \quad {\D}^2_{y}\kappa = \eta \genprod {\D}^2_{y}\kappa_{x'},
\end{equation}
where $\mu_{x'} \in BD(\calY)$, $\int_{\calY} \mu_{x'}(y) \,dy = 0$ and $\kappa_{x'} \in BH(\calY)$, $\int_{\calY} \kappa_{x'}(y) \,dy = 0$.
\end{lemma}

\begin{remark} \label{P_x' form on the interface}
The measure $\eta$ can be modified in such a way that for a measure $P$ which satisfies \eqref{P form on the interface} we additionally have the disintegration
\begin{align}
	\label{disintegration result 5 - regime zero} \zeta & = \eta \genprod \zeta_{x'},
\end{align}
as well as the following equality
\begin{equation}
	P_{x'}\mres{I \times (\Gamma_{ij} \setminus S)} = \big[ \bar{c}(x',\cdot) \odot \nu + x_3 \hat{c}(x',\cdot) \, \nu \otimes \nu \big]\,\zeta_{x'} \otimes \calL_{x_3}^1.
\end{equation}
Here $\zeta_{x'} \in \calM_b^{+} (\Gamma_{ij} \setminus \calS) $. This can be easily seen by firstly disintegrating $\zeta$ and then adding to $\eta$ from \Cref{disintegration result - regime zero} the first component in the generalized product of disintegrated $\zeta$ and then doing the disintegration of both the measures in \Cref{disintegration result - regime zero} and $\zeta$ with respect to this new measure. 
\end{remark}

\begin{remark} \label{admissible configurations and disintegration - regime zero}
From the disintegration in \Cref{disintegration result - regime zero}, we have that, for $\eta$-a.e. $x' \in \ext{\omega}$, 
\begin{equation*}
	{\E}_{y}\mu_{x'} - x_3 {\D}^2_{y}\kappa_{x'} = \left[ C(x') E(x,y) - \left( A_1(x') + x_3 A_2(x') \right) \right] \calL^{1}_{x_3} \otimes \calL^{2}_{y} + P_{x'} \quad \textit{ in } I \times \calY.
\end{equation*}
Thus, the quadruplet
\begin{equation*}
	\left( \mu_{x'}, \kappa_{x'}, \left[ C(x') E(x,y) - \left( A_1(x') + x_3 A_2(x') \right) \right], P_{x'} \right)
\end{equation*}
is an element of $\calA_{0}$.
\end{remark}

The following proposition is a restatement of the results obtained in \cite[Proposition 5.30.]{Buzancic.Davoli.Velcic.2024.2nd}, which establishes the existence of a measure $\lambda$ defined through two-scale stress-strain duality. 

\begin{proposition} \label{two-scale stress-strain duality - regime zero}
Let $\Sigma \in \calK^{\rm hom}_{0}$ and $(u,E,P) \in \calA^{\rm hom}_{0}(w)$ with the associated $\mu \in \calXzero{\ext{\omega}}$, $\kappa \in \calYzero{\ext{\omega}}$. 
There exists an element $\lambda \in \Mb(\ext{\Omega} \times \calY)$ which can be constructed as the weak* limit of properly defined measures $\lambda_n$ as follows:
\begin{enumerate}[label=\arabic*.]
	\item There exists a finite open covering $\{U_{k}\}$ of $\closure{\omega}$ such that $\omega \cap U_{k}$ is (strongly) star-shaped with $C^2$ boundary, and a smooth partition of unity $\{\psi_{k}\}$ subordinate to the covering $\{U_{k}\}$.
	\item For each subdomain $\omega \cap U_{k}$ there exists a sequence $\{ \Sigma^{k}_n \} \subset C^\infty(\R^2;L^2(I \times \calY;\M^{3 \times 3}_{\sym}))$ such that $\Sigma^{k}_n(x',\cdot) \in \calK_{0}$ for every $x' \in \omega$,
	and there exists a sequence of well-defined bounded measures 
	\begin{equation*}
		\lambda^{k}_n := \eta \genprod [ \Sigma^{k}_n(x',\cdot) : P_{x'} ] \in \Mb((\widetilde{\omega} \cap U_{k}) \times I \times \calY),
	\end{equation*}
	such that
	\begin{equation*}
		\lambda^{k}_n \weakstar \lambda^{k} \quad \text{weakly* in $\Mb((\widetilde{\omega} \cap U_{k}) \times I \times \calY)$}.
	\end{equation*}
	\item The global measure $\lambda$ is defined on $\ext{\Omega} \times \calY$ by 
	letting, for every $\phi \in C_0(\ext{\Omega} \times \calY)$,
	\begin{equation*}
		\langle \lambda_n, \phi \rangle := \sum_{k} \langle \lambda^{k}_n, \psi_{k}\,\phi \rangle,
	\end{equation*}
	and
	\begin{equation*}
		\lambda_n \weakstar \lambda \quad \text{weakly* in $\Mb(\ext{\Omega} \times \calY)$}.
	\end{equation*}
\end{enumerate}
Furthermore, the mass of $\lambda$ is given by
\begin{equation} \label{mass of lambda - regime zero}
	\lambda(\ext{\Omega} \times \calY) = -\int_{\Omega \times \calY} \Sigma : E \,{\PINK dx\,dy\BLACK} + \int_{\omega} {\PINK\zerothI{\sigma}\BLACK} : {\E}\bar{w} \,dx' - \frac{1}{12} \int_{\omega} {\PINK\firstI{\sigma}\BLACK} : {\D}^2w_3 \,dx'.
\end{equation}
\end{proposition}

The following theorem provides a lower bound for dissipation functional which generalizes \cite[Theorem 5.31.]{Buzancic.Davoli.Velcic.2024.2nd}.

\begin{theorem} \label{two-scale Hill's principle - regime zero}
Let $\Sigma \in \calK^{\rm hom}_{0}$ and $(u,E,P) \in \calA^{\rm hom}_{0}(w)$ with the associated $\mu \in \calXzero{\ext{\omega}}$, $\kappa \in \calYzero{\ext{\omega}}$. 
Then
\begin{equation*}
	\zerothI{\disspot^{\rm red}\!\left(P\right)} \geq {\PINK\zerothI{\lambda}\BLACK}, 
\end{equation*}
where $\lambda \in \Mb(\ext{\Omega} \times \calY)$ is given by \Cref{two-scale stress-strain duality - regime zero}.
\end{theorem}

\begin{proof} 
Let $\{\psi_{k}\}$, $\{\Sigma^{k}_n\}$, $\{\lambda^{k}_n\}$ and $\lambda_n$ be defined as in \Cref{two-scale stress-strain duality - regime zero}. 
By \Cref{cell Hill's principle - regime zero}, we have for $\eta$-a.e. $x' \in \ext{\omega}$
\begin{equation} \label{pom step} 
	\int_{I \times \calY} \varphi(x',y)\,d\disspot^{\rm red}\!(P_{x'}) \geq \int_{I \times \calY} \varphi(x',y) \,d[ \Sigma^{k}_n(x',\cdot) : P_{x'} ], \ \text{ for every } \varphi \in C_{c}(\ext{\omega} \times \calY), \varphi \geq 0.
\end{equation}
Since $\frac{dP}{d|P|}(x,y) = \frac{dP_{x'}}{d|P_{x'}|}(x_3,y)$ for $|P_{x'}|$-a.e. $(x_3,y) \in I \times \calY$ by \cite[Proposition 2.2]{Buzancic.Davoli.Velcic.2024.1st}, we can conclude that, for every $i$, on $\ext{\Omega} \times \calY_i$ (recall \Cref{dissipation on acell})
\begin{align*}
	\disspot^{\rm red}(P) 
	& = \disspot^{\rm red}_i\!\left(\frac{dP}{d|P|}\right)\,|P| 
	= \eta \genprod \disspot^{\rm red}_i\!\left(\frac{dP}{d|P|}(x',\cdot)\right)\,|P_{x'}| 
	\\& = \eta \genprod \disspot^{\rm red}_i\!\left(\frac{dP_{x'}}{d|P_{x'}|}\right)\,|P_{x'}|
	= \sum_{k} \psi_{k} \eta \genprod \disspot^{\rm red}_i\!\left(\frac{dP_{x'}}{d|P_{x'}|}\right)\,|P_{x'}|
	\\& = \sum_{k} \psi_{k} \eta \genprod \disspot^{\rm red}(P_{x'}).
\end{align*}
Similarly, recalling \Cref{P_x' form on the interface}, on $\ext{\Omega} \times (\Gamma_{ij} \setminus S)$ we have
\begin{align*}
	\disspot^{\rm red}(P)
	& = \disspot_{ij}^{\rm red}\left(\nu,\bar{c},\hat{c}\right) \zeta \otimes \calL_{x_3}^1 
	= \eta \genprod \disspot_{ij}^{\rm red}\left(\nu,\bar{c}(x',\cdot),\hat{c}(x',\cdot)\right) \zeta_{x'} \otimes \calL_{x_3}^1
	\\& = \sum_{k} \psi_{k} \eta \genprod \disspot_{ij}^{\rm red}\left(\nu,\bar{c}(x',\cdot),\hat{c}(x',\cdot)\right) \zeta_{x'} \otimes \calL_{x_3}^1
	\\& = \sum_{k} \psi_{k} \eta \genprod \disspot^{\rm red}(P_{x'}).
\end{align*}
Consequently using \eqref{pom step}, 
\begin{align*}
	\int_{\ext{\Omega} \times \calY} \varphi(x',y)\,d\disspot^{\rm red}\!(P) 
	& = \sum_{k} \int_{\ext{\omega}} \psi_{k}(x') \!\left( \int_{I \times \calY} \varphi(x',y)\,d\disspot^{\rm red}\!(P_{x'}) \right)\! \,d\eta(x')
	\\& \geq \sum_{k} \int_{\ext{\omega}} \psi_{k}(x') \!\left( \int_{I \times \calY} \varphi(x',y) \,d[ \Sigma^{k}_n(x',\cdot) : P_{x'} ] \right)\! \,d\eta(x')
	\\& = \sum_{k} \int_{\ext{\Omega} \times \calY} \psi_{k}(x') \varphi(x',y) \,d\lambda^{k}_n(x,y) 
	= \int_{\ext{\Omega} \times \calY} \varphi(x',y) \,d\lambda_n.
\end{align*}
By passing to the limit, we infer the desired inequality.
\end{proof}

As a direct consequence of the previous theorem and \eqref{mass of lambda - regime zero}, we are now in a position to state a principle of maximum plastic work in our setting.

\begin{corollary} \label{two-scale dissipation and plastic work inequality}
For every $\Sigma \in \calK^{hom}_{0}$ and $(u,E,P) \in \calA^{hom}_{0}(w)$ we have
\begin{equation*}
	\calR^{hom}_{0}(P) \geq -\int_{\Omega \times \calY} \Sigma : E \,{\PINK dx\,dy\BLACK} + \int_{\omega} {\PINK\zerothI{\sigma}\BLACK} : {\E}_{{\RED x' \BLACK}}\bar{w} \,dx' - \frac{1}{12} \int_{\omega} {\PINK\firstI{\sigma}\BLACK} : {\D}^2_{{\RED x' \BLACK}}w_3 \,dx'.
\end{equation*}
\end{corollary}

\subsection{Two-scale quasistatic evolution} \label{section two-scale quasistatic evolution} 

In this section we define the two-scale quasistatic evolution and prove the convergence result in \Cref{main result 2} as well as additional regularity result of the limit two-scale evolution given in \Cref{cor improved 2}. 
The associated $\calR^{\rm hom}_{0}$-variation of a function $P : [0,T] \to \Mb(\ext{\Omega} \times \calY;{\RED\M^{2 \times 2}_{\sym}\BLACK})$ on $[a,b]$ is defined as
\begin{equation*}
	\calV_{\calR^{\rm hom}_{0}}(P; a, b) := \sup\left\{ \sum_{i = 1}^{n-1} \calR^{\rm hom}_{0}\!\left(P(t_{i+1}) - P(t_i)\right) : a = t_1 < t_2 < \ldots < t_n = b,\ n \in \N \right\}.
\end{equation*}
As already said we prescribe for every $t \in [0, T]$ a boundary datum $w(t) \in H^1_{KL}(\ext{\Omega})$ and we assume the map $t\mapsto w(t)$ to be absolutely continuous from $[0, T]$ into $H^1(\ext{\Omega};\R^3)$.

We now give the notion of the limiting quasistatic elasto-plastic evolution.

\begin{definition} \label{two-scale quasistatic evolution}
A \emph{two-scale quasistatic evolution} for the boundary datum $w(t)$ is a function $t \mapsto (u(t), E(t), P(t))$ from $[0,T]$ into $BD_{KL}(\ext{\Omega}) \times L^2(\ext{\Omega} \times \calY;\M^{2 \times 2}_{{\RED\sym\BLACK}}) \times \Mb(\ext{\Omega} \times \calY;\M^{2 \times 2}_{\sym})$ which satisfies the following conditions:
\begin{enumerate}[label=(qs\arabic*)$^{\rm hom}_{0}$]
	\item \label{hom-qs S} for every $t \in [0,T]$ we have $(u(t), E(t), P(t)) \in \calA^{\rm hom}_{0}(w(t))$ and
	\begin{equation*}
		\calQ^{\rm hom}_{0}(E(t)) \leq \calQ^{\rm hom}_{0}(H) + \calR^{\rm hom}_{0}(\Pi-P(t)),
	\end{equation*}
	for every $(\upsilon,H,\Pi) \in \calA^{\rm hom}_{0}(w(t))$.
	\item \label{hom-qs E} the function $t \mapsto P(t)$ from $[0, T]$ into $\Mb(\ext{\Omega} \times \calY;\M^{2 \times 2}_{\dev})$ has bounded variation and for every $t \in [0, T]$
	\begin{equation*}
		\calQ^{\rm hom}_{0}(E(t)) + \calV_{\calR^{\rm hom}_{0}}(P; 0, t) = \calQ^{\rm hom}_{0}(E(0)) 
		+ \int_0^t \int_{\Omega \times \calY} \red{\C}(y) E(s) : {\E}\dot{w}(s) \,{\PINK dx\,dy\,ds\BLACK}.
	\end{equation*} 
\end{enumerate}
\end{definition}

Recalling the definition of a $h$-quasistatic evolution for the boundary datum $w(t)$ given in \Cref{h-quasistatic evolution}, we are in a position to formulate the main result of the paper.

\begin{theorem} \label{main result 2}
Let $t \mapsto w(t)$ be absolutely continuous from $[0,T]$ into $H^1_{KL}(\ext{\Omega})$. 
Let $\calY$ be a geometrically admissible multi-phase torus. 
Assume also \eqref{tensorassumption1}--\eqref{tensorassumption3} and \eqref{coercivity of R_i} and 
that there exists a sequence of triples $(u^\eps_0, e^\eps_0, p^\eps_0, \alpha^\eps_0) \in \calA_{0}^{\rm hard,\eps}(w(0))$ that satisfy \ref{eps-0-qs S} such that
\begin{align}
	& u^\eps_0 \weakstar u_0 \quad \text{weakly* in $BD(\ext{\Omega})$}, \label{main result u^eps_0 condition}\\
	& e^\eps_0 \strongtwoscale E_0 \quad \text{two-scale strongly in $L^2(\ext{\Omega} \times \calY;\M^{2 \times 2}_{\sym})$}, \label{main result e^eps_0 condition}\\
	& p^\eps_0 \weakstartwoscale P_0 \quad \text{two-scale weakly* in $\Mb(\ext{\Omega} \times \calY;\M^{3 \times 3}_{\dev})$}, \label{main result p^eps_0 condition}\\ \label{initial con}
	& \delta(\eps) \|p_0^\eps\|_{L^2(\Omega;\M^{3\times 3}_{\dev})} \to 0, \quad \delta(\eps) \|\alpha_0^\eps\|_{L^2(\Omega)} \to 0, 
\end{align}
for some $(u_0,E_0,P_0^{1,2}) \in \calA^{\rm hom}_{0}(w(0))$. 
For every $\eps > 0$, let 
\begin{equation*}
	t \mapsto (u^\eps(t), e^\eps(t), p^\eps(t), \alpha^{\eps}(t))
\end{equation*}
be a $\eps$-quasistatic evolution in the sense of \Cref{eps-limiting quasistatic evolution} for the boundary datum $w$ such that $u^\eps(0) = u^\eps_0$, $e^\eps(0) = e^\eps_0$, and $p^\eps(0) = p^\eps_0$, $\alpha^\eps(0) = \alpha^\eps_0$. 
We extend it on $\ext{\Omega}$ by setting $u^{\eps}=w$, $e^{\eps}={\E}w$, $p^{\eps}=\alpha^{\eps}=0$ on $\ext{\Omega} \setminus \Omega \times [0,T[$. 
Then, there exists a two-scale quasistatic evolution
\begin{equation*}
	t \mapsto (u(t), E(t), P(t)) 
\end{equation*}
for the boundary datum $w(t)$ such that $u(0) = u_0$,\, $E(0) = E_0$, and $P(0) = P_0^{1,2}$, and such that (up to subsequence)
for every $t \in [0,T]$ we have
\begin{align}
	& u^\eps(t) \weakstar u(t) \quad \text{weakly* in $BD(\ext{\Omega})$}, \label{main result 0 u^eps(t) - regime zero}\\
	& e^\eps(t) \weaktwoscale E(t) \quad \text{two-scale weakly in $L^2(\ext{\Omega} \times \calY;\M^{2 \times 2}_{\sym})$}, \label{main result e^eps(t) - regime zero}\\
	& (p^\eps)^{1,2}(t) \weakstartwoscale P(t) \quad \text{two-scale weakly* in $\Mb(\ext{\Omega} \times \calY;\M^{2 \times 2}_{\sym})$}. \label{main result p^eps(t) - regime zero}
\end{align}
\end{theorem}

\begin{proof}
The proof follows the lines of \Cref{main result 1}. 

\noindent{\bf Step 1: \em Compactness.}

Recalling \Cref{lema apriori estimates}, we infer the estimates
\begin{equation} \label{boundness in time 3}
	\calV(p^\eps; 0, T) \leq C,
\end{equation}
for every $\eps>0$. 
Hence, by a generalized version of Helly's selection theorem (see \cite[Lemma 7.2]{DalMaso.DeSimone.Mora.2006}), there exists a (not relabeled) subsequence, independent of $t$, and $P \in BV(0,T;\Mb(\ext{\Omega} \times \calY;\M^{3 \times 3}_{\dev}))$ such that
\begin{equation*}
	p^\eps(t) \weakstartwoscale P(t) \quad \text{two-scale weakly* in $\Mb(\ext{\Omega} \times \calY;\M^{3 \times 3}_{\dev})$},
\end{equation*}
for every $t \in [0,T]$, and $\calV(P; 0, T) \leq C$.
We extract a further subsequence (possibly depending on $t$),
\begin{align*}
	& u^{\eps_t}(t) \weakstar u(t) \quad \text{weakly* in $BD(\ext{\Omega})$},\\
	& e^{\eps_t}(t) \weaktwoscale E(t) \quad \text{two-scale weakly in $L^2(\ext{\Omega} \times \calY;\M^{3 \times 3}_{\sym})$},
\end{align*}
for every $t \in [0,T]$. Here $u(t) \in BD_{KL}(\ext{\Omega})$. 
Furthermore, according to \Cref{two-scale limit of symmetrized gradients} and \Cref{two-scale limit of Hessians}, one can choose the above subsequence in a way such that there exist $\mu(t) \in \calXzero{\ext{\omega}}$, $\kappa(t) \in \calYzero{\ext{\omega}}$ such that
\begin{equation*}
	{\E}u^{\eps_t}(t) \weakstartwoscale {\E}u(t) \otimes \calL^{2}_{y} + \begin{pmatrix} \begin{matrix} {\E}_{y}\mu(t) - x_3 {\D}^2_{y}\kappa(t) \end{matrix} & 0 \\ 0 & 0 \end{pmatrix}.
\end{equation*}
Since ${\E}_{x'}\bar{u}^{\eps_t}(t) - x_3 {\D}^2_{x'}u_3^{\eps_t}(t) = e^{\eps_t}(t) + (p^{\eps_t}(t))^{1,2}$ in $\ext{\Omega}$ for every $t \in [0,T]$, we deduce that $(u(t),E(t),P^{1,2}(t)) \in \calA^{\rm hom}_{0}(w(t))$.

Lastly, we consider for every $t \in [0,T]$
\begin{equation*}
	\sigma^{\eps_t}(t) := \red{\C}\left({\RED\frac{x'}{\eps_t}\BLACK}\right) e^{\eps_t}(t),
	\quad \chikin^{\eps_t}(t) := - \delta(\eps_t)\,\Hkin\!\left(\frac{x'}{\eps_t}\right) p^{\eps_t}(t)
	\;\text{ and }\; \chiiso^{\eps_t}(t) := - \delta(\eps_t)\,\Hiso\!\left(\frac{x'}{\eps_t}\right) \alpha^{\eps_t}(t).
\end{equation*}
Then we can choose (not relabeled) subsequences such that $\chikin^{\eps_t}(t)$ and {\RED$\chiiso^{\eps_t}(t)$\BLACK} tend to zero, and
\begin{equation} \label{main result sigma^eps(t)}
	\sigma^{\eps_t}(t) \weaktwoscale \Sigma(t) \quad \text{two-scale weakly in $L^2(\ext{\Omega} \times \calY;\M^{{\RED 3 \times 3 \BLACK}}_{\sym})$},
\end{equation}
where $\Sigma(t) := \red{\C}(y) E(t)$. 
Since $(\sigma^{\eps_t}(t),\, \chikin^{\eps_t}(t),\, \chiiso^{\eps_t}(t)) \in \calK^{\eps_t}_{0}$ for every $t \in [0,T]$, by \Cref{two-scale weak limit of admissible stress - regime zero} we can conclude $\Sigma(t) \in \calK^{\rm hom}_{0}$. 

\noindent{\bf Step 2: \em Global stability.}

Since from Step 1 we have $(u(t),E(t),P^{1,2}(t)) \in \calA^{\rm hom}_{0}(w(t)) $ with the associated $\mu(t) \in \calXzero{\ext{\omega}}$, $\kappa(t) \in \calYzero{\ext{\omega}}$, then for every $(\upsilon,H,\Pi) \in \calA^{\rm hom}_{0}(w(t)) $ with the associated $\nu(t) \in \calXzero{\ext{\omega}}$, $\lambda(t) \in \calYzero{\ext{\omega}}$ we have
\begin{equation*}
	(\upsilon-u(t),H-E(t),\Pi-P^{1,2}(t)) \in \calA^{\rm hom}_{0}(0).
\end{equation*}
Furthermore, since from the first step of the proof $\red{\C}(y) E(t) \in \calK^{\rm hom}_{0}$, by \Cref{two-scale dissipation and plastic work inequality} we have 
\begin{align*}
	\calR^{\rm hom}_{0}(\Pi-P^{1,2}(t)) & \geq -\int_{\Omega \times \calY} \red{\C}(y) E(t) : (H-E(t)) \,{\PINK dx\,dy\BLACK}\\
	& = \calQ^{\rm hom}_{0}(E(t)) + \calQ^{\rm hom}_{0}(H-E(t)) - \calQ^{\rm hom}_{0}(H),
\end{align*}
where the last equality is a straightforward computation. 
From the above, we immediately deduce
\begin{equation*}
	\calR^{\rm hom}_{0}(\Pi-P^{1,2}(t)) + \calQ^{\rm hom}_{0}(H) \geq \calQ^{\rm hom}(E(t)) + \calQ^{\rm hom}_{0}(H-E(t)) \geq \calQ^{\rm hom}_{0}(E(t)),
\end{equation*}
hence the global stability of the two-scale quasistatic evolution \ref{hom-qs S}.

We proceed by proving that the limit functions $u(t)$ and $E(t)$ do not depend on the subsequence. 
Assume $(\upsilon(t),H(t),P(t)) \in \calA^{\rm hom}_{0}(w(t)) $ with the associated $\nu(t) \in \calXzero{\ext{\omega}}$, $\lambda(t) \in \calYzero{\ext{\omega}}$ also satisfy the global stability of the two-scale quasistatic evolution. 
By the strict convexity of $\calQ^{\rm hom}_{0}$, we immediately obtain that
\begin{equation*}
	H(t) = E(t).
\end{equation*}
Identifying ${\E}u(t), {\E}\upsilon(t)$ with elements of $\Mb(\ext{\Omega};\M^{2 \times 2}_{\sym})$ and using \eqref{admissible two-scale configurations - regime zero}, we have that
\begin{align*}
	{\E}\upsilon(t) \otimes \calL^{2}_{y} + {\E}_{y}\nu(t) - x_3 {\D}^2_{y}\kappa(t) & = H(t) \,\calL^{3}_{x} \otimes \calL^{2}_{y} + P^{1,2}(t)\\
	& = E(t) \,\calL^{3}_{x} \otimes \calL^{2}_{y} + P^{1,2}(t)\\
	& = {\E}u(t) \otimes \calL^{2}_{y} + {\E}_{y}\mu(t) - x_3 {\D}^2_{y}\kappa(t).
\end{align*}
Integrating over $\calY$, we obtain
\begin{equation*}
	{\E}\upsilon(t) = {\E}u(t).
\end{equation*}
Using the variant of Poincar\'{e}--Korn's inequality, see \cite[Chapter II, Proposition 2.4]{Temam.1985}), we can infer that $\upsilon(t) = u(t)$ on $\ext{\Omega}$.

This implies that the whole sequences converge without depending on $t$, i.e.
\begin{align*}
	& u^h(t) \weakstar u(t) \quad \text{weakly* in $BD(\ext{\Omega})$},\\
	& e^\eps(t) \weaktwoscale E(t) \quad \text{two-scale weakly in $L^2(\ext{\Omega} \times \calY;\M^{3 \times 3}_{\sym})$}.
\end{align*}

\noindent{\bf Step 3: \em Energy balance.}

In order to prove the energy balance of the two-scale quasistatic evolution \ref{hom-qs E}, it is enough (by arguing as in, e.g. \cite[Theorem 4.7]{DalMaso.DeSimone.Mora.2006} and \cite[Theorem 2.7]{Francfort.Giacomini.2012}) to prove the energy inequality
\begin{align} \label{main result hom - step 3 inequality}
\begin{split}
	& \calQ^{\rm hom}_{0}(E(t)) + \calV_{\calR^{\rm hom}_{0}}(P^{1,2}; 0, t) \\
	& \leq \calQ^{\rm hom}_{0}(E(0)) 
	+ \int_0^t \int_{\Omega \times \calY} \red{\C}(y) E(s) : {\E}\dot{w}(s) \,{\PINK dx\,dy\,ds\BLACK}.
\end{split}
\end{align}

For a fixed $t \in [0,T]$, let us consider a subdivision $0 = t_1 < t_2 < \ldots < t_n = t$ of $[0,t]$. 
Due to the lower semicontinuity from \Cref{lower semicontinuity of energies - eps to 0}, we have 
\begin{align*}
	& \calQ^{\rm hom}_{0}(E(t)) + \sum_{i = 1}^{n} \calR^{\rm hom}_{0}\left(P(t_{i+1}) - P(t_i)\right)\\
	& \leq \liminf\limits_{\eps}\left( \calQ^{\eps}_{0}(e^\eps(t)) + \delta(\eps)\,\calQ^{\rm hard,\eps}_{0}(p^\eps(t), \alpha^\eps(t)) + \sum_{i = 1}^{n} \calR^{\eps}_{0}\left(p^\eps(t_{i+1}) - p^\eps(t_i)\right) \right)\\
	& \leq \liminf\limits_{\eps}\left( \calQ^{\eps}_{0}(e^\eps(t)) + \delta(\eps)\,\calQ^{\rm hard,\eps}_{0}(p^\eps(t), \alpha^\eps(t)) + \calV_{\calR^{\eps}_{0}}(p^\eps; 0, t) \right)\\
	& = \liminf\limits_{\eps}\left( \calQ^{\eps}_{0}(e^\eps(0)) + \delta(\eps)\,\calQ^{\rm hard,\eps}_{0}(p^\eps(0), \alpha^\eps(0))
	+ \int_0^t \int_{\Omega} \C\left(\tfrac{x'}{\eps}\right) e^\eps(s) : {\E}\dot{w}(s) \,dx ds \right),
\end{align*}
where the last equality follows from \ref{eps-0-qs E}. 
In view of the strong convergence assumed in \eqref{main result e^eps_0 condition}, \eqref{initial con} and \eqref{main result sigma^eps(t)}, by the Lebesgue's dominated convergence theorem and \Cref{lema apriori estimates} we infer
\begin{align*}
	& \lim\limits_{\eps}\left( \calQ^{\eps}_{0}(e^\eps(0)) + \delta(\eps)\,\calQ^{\rm hard,\eps}_{0}(p^\eps(0), \alpha^\eps(0)) \
	+ \int_0^t \int_{\Omega} \C\left(\tfrac{x'}{\eps}\right) e^\eps(s) : {\E}\dot{w}(s) \,dx ds \right)\\
	& = \calQ^{\rm hom}_{0}(E(0)) 
	+ \int_0^t \int_{\Omega \times \calY} \red{\C}(y) E(s) : {\E}\dot{w}(s) \,{\PINK dx\,dy\,ds\BLACK}.
\end{align*}
Hence, we have
\begin{align*}
	& \calQ^{\rm hom}_{0}(E(t)) + \sum_{i = 1}^{n} \calR^{\rm hom}_{0}\left(P^{1,2}(t_{i+1}) - P^{1,2}(t_i)\right) \\
	& \leq \calQ^{\rm hom}_{0}(E(0)) 
	+ \int_0^t \int_{\Omega \times \calY} \red{\C}(y) E(s) : {\E}\dot{w}(s) \,{\PINK dx\,dy\,ds\BLACK}.
\end{align*}
Taking the supremum over all partitions of $[0,t]$ yields \eqref{main result hom - step 3 inequality}, which concludes the proof, after replacing $P$ with $P^{1,2}$. 
\end{proof}

The claim of \Cref{main result 2} can be improved in the following way.

\begin{corollary} \label{cor improved 2} 
Any solution of two-scale quasistatic evolution $t \mapsto (u(t), E(t), P(t))$ given by \Cref{two-scale quasistatic evolution} is absolutely continuous from $[0,T]$ into $L^1(\ext{\Omega}) \times L^2(\ext{\Omega} \times \calY;\M^{3 \times 3}_{\sym}) \times \Mb(\ext{\Omega} \times \calY;\M^{3 \times 3}_{\dev})$. Let the conditions of \Cref{main result 2} be valid. In addition to convergences \eqref{main result 0 u^eps(t) - regime zero}, \eqref{main result e^eps(t) - regime zero} and \eqref{main result p^eps(t) - regime zero} we have that for every $t \in [0,T]$
\begin{align} 
	& \label{con1} e^\eps(t) \strongtwoscale E(t) \quad \text{two-scale strongly in $L^2(\ext{\Omega} \times \calY;\M^{2 \times 2}_{\sym})$} \\ 
	& \label{con2} \delta(\eps) \|p_0^\eps(t)\|_{L^2(\Omega;\M^{3\times 3}_{\dev})} \to 0, \quad \delta(\eps) \|\alpha_0^\eps(t)\|_{L^2(\Omega)} \to 0. 
\end{align} 
\end{corollary}

\begin{proof} 
The first part of the claim can be proved in the same way as \cite[Theorem 5.2]{DalMaso.DeSimone.Mora.2006}, which is much simpler in the case of zero loads (cf. \Cref{improved convergence 1}), since the proof relies on global stability, energy equality and coercivity property of the dissipation functional (see also \eqref{coercivity interface}).
To prove the second part, note that as the consequence of the proof of energy balance in \Cref{main result 2}, Step 3 and energy equality of two-scale quasistatic evolution we have that for every $t \in [0,T]$
\[
	\calQ^{\eps}_{0}(e^\eps(t)) \to \calQ^{\rm hom}_{0}(E(t)), \quad \delta(\eps)\,\calQ^{\rm hard,\eps}_{0}(p^\eps(t), \alpha^\eps(t)) \to 0,
\]
Convergences \eqref{con1} and \eqref{con2} immediately follow using strict convexity of $\C^{red}$ and coercivity of $\calQ^{\rm hard,\eps}_{0}$.
\end{proof}

\begin{remark}
It is not clear, from the arguments above, that for given $(u(0), E(0), P(0)) \in \calA^{\rm hom}_{0}(w(0))$ that satisfy \ref{hom-qs S} there exist $(u^\eps_0, e^\eps_0, p^\eps_0, \alpha^\eps_0) \in \calA_{0}^{\rm hard,\eps}(w(0))$ that satisfy \ref{eps-0-qs S} and additionally the convergences \eqref{main result u^eps_0 condition}--\eqref{initial con} cf. \Cref{remark gamma convergence}. Note that in \Cref{remark gamma convergence} we were able to find some subset of the stable initial states that converges to some member of the stable limit initial conditions.

Using a different argument from that in \Cref{remark gamma convergence}, we can conclude that for any $w(0) \in BD_{KL}(\widetilde{\Omega)}$ we can find globally stable $(u(0), E(0), P(0)) \in \calA^{\rm hom}_{0}(w(0))$ and $(u^\eps_0, e^\eps_0, p^\eps_0, \alpha^\eps_0) \in \calA_{0}^{\rm hard,\eps}(w(0))$ such that convergences \eqref{main result u^eps_0 condition}--\eqref{initial con} are satisfied.
Namely, for a given $w(0) \in BD_{KL}(\ext{\Omega})$ we can consider an $\eps$-limiting quasistatic evolution $(u^\eps(t), e^\eps(t), p^\eps(t), \alpha^\eps(t))$ with zero initial condition corresponding to the boundary condition $\tilde{w}(t) := t \cdot w(0)$, $t \in [0,1]$.
Then, as a consequence of \Cref{main result 2} and \Cref{cor improved 2}, we have that $(u^\eps(1), e^\eps(1), p^\eps(1), \alpha^\eps(1)) \in \calA^{\rm hard,\eps}_0 (w(0))$ and $(u(1),E(1),P(1)) \in \calA^{\rm hom}_{0}(w(0))$ satisfy both the global stability condition and the convergences \eqref{main result u^eps_0 condition}--\eqref{initial con}. 
\end{remark}

\section*{Acknowledgement}

This work was supported by the Croatian Science Foundation under the project number HRZZ-IP-2022-10-5181 and by the NextGenerationEU framework through the project "DEEPWAVE" at the University of Zagreb Faculty of Electrical Engineering and Computing. 
The views and opinions expressed are solely those of the author(s) and do not necessarily reflect those of the European Union or the European Commission. Neither the European Union nor the European Commission can be held responsible for them.

The authors wish to thank Gilles A. Francfort for discussions related to this work.

\vspace{+1ex} 
\noindent \textbf{Data Availability:} Due to the nature of the research, there is no supporting data.

\vspace{+1ex} 
\noindent \textbf{Conflict of Interest:} The authors declare to have no conflict of interest.

\printbibliography

\end{document}